\documentclass[11pt]{article}

\usepackage{authblk}
\usepackage{selectp}
\usepackage{natbib}
\usepackage{setspace}
\usepackage{subfloat, color, xcolor,colortbl, tcolorbox,multirow}
\usepackage{pifont}
\newcommand{\cmark}{\ding{51}}
\newcommand{\xmark}{\ding{54}}

\PassOptionsToPackage{numbers, compress}{natbib}

\usepackage[lmargin=1in]{geometry}

\newcommand{\cg}{\cellcolor[gray]{0.7}}

\makeatother
\makeatletter
\newcommand{\pushright}[1]{\ifmeasuring@#1\else\omit\hfill$\displaystyle#1$\fi\ignorespaces}
\makeatother

\usepackage{adjustbox}

\usepackage{rotating} 

\usepackage{mathtools}
\usepackage{algpseudocode}
\usepackage{algorithm}
\usepackage[utf8]{inputenc} 
\usepackage[T1]{fontenc}    
\usepackage[hidelinks]{hyperref}       
\usepackage{url}            
\usepackage{booktabs}       
\usepackage{amsfonts}       
\usepackage{nicefrac}       
\usepackage{microtype}      
\usepackage{xcolor}         
\usepackage{enumerate}

\usepackage{graphicx}
\usepackage{amsmath, amsthm, amssymb}
\usepackage{bm}
\usepackage{mathabx}

\newtheorem{defi}{Definition}
\newtheorem{thm}{Theorem}
\newtheorem{prop}{Proposition}
\newtheorem{lem}{Lemma}
\newtheorem{obs}{Observation}
\newtheorem{coro}{Corollary}

\newcommand{\bsigma}{\bm{\sigma}}
\newcommand{\btheta}{\bm{\theta}}
\newcommand{\mx}{\mathrm{x}}
\newcommand{\my}{\mathrm{y}}
\newcommand{\mz}{\mathrm{z}}
\newcommand{\bbeta}{\bm{\beta}}

\newcommand{\bxi}{\bm{\xi}}
\newcommand{\bx}{\bm{x}}

\newcommand{\bq}{\bm{q}}
\newcommand{\bp}{\bm{p}}
\newcommand{\bz}{\bm{z}}
\newcommand{\bw}{\bm{w}}
\newcommand{\bv}{\bm{v}}
\newcommand{\bt}{\bm{t}}

\newcommand{\bs}{\bm{s}}

\newcommand{\mX}{\mathbb{X}}
\newcommand{\mY}{\mathbb{Y}}
\newcommand{\mZ}{\mathbb{Z}}
\newcommand{\mR}{\mathbb{R}}
\newcommand{\mB}{\mathbb{B}}
\newcommand{\mC}{\mathbb{Z}}
\newcommand{\mP}{\mathbb{P}}
\newcommand{\mQ}{\mathbb{Q}}
\newcommand{\mE}{\mathbb{E}}

\usepackage{dsfont}

\newcommand{\indicate}[1]{\mathds{1}\small{[#1]}}

\newcounter{defcounter}
\newenvironment{myequation}{%
\addtocounter{equation}{-1}
\refstepcounter{defcounter}

\begin{equation}}
{\end{equation}}

\title{\vspace{-2cm} It's All in the Mix: \\ Wasserstein {\color{black} Classification and Regression} with Mixed Features}

\author{Reza Belbasi}
\author{Aras Selvi}
\author{Wolfram Wiesemann}

\affil{\small \textit{Imperial College Business School, London, United Kingdom} \\ \texttt{\{r.belbasi21, a.selvi19, ww\}@imperial.ac.uk}}

\date{}

\begin{document}

\maketitle

\begin{abstract}
    {\color{black}
    \noindent \textbf{Problem definition:} A key challenge in supervised learning is data scarcity, which can cause prediction models to overfit to the training data and perform poorly out of sample. A contemporary approach to combat overfitting is offered by distributionally robust problem formulations that consider all data-generating distributions close to the empirical distribution derived from historical samples, where `closeness' is determined by the Wasserstein distance. While such formulations show significant promise in prediction tasks where all input features are continuous, they scale exponentially when discrete features are present. \\
    \noindent \textbf{Methodology/results:} We demonstrate that distributionally robust mixed-feature classification and regression problems can indeed be solved in polynomial time. Our proof relies on classical ellipsoid method-based solution schemes that do not scale well in practice. To overcome this limitation, we develop a practically efficient (yet, in the worst case, exponential time) cutting plane-based algorithm that admits a polynomial time separation oracle, despite the presence of exponentially many constraints. We compare our method against alternative techniques both theoretically and empirically on standard benchmark instances. \\}
    \noindent \textbf{Managerial implications:} Data-driven operations management problems often involve prediction models with discrete features. We develop and analyze {\color{black} distributionally robust prediction models} that faithfully account for the presence of discrete features, and we demonstrate that our {\color{black} models} can significantly outperform existing methods that are agnostic to the presence of discrete features, both theoretically and on standard benchmark instances.
\end{abstract}

\section{Introduction}\label{sec-intro}

{\color{black} The recent application of machine learning tools across all areas of operations management has led to a plethora of data-driven and end-to-end approaches that blend predictive models from machine learning with optimization frameworks from operations research and operations management.} Notable examples include inventory management \citep{BR19:big_data_newsvendor,BK20:perscriptive}, logistics \citep{BDJM19:travel_time,BSW23:courier_scheduling} and supply chain management \citep{GFS19:retail_location}, assortment optimization \citep{KU20:dynamic_assortment,FZLZ22:customer_choice} and revenue management \citep{FLS16:online_retailer,ABHHMP23:pricing_hetero} as well as healthcare operations \citep{BORS16:chemo,BB20:online_decision_high,BP23:hospital}.

The machine learning algorithms used for prediction are prone to overfitting the available data. Overfitted models perform well on the training data used to calibrate the model, but their performance deteriorates when exposed to new, unseen data. This undesirable effect is amplified if the output of a machine learning model is used as input to a downstream optimization model; this phenomenon is known by different communities as the \emph{Optimizer's Curse} \citep{smith2006optimizer} or the \emph{Error-Maximization Effect of Optimization} \citep{michaud1989markowitz}. Traditionally, overfitting is addressed with regularization techniques that penalize complex models characterized by large and/or dense model parameters \citep{hastie2009elements, murphy2022}. A contemporary alternative from the robust optimization community frames machine learning problems as Stackelberg leader-follower games where the learner selects a model that performs best against a worst-case data-generating distribution selected by a conceptual adversary (`nature') from a predefined ambiguity set \citep{BTEGN09:rob_opt, RM19:dro, BdH22:rob_opt}. We talk about Wasserstein machine learning problems when the ambiguity set constitutes a Wasserstein ball centered around the empirical distribution of the available historical observations \citep{med17, BM19:quantifying_distr_model_risk, Gao_Kleywegt_2016}. Over the last few years, Wasserstein machine learning problems have attracted enormous attention in the machine learning and optimization communities; we refer to \cite{kmns19, KSW24:dro-survey} for recent reviews of the literature. Interestingly, Wasserstein learning problems admit dual characterizations as regularized learning problems \citep{NIPS2015, JMLR:v20:17-633, GCK22:wasserstein_dro_variation_reg}, and they thus contribute to a deeper understanding of the impact of regularization in machine learning. We note that other classes of ambiguity sets have been explored as well, such as moment ambiguity sets and those based on $\phi$-divergences (such as the Kullback-Leibler divergence). We will not delve into the comparative advantages of different ambiguity sets, and we instead refer the interested reader to the existing literature (see, \emph{e.g.}, \citealp{VPEK21:dro_is_optimal}, \citealp{kmns19, KSW24:dro-survey} and \citealp{L19:kl_divergence}).

Although Wasserstein formulations of many classical machine learning tasks admit formulations as convex optimization problems, these formulations scale exponentially in the discrete input features. This limitation has, thus far, confined the use of Wasserstein machine learning models primarily to datasets with exclusively continuous features. This constitutes a major restriction in operations management, where estimation problems frequently include discrete features. Recent examples 
include \cite{QCS22:dr_quantile_prediction}, who apply Wasserstein-based quantile regression to a bike sharing inventory management problem characterized by numeric and discrete features (\emph{e.g.}, the weather conditions, the hour of the day as well as the day of the week); \cite{SHBLS22:overbooked_overlooked}, who study appointment scheduling problems where most features are discrete (\emph{e.g.}, the day of the week, the patient's marital status and her insurance type); \cite{LTMCO23:detecting_human_trafficking}, who detect human trafficking from user review websites (here, the discrete features describe the presence or absence of indicative words and phrases); \cite{CMCSUWZ23:got_milk}, who predict the macronutrient content of human milk donations using discrete features such as the infant status (term vs preterm); and \cite{DHN23:DR_for_latent_covariate_mixtues}, who enforce fairness in offender recidivism prediction through the use of discrete features such as the offender's race, gender and the existence of prior misdemeanour charges. More broadly, at the time of writing, 240 of the 496 classification and 64 of the 159 regression problems in the popular UCI machine learning repository contain discrete input features \citep{UCI}.

{\color{black} We will demonstrate that na\"ively replacing discrete features with unbounded continuous features leads to pathological ambiguity sets in the Wasserstein learning problem whose worst-case distributions lack theoretical appeal and whose resulting prediction models underperform in practice. We also show that replacing discrete features with continuous features that are supported on suitably chosen convex hulls of the discrete supports leads to ambiguity sets that are \emph{equivalent} to those resulting from faithfully accounting for discrete features. Unfortunately, however, tractable reformulations of the Wasserstein learning problem over such bounded continuous features are only available for piece-wise affine convex loss functions, where we will demonstrate empirically that their solution times can be prohibitively long.} 

This paper studies {\color{black} linear} Wasserstein {\color{black} classification and regression problems} with mixed (continuous and discrete) features from a theoretical, computational and numerical perspective. {\color{black} Our work makes several contributions to the state-of-the-art:}
\begin{enumerate}
    \item[\emph{(i)}] From a \emph{theoretical perspective}, we demonstrate that while {\color{black} linear} Wasserstein {\color{black} classification and regression} with mixed features are inherently NP-hard, a wide range of problems can be solved in polynomial time. Also, contrary to {\color{black} problems} with exclusively continuous features, we establish that mixed-feature {\color{black} linear} Wasserstein {\color{black} classification and regression do} not reduce to regularized {\color{black} problems}. {\color{black} On the other hand, we demonstrate that mixed-feature linear Wasserstein classification and regression problems are equivalent to bounded continuous-feature problems with suitably chosen supports.}
    \item[\emph{(ii)}] From a \emph{computational perspective}, we propose a cutting plane scheme that solves progressively refined relaxations of the {\color{black} linear} Wasserstein {\color{black} classification and regression problems} as convex optimization problems. While our overall scheme is not guaranteed to terminate in polynomial time, we show that the key step of our algorithm---the identification of the most violated constraint---can be implemented efficiently for broad classes of problems, despite {\color{black} the presence of exponentially many constraints.}
    \item[\emph{(iii)}] {\color{black} From a \emph{numerical perspective}, we show that our cutting plane scheme is substantially faster than an equivalent polynomial-size problem reformulation via bounded continuous features.} We also show that our model can perform favorably against classical, regularized and alternative robust problem formulations on standard benchmark instances.
\end{enumerate}

Our paper is most closely related to the recent work of \cite{NIPS2015} and \cite{JMLR:v20:17-633}. \cite{NIPS2015} formulate the Wasserstein logistic regression problem as a convex optimization problem, they discuss the out-of-sample guarantees of their model, and they report numerical results on simulated and benchmark instances. \cite{JMLR:v20:17-633} extend their previous work to a wider class of Wasserstein classification and regression problems. Both papers focus on problems with exclusively continuous features, and their proposed formulations would scale exponentially in any discrete features. In contrast, our work studies {\color{black} linear} Wasserstein {\color{black} classification and regression} problems with mixed features: we examine the theoretical properties of such problems, we develop a practically efficient solution scheme, and we report numerical results. Our work also relates closely to a recent stream of literature that characterizes Wasserstein learning problems as regularized learning problems \citep{NIPS2015, JMLR:v20:17-633, blanchet_kang_murthy_2019, GCK22:wasserstein_dro_variation_reg}. In particular, we demonstrate that our mixed-feature Wasserstein learning problems do not admit an equivalent representation as regularized learning problems, which forms a notable contrast to the existing findings from the literature.

The present work constitutes a completely revised and substantially expanded version of a conference paper \citep{selvi2023wasserstein-non-redacted}. While that work focuses on logistic regression, the present paper studies broad classes of {\color{black} linear} Wasserstein classification and regression problems. This expansion necessitates significant adaptations of the proof for the computational complexity of the Wasserstein learning problem (\emph{cf.}~Theorem~\ref{thm:complexity} in {\color{black} Section~\ref{sec-complexity}}), an entirely new proof for the absence of regularized problem formulations (\emph{cf.}~Theorem~\ref{thm-equivalency} in {\color{black} Section~\ref{sec-complexity}}) that applies to any loss function (as opposed to only the log-loss function in \citealp{selvi2023wasserstein-non-redacted}) {\color{black} and any non-trivial Wasserstein learning instance}, as well as a substantially generalized cutting plane scheme (\emph{cf.}~Algorithms~\ref{alg-cutting} and~\ref{alg-most-violated} as well as Theorem~\ref{thm-violation} in Section~\ref{sec-cutting}). We also {\color{black} show that the mixed-feature Wasserstein learning problem is equivalent to a bounded continuous-feature formulation, and we} present a considerably augmented set of numerical results that encompass both classification and regression problems (\emph{cf.}~Section~\ref{sec-numerics}).

The remainder of this paper is organized as follows. {\color{black} Section~\ref{sec-wasser} introduces the mixed-feature linear Wasserstein classification and regression problems of interest, and it develops a unified representation of both problems as convex programs of exponential size. Section~\ref{sec-cutting} presents and analyzes our cutting plane solution approach for this exponential-size convex program. Section~\ref{sec-complexity} shows that while mixed-feature linear Wasserstein classification and regression problems are generically NP-hard, important special cases can be solved in polynomial time. We also show that mixed-feature linear Wasserstein classification and regression problems do not reduce to regularized problems. Section~\ref{sec-comparison} contrasts our mixed-feature model against formulations that replace discrete features with (bounded or unbounded) continuous ones. We report on numerical experiments in Section~\ref{sec-numerics}, and we offer concluding remarks in Section~\ref{sec-conclusions}. Exponential-size convex reformulations of the mixed-feature linear Wasserstein classification and regression problems, which are unified by our representation in Section~\ref{sec-wasser}, as well as all proofs are relegated to the e-companion.} All datasets and source codes accompanying this work are available open source.\footnote{Website: \url{https://anonymous.4open.science/r/Wasserstein-Mixed-Features-088D/}.}

\textbf{Notation.}
We denote by $\mR$ ($\mR_+$, $\mR_-$) the set of (non-negative, non-positive) real numbers, by $\mathbb{N}$ the set of positive integers, and we define $\mathbb{B} = \{ 0, 1 \}$ as well as $[N] = \{ 1, \ldots, N \}$ for $N \in \mathbb{N}$. For a proper cone $\mathcal{C} \subseteq \mR^{n}$, we write $\bx \preccurlyeq_{\mathcal{C}} \bx'$ and $\bx \prec_{\mathcal{C}} \bx'$ to abbreviate $\bx' - \bx \in \mathcal{C}$ and $\bx' - \bx \in \text{int } \mathcal{C}$, respectively. The dual norm of $\lVert \cdot \rVert$ is $\lVert \bx \rVert_{*} = \sup_{\bx' \in \mR^n} \{ \bx^\top \bx' \, : \, \lVert \bx' \rVert \leq 1 \}$, and the cone dual to a cone $\mathcal{C}$ is $\mathcal{C}^* = \{ \bx' \, : \, \bx^\top \bx' \geq 0 \;\; \forall \bx \in \mathcal{C} \}$. The support function of a set $\mX \subseteq \mR^n$ is $\mathcal{S}_{\mX}(\bx) = \sup \{ \bx^\top\bx' \, : \, \bx' \in \mX \}$. For a function $L: \mX \rightarrow \mR$, we define the Lipschitz modulus as $\mathrm{lip}(L) = \sup \{ \lvert L(\bx) - L(\bx') \rvert \, / \, \lVert \bx - \bx' \rVert \, : \, \bx, \bx' \in \mX, \; \bx \neq \bx' \}$.
The set $\mathcal{P}_0 (\Xi)$ contains all probability distributions supported on $\Xi$, and the Dirac distribution $\delta_{\bm{x}} \in \mathcal{P}_0 (\mathbb{R}^n)$ places unit probability mass on $\bm{x} \in \mathbb{R}^n$. The indicator function $\indicate{\mathcal{E}}$ attains the value $1$ ($0$) whenever the logical expression $\mathcal{E}$ is (not) satisfied.

\section{\color{black} Mixed-Feature Wasserstein Classification and Regression}\label{sec-wasser}

We study learning problems over $N$ data points $\bxi^n = (\bx^n, \bz^n, y^n) \in \Xi = \mX \times \mC \times \mY$, $n \in [N]$, where $\bx^n$, $\bz^n$ and $y^n$ represent the numerical features, the discrete features and the output variable, respectively. We assume that the support $\mX$ of the numerical features is a closed and convex subset of $\mR^{M_\mx}$. The support $\mC$ of the $K$ discrete features satisfies $\mC = \mC (k_1) \times \ldots \times \mC (k_K)$, where $k_m \in \mathbb{N} \setminus \{ 1 \}$ denotes the number of values that the $m$-th discrete feature can attain, $m \in [K]$, and $\mC (s) = \{ \bz \in \mB^{s - 1} \, : \, \sum_{i \in [s-1]} z_i \leq 1 \}$ is the one-hot feature encoding. We let $M_\mz = \sum_{m \in [K]} (k_m - 1)$ denote the number of coefficients associated with the discrete features. The support $\mY$ of the output variable is $\{ -1, +1 \}$ for classification and a closed and convex subset of $\mR$ for regression problems, respectively. \mbox{We wish to solve the Wasserstein learning problem}
\begin{equation}\label{eq:the_mother_of_all_problems}
    \begin{array}{l@{\quad}l}
        \displaystyle \mathop{\mathrm{minimize}}_{\bbeta} & \displaystyle \sup_{\mQ \in \mathfrak{B}_\epsilon (\widehat{\mP}_N)} \; \mE_\mQ \left[ l_{\bbeta} (\bx, \bz, y) \right] \\[5mm]
        \displaystyle \mathrm{subject\;to} & \displaystyle \bbeta = (\beta_0, \bbeta_\mx, \bbeta_\mz) \in \mR^{1 + M_\mx + M_\mz},
    \end{array}
\end{equation}
where the ambiguity set $\mathfrak{B}_\epsilon (\widehat{\mP}_N) = \{ \mQ \in \mathcal{P}_0 (\Xi) \, : \, \mathrm{W} (\mQ, \widehat{\mP}_N) \leq \epsilon \}$ represents the Wasserstein ball of radius $\epsilon > 0$ that is centered at the empirical distribution $\widehat{\mathbb{P}}_N = \frac{1}{N} \sum_{n \in [N]} \delta_{\bm{\xi}^n}$ placing equal probability mass on the $N$ data points $\bm{\xi}^n$, $n \in [N]$, as per the following definition.

\begin{defi}[Wasserstein Distance]\label{def:wasserstein}
    The type-1 \emph{Wasserstein} (\emph{Kantorovich-Rubinstein}, or \emph{earth mover's}) \emph{distance} between two distributions $\mP \in \mathcal{P}_0 (\Xi)$ and $\mQ \in \mathcal{P}_0 (\Xi)$ is defined as
    \begin{equation*}\label{eq:Wass}
        \mspace{-11mu}
        \mathrm{W} (\mP, \mQ) := \inf_{\Pi \in \mathcal{P}_0 (\Xi^2)} \left\{ \int_{\Xi^2} d(\bxi,\bxi') \, \Pi(\mathrm{d} \bxi, \mathrm{d} \bxi')  \ : \  \Pi(\mathrm{d} \bxi, \Xi) = \mP(\mathrm{d} \bxi), \, \Pi(\Xi, \mathrm{d} \bxi') = \mQ(\mathrm{d} \bxi') \right\},
    \end{equation*}
    where the ground metric $d$ on $\Xi$ satisfies
    \begin{subequations}\label{eq:ground_metric}
        \begin{equation}\label{eq:ground_metric:overall}
            d (\bxi,\bxi') = 
            \lVert \bx-\bx' \rVert + \kappa_\mz d_\mz(\bz, \bz') + \kappa_\my d_\my (y,y')
            \quad
            \forall \bxi = (\bx,\bz,y) \in \Xi, \, \bxi' = (\bx',\bz',y') \in \Xi
        \end{equation}
        with $\kappa_\mz, \kappa_\my > 0$ as well as, for some $p \geq 1$,
        \begin{equation}\label{eq:ground_metric:discrete}
            d_\mz (\bz, \bz') = \left( \sum_{m \in [K]} \indicate{\bz_m \neq \bz'_m} \right)^{1 / p}
            \quad \text{and} \quad
            d_\my (y, y') = \begin{cases}
                \indicate{y \neq y'} & \text{if } \mY = \{ -1, +1 \}, \\
                \lvert y - y' \rvert & \text{otherwise}.
            \end{cases}
        \end{equation}
    \end{subequations}
\end{defi}
{\color{black} An important special case of Definition~\ref{def:wasserstein} arises when $\kappa_\my = \infty$. In this setting, any transport between samples with different outputs incurs infinite cost, and hence the ambiguity set only contains distributions that preserve the empirical output values. In other words, $\mathfrak{B}_\epsilon(\widehat{\mP}_N)$ then consists of perturbations of the feature distribution conditional on $y$, while the marginal distribution of $y$ remains fixed.}

The loss function $l_{\bbeta} (\bx, \bz, y): \mX \times \mC \times \mY \rightarrow \mR_+$ in problem~\eqref{eq:the_mother_of_all_problems} satisfies 
\begin{equation*}
    l_{\bbeta}(\bx,\bz,y) = \begin{cases}
        L(y \cdot [\beta_0 + \bbeta_\mx{}^\top \bx + \bbeta_\mz{}^\top \bz]) & \text{if } \mY = \{ -1, +1 \}, \\
        L ( \beta_0 + \bbeta_\mx{}^\top \bx + \bbeta_\mz{}^\top \bz - y ) & \text{otherwise},
    \end{cases}
\end{equation*}
where $L : \mR \rightarrow \mR_+$ measures the similarity between the prediction $\beta_0 + \bbeta_\mx{}^\top \bx + \bbeta_\mz{}^\top \bz$ and the output $y$. {\color{black} With a slight abuse of terminology, we also refer to $L$ as the loss function whenever the context is clear.} For both classification and regression problems, we consider two settings:
\begin{enumerate}
    \item[(i)] $L$ is convex and Lipschitz continuous with Lipschitz modulus $\mathrm{lip}(L)$, $\mX = \mR^{M_\mx}$ and $\mY = \{ -1, +1 \}$ (for classification problems) or $\mY = \mR$ (for regression problems);
    \item[(ii)] $L$ satisfies $L(e) = \max_{j \in [J]} \{ a_j e + b_j \}$, and $\mX \subseteq \mR^{M_\mx}$ and $\mY \subseteq \mR$ are closed and convex.
\end{enumerate}
In either case, we assume that $L$ is not constant.

The Wasserstein learning problem~\eqref{eq:the_mother_of_all_problems} offers attractive generalization guarantees. While the classical choice of Wasserstein radii suffers from the curse of dimensionality \citep{med17}, recent work has developed asymptotic \citep{blanchet_kang_murthy_2019, blanchet_and_kang} as well as finite sample guarantees \citep{JMLR:v20:17-633, G22:finite_sample_guarantees} that apply to Wasserstein radii of the order $\mathcal{O} (1 / \sqrt{N})$.

We next review a result that expresses the Wasserstein learning problem~\eqref{eq:the_mother_of_all_problems} as a convex optimization problem.

\begin{obs}\label{obs-derivation}
    The Wasserstein learning problem~\eqref{eq:the_mother_of_all_problems} admits the equivalent formulation
    \begin{equation}\label{eq-dro-gen2}
        \begin{array}{l@{\quad}l}
        \displaystyle \mathop{\mathrm{minimize}}_{\bbeta, \lambda, \bs} & \displaystyle \lambda \epsilon + \frac{1}{N} \sum_{n \in [N]} s_n \\[5mm]
        \displaystyle \mathrm{subject\;to} & \displaystyle \sup_{(\bx, y) \in \mX \times \mY} \left\{ l_{\bbeta} (\bx, \bz, y) - \lambda \lVert \bx - \bx^n \rVert  - \lambda \kappa_\my d_\my (y, y^n) \right\} - \lambda \kappa_\mz d_\mz (\bz, \bz^n) \leq s_n \\[-2mm]
        & \displaystyle \mspace{415mu} \forall n \in [N], \; \forall \bz \in \mC \\[2mm]
        & \displaystyle \bbeta = (\beta_0, \bbeta_\mx, \bbeta_\mz) \in \mR^{1 + M_\mx + M_\mz}, \;\; \lambda \in \mathbb{R}_+, \;\; \bs \in \mR^N_+.
        \end{array}
    \end{equation}
\end{obs}

{\color{black} When $\kappa_\my = \infty$, the semi-infinite constraint in Observation~\ref{obs-derivation} simplifies to
\begin{equation*}
    \sup_{\bx \in \mX} \left\{ l_{\bbeta} (\bx, \bz, y^n) - \lambda \lVert \bx - \bx^n \rVert  \right\} - \lambda \kappa_\mz d_\mz (\bz, \bz^n) \leq s_n \qquad \forall n \in [N], \; \forall \bz \in \mC.
\end{equation*}}

Problem~\eqref{eq-dro-gen2} contains embedded maximization problems, and it comprises exponentially many constraints. {\color{black} We next prove the existence of equivalent reformulations of~\eqref{eq-dro-gen2} for classification and regression problems, respectively, that do not contain embedded optimization problems. Since the resulting reformulations exhibit an exponential number of constraints, the next section will develop a cutting plane approach to introduce these constraints iteratively.}

{\color{black}Our equivalent reformulations of~\eqref{eq-dro-gen2} leverage the unified problem representation
\begin{equation}\label{eq-most-vio-generic}
    \begin{array}{l@{\quad}l@{\qquad}l}
        \displaystyle \mathop{\mathrm{minimize}}_{\btheta, \bsigma} & \displaystyle f_0(\btheta,\bsigma) \\
        \displaystyle \mathrm{subject\;to} & \displaystyle f_{ni} (g_{ni} (\btheta, \bm{\xi}^n_{-\mathbf{z}}; \bz)) - h_{ni} (\btheta, \bm{\xi}^n_{-\mathbf{z}}; d_{\mz} (\bz, \bz^n)) \leq \sigma_n & \displaystyle \forall n \in [N], \; \forall i \in \mathcal{I}, \; \forall \bz \in \mC \\
        & \displaystyle \btheta \in \Theta, \;\; \bsigma \in \mR^{N}_+,
    \end{array}
\end{equation}
where $\bm{\xi}^n_{-\mathbf{z}} = (\bm{x}^n, y^n)$ and $\mathcal{I}$ is a finite index set. To ensure that~\eqref{eq-most-vio-generic} is convex, we stipulate that $f_0 : \mR^{M_{\btheta}} \times \mR^{N}_+ \rightarrow \mR$  and $f_{ni} : \mR \rightarrow \mR$ are convex, $g_{ni} : \mR^{M_{\btheta}} \times (\mathbb{X} \times \mathbb{Y}) \times \mR^{M_\mz} \rightarrow \mR$ is bi-affine in $\btheta$ and $\bz$ for every fixed $\bm{\xi}^n_{-\mathbf{z}}$, $h_{ni}: \mR^{M_{\btheta}} \times (\mathbb{X} \times \mathbb{Y}) \times \mR \rightarrow \mR$ is concave in $\btheta$ for every fixed $\bm{\xi}^n_{-\mathbf{z}}$ and every fixed value of its last component, $n \in [N]$ and $i \in \mathcal{I}$, and $\Theta \subseteq \mR^{M_{\btheta}}$ is a convex set. We further assume that $f_0$ is non-decreasing in $\bsigma$ so that~\eqref{eq-most-vio-generic} is not unbounded.

\begin{thm}\label{thm-convex-reformulation}
    The following problems admit equivalent reformulations in the form of~\eqref{eq-most-vio-generic}:
    \begin{enumerate}
        \item[\emph{(i)}] the mixed-feature Wasserstein classification and regression problems with convex and Lipschitz continuous loss functions $L$ as well as the continuous feature support $\mX = \mR^{M_\mx}$.
        \item[\emph{(ii)}] the mixed-feature Wasserstein classification and regression problems with piece-wise affine convex loss functions $L(e) = \max_{j \in [J]} \{ a_j e + b_j \}$ as well as the continuous feature support
        \begin{equation*}
            \begin{array}{l@{}l}
                \displaystyle
                \mX = \{ \bx \in \mR^{M_\mx} \, : \, \bm{C} \bx \preccurlyeq_{\mathcal{C}} \bm{d} \} \;\;
                & \text{for some }
                \bm{C} \in \mR^{r \times {M_\mx}}, \bm{d} \in \mR^r \\
                & \displaystyle
                \text{and proper convex cone } \mathcal{C} \subseteq \mR^{r}
            \end{array}
        \end{equation*}
        for classification problems as well as the support
        \begin{equation*}
            \begin{array}{l@{}l}
                \displaystyle
                \mX \times \mY = \left\{ (\bx,y) \in \mR^{M_\mx+1} \, : \, \bm{C}_\mx \bx + \bm{c}_\my \cdot y \preccurlyeq_{\mathcal{C}} \bm{d} \right\} \;\;
                & \text{for some }
                \bm{C}_\mx \in \mR^{r \times {M_\mx}}, \bm{c}_\my \in \mR^{r}, \bm{d} \in \mR^{r} \\
                & \displaystyle
                \text{and proper convex cone } \mathcal{C} \subseteq \mR^{r}
            \end{array}
        \end{equation*}
        for regression problems, assuming that the supports admit Slater points.
    \end{enumerate}
\end{thm}

The Appendices~A and~B in the e-companion prove Theorem~\ref{thm-convex-reformulation} for classification and regression problems, respectively. The proofs follow similar arguments as those of \cite{JMLR:v20:17-633}, adapted to the presence of discrete features as well as our ground metric.}

\section{Cutting Plane Solution Scheme}\label{sec-cutting}

{\color{black} The unified problem representation~\eqref{eq-most-vio-generic} of the mixed-feature Wasserstein regression and classification problems (\emph{cf.}~Theorem~\ref{thm-convex-reformulation}) is convex, but it comprises  constraints whose number scales exponentially in $K$, the number of discrete features. We therefore develop a cutting plane approach that iteratively introduces only those constraints that are most violated by a sequence of incumbent solutions.}

\begin{algorithm}[tb]
    \begin{algorithmic}
        \caption{Cutting Plane Scheme for Problem~\eqref{eq-most-vio-generic}} \label{alg-cutting}
        \Require (Possibly empty) initial constraint set $\mathcal{W} \subseteq [N] \times \mathcal{I} \times \mC$.
        \Ensure Optimal solution $(\btheta^\star, \bsigma^\star)$ to problem~\eqref{eq-most-vio-generic}.
        \State Initialize $(\mathrm{LB}, \mathrm{UB}) = (-\infty, +\infty)$ as lower and upper bounds for problem~\eqref{eq-most-vio-generic}.
        \While{$\mathrm{LB} < \mathrm{UB}$}
            \State Let $(\btheta^\star, \bsigma^\star)$ be optimal in the relaxation of~\eqref{eq-most-vio-generic} involving the constraints $(n, i, \bm{z}) \in \mathcal{W}$.
            \For{$n \in [N]$}
                \State Identify, for each $i \in \mathcal{I}$, a most violated constraint
                \begin{equation*}
                    \bm{z} (n, i) \in \mathop{\arg \max}_{\bm{z} \in \mC} \left\{
                    f_{ni} (g_{ni} (\btheta^\star, \bm{\xi}^n_{-\mathbf{z}}; \bz)) - h_{ni} (\btheta^\star, \bm{\xi}^n_{-\mathbf{z}}; d_{\mz} (\bz, \bz^n)) - \sigma^\star_n
                    \right\}
                \end{equation*}
                \State associated with $(\btheta^\star, \bsigma^\star)$ and $(n, i)$, and denote the constraint violation by $\vartheta (n, i)$.
                \State Let $i(n) \in \arg \max \{ \vartheta (n, i) : i \in \mathcal{I} \}$ and add $(n, i(n), \bm{z} (n, i(n)))$ to $\mathcal{W}$ if $\vartheta (n, i(n)) > 0$.
            \EndFor
            \State Define $\bm{\vartheta}^\star \in \mathbb{R}^N$ via $\vartheta^\star_n = \max \{ \vartheta (n, i(n)), \, 0 \}$, $n \in [N]$.
            \State Update $\mathrm{LB} = f_0 (\btheta^\star, \bsigma^\star)$ and $ \mathrm{UB} = \min \{ \mathrm{UB}, \, f_0 (\btheta^\star,\bsigma^\star + \bm{\vartheta}^\star) \}$.
        \EndWhile
    \end{algorithmic}
\end{algorithm}

{\color{black} Algorithm~\ref{alg-cutting} outlines our cutting plane scheme, which follows the standard methodology of this class of techniques. The algorithm iteratively solves a sequence of relaxations of problem~\eqref{eq-most-vio-generic}, where each relaxation only includes a subset $\mathcal{W}$ of the constraints indexed by $[N] \times \mathcal{I} \times \mC$. Consequently, any optimal solution $(\btheta^\star, \bsigma^\star)$ obtained from the relaxation gives rise to a valid lower bound for the optimal value of problem~\eqref{eq-most-vio-generic}. At each iteration, the algorithm identifies and incorporates the most violated constraint from problem~\eqref{eq-most-vio-generic} with respect to the current solution $(\btheta^\star, \bsigma^\star)$. This procedure ensures the generation of a monotonically non-decreasing sequence of lower bounds, which is guaranteed to converge to the optimal value of problem~\eqref{eq-most-vio-generic} in a finite number of iterations. A common challenge in cutting plane techniques is the derivation of upper bounds, which allow us to quantify the optimality gap, particularly when the algorithm is terminated prematurely, such as when reaching a predefined time limit. In our case, the special structure of problem~\eqref{eq-most-vio-generic} enables us to obtain such upper bounds efficiently: We adjust the optimal solution $(\btheta^\star, \bsigma^\star)$ of the relaxation by increasing the slack variables $\bsigma^\star$ so as to account for the maximum constraint violations $\bm{\vartheta}^\star$.

The next result formalizes our intuition.}

\begin{prop}\label{prop-cutting}
    Algorithm~\ref{alg-cutting} terminates in finite time with an optimal solution $(\btheta^\star, \bsigma^\star)$ to problem~\eqref{eq-most-vio-generic}. Moreover, $\mathrm{LB}$ and $\mathrm{UB}$ constitute monotonic sequences of lower and upper bounds on the optimal value of~\eqref{eq-most-vio-generic} throughout the execution of the algorithm.
\end{prop}

A key step in Algorithm~\ref{alg-cutting} concerns the identification of a constraint $(n, i, \bm{z}) \in [N] \times \mathcal{I} \times \mC$ that is most violated by the incumbent solution $(\btheta^\star, \bm{\sigma}^\star)$. We next argue that the underlying combinatorial problem can be solved efficiently by Algorithm~\ref{alg-most-violated}.

\begin{algorithm}[tb]
    \begin{algorithmic}
        \caption{Identification of a Most Violated Constraint in Problem~\eqref{eq-most-vio-generic}}
        \label{alg-most-violated}
        \Require Incumbent solution $(\btheta^\star, \bsigma^\star)$ and constraint group index $(n, i) \in [N] \times \mathcal{I}$.
        \Ensure A most violated constraint index $\bz (n, i)$ in constraint group $(n, i)$.
        \State Initialize the candidate constraint index set $\mathcal{Z} = \emptyset$.
        \State Let $(\bm{w}, w_0) \in \mR^{M_\mz} \times \mR$ be such that $g_{ni} (\btheta^\star, \bm{\xi}^n_{-\mathbf{z}}; \bz) = \bm{w}^\top \bz + w_0$.
        \For{$\mu \in \{ \pm 1 \}$}
            \State Compute $\bz^\star_m \in \arg \max \{ \mu \cdot \bm{w}_m{}^\top \bz_m \, : \, \bz_m \in \mC (k_m) \setminus \{ \bz^n_m \} \}$ for all $m \in [K]$.
            \State Compute a permutation $\pi : [K] \rightarrow [K]$ such that
            \begin{equation*}
                \mu \cdot {\bw}_{\pi(m)}{}^\top (\bz^{\star}_{\pi (m)} - \bz^n_{\pi (m)})
                \; \geq \;
                \mu \cdot {\bw}_{\pi(m')}{}^\top (\bz^{\star}_{\pi (m')} - \bz^n_{\pi (m')})
                \qquad \forall 1 \leq m \leq m' \leq K.
            \end{equation*}
            \For{$\delta \in [K] \cup \{ 0 \}$}
                \State Add $\bz (\delta)$ to $\mathcal{Z}$, where $\bz_m (\delta) = \bz^{\star}_m$ if $\pi (m) \leq \delta$; $= \bz^n_m$ otherwise, $m \in [K]$.
            \EndFor
        \EndFor
        \State Select $\bz (n, i) \in \arg \max \{ f_{ni} (g_{ni} (\btheta^\star, \bm{\xi}^n_{-\mathbf{z}}; \bz)) - h_{ni} (\btheta^\star, \bm{\xi}^n_{-\mathbf{z}}; d_{\mz} (\bz, \bz^n)) - \sigma^\star_n \, : \, \bm{z} \in \mathcal{Z} \}$.
    \end{algorithmic}
\end{algorithm}

{\color{black}
Algorithm~\ref{alg-most-violated} determines a most violated constraint index, denoted as $\bm{z} (n, i)$, for a given incumbent solution $(\btheta^\star, \bm{\sigma}^\star)$ and a given constraint group $(n, i) \in [N] \times \mathcal{I}$ in the optimization problem~\eqref{eq-most-vio-generic}. We obtain such a constraint by maximizing the constraint left-hand side in~\eqref{eq-most-vio-generic},
\begin{equation*}
    f_{ni} (g_{ni} (\btheta^\star, \bm{\xi}^n_{-\mathbf{z}}; \bz)) - h_{ni} (\btheta^\star, \bm{\xi}^n_{-\mathbf{z}}; d_{\mz} (\bz, \bz^n)),
\end{equation*}
over $\bm{z} \in \mathbb{Z}$. Note that the constraint function comprises two distinct terms. The second term, $h_{ni} (\btheta^\star, \bm{\xi}^n_{-\mathbf{z}}; d_{\mz} (\bz, \bz^n))$, depends on $\bm{z}$ only through the distance metric $d_{\mz} (\bz, \bz^n)$ within the ground metric $d$ (\emph{cf.}~Definition~\ref{def:wasserstein}). Consequently, we can decompose the maximization of the difference of both terms into two steps. In a first step, we maximize the first term $f_{ni} (g_{ni} (\btheta^\star, \bm{\xi}^n_{-\mathbf{z}}; \bz))$ over all $\bm{z} \in \mathbb{Z}$ that satisfy $d_{\mz} (\bz, \bz^n) = \delta$ for any fixed $\delta \in [K] \cup \{ 0 \}$. This is what the inner for-loop in Algorithm~\ref{alg-most-violated} accomplishes, and it results in $K + 1$ candidate solutions $\bm{z} (\delta)$. In a second step, we can then select an optimal solution $\bm{z} (\delta^\star)$ that maximizes
\begin{equation*}
    f_{ni} (g_{ni} (\btheta^\star, \bm{\xi}^n_{-\mathbf{z}}; \bz (\delta))) - h_{ni} (\btheta^\star, \bm{\xi}^n_{-\mathbf{z}}; \delta)
\end{equation*}
across all $\delta \in [K] \cup \{ 0 \}$ by evaluating the constraint left-hand side for the $K + 1$ candidate solutions $\bm{z} (\delta)$. This is what the final step in Algorithm~\ref{alg-most-violated} implements.

To maximize $f_{ni} (g_{ni} (\btheta^\star, \bm{\xi}^n_{-\mathbf{z}}; \bz))$ over all $\bm{z} \in \mathbb{Z}$ satisfying $d_{\mz} (\bz, \bz^n) = \delta$, note first that $f_{ni}$ is assumed to be univariate and convex. Thus, its maximum over any compact set $\mathcal{X} \subseteq \mathbb{R}$ is attained either at $\min \{ x \, : \, x \in \mathcal{X} \}$ or $\max \{ x \, : \, x \in \mathcal{X} \}$. Hence, to maximize $f_{ni} (g_{ni} (\btheta^\star, \bm{\xi}^n_{-\mathbf{z}}; \bz))$ over all $\bm{z} \in \mathbb{Z}$ with $d_{\mz} (\bz, \bz^n) = \delta$, we only need to maximize and minimize the inner function $g_{ni} (\btheta^\star, \bm{\xi}^n_{-\mathbf{z}}; \bz)$ over all $\bm{z} \in \mathbb{Z}$ with $d_{\mz} (\bz, \bz^n) = \delta$. The maximization and minimization of $g_{ni}$ is unified by the outer for-loop in Algorithm~\ref{alg-most-violated}, which maximizes $\mu \cdot g_{ni} (\btheta^\star, \bm{\xi}^n_{-\mathbf{z}}; \bz)$ for $\mu \in \{ \pm 1 \}$. 

We claim that the optimization of $g_{ni} (\btheta^\star, \bm{\xi}^n_{-\mathbf{z}}; \bz)$ over $\bm{z} \in \mathbb{Z}$ with $d_{\mz} (\bz, \bz^n) = \delta$ can be decomposed into separate optimizations over $\bm{z}_m \in \mathbb{Z} (k_m)$, $m \in [K]$. Indeed, note that $g_{ni}$ is affine in $\bm{z}$ for fixed $\bm{\theta}^\star$, which allows us to express the mapping $\bm{z} \mapsto g_{ni} (\btheta^\star, \bm{\xi}^n_{-\mathbf{z}}; \bz)$ as $\bm{w}^\top \bm{z} + w_0$ for some $\bm{w} \in \mathbb{R}^{M_{\mz}}$ and $w_0 \in \mathbb{R}$. Moreover, the distance metric $d_{\mz} (\bz, \bz^n)$ only counts the number of subvectors $m \in [K]$ where $\bm{z}_m$ differs from $\bm{z}_m^n$, without considering the magnitude of the difference (as both vectors are one-hot encoded). Finally, we recall that $\mathbb{Z}$ is rectangular, that is, $\mathbb{Z} = \mathbb{Z} (k_1) \times \ldots \times \mathbb{Z} (k_K)$. Taken together, the aforementioned three observations imply that the optimization of $g_{ni} (\btheta^\star, \bm{\xi}^n_{-\mathbf{z}}; \bz)$ over $\bm{z} \in \mathbb{Z}$ with $d_{\mz} (\bz, \bz^n) = \delta$ can be broken into two interconnected steps: We need to select the $\delta$ subvectors $m_1 < \ldots < m_\delta$ where $\bm{z}$ differs from $\bm{z}^n$, and for each differing subvector $m$, we need to select a solution $\bm{z}^\star_m$ that optimizes $\bm{w}_m{}^\top (\bm{z}_m - \bm{z}^n_m)$ across all $\bm{z}_m \neq \bm{z}^n_m$. To this end, Algorithm~\ref{alg-most-violated} first computes for every $m \in [K]$ a solution $\bm{z}^\star_m$ that optimizes $\bm{w}_m{}^\top (\bm{z}_m - \bm{z}^n_m)$ over $\bm{z}_m \neq \bm{z}^n_m$ (first step inside the outer for-loop), it subsequently sorts the contributions $\bm{w}_m{}^\top (\bm{z}^{\star}_m - \bm{z}^n_m)$, $m \in [K]$ (second step inside the outer for-loop), and it finally selects the $\delta$ largest (smallest) contributions to maximize (minimize) $g_{ni}$ (inner for-loop).

We formalize the above intuition in the next result.}

\begin{thm}\label{thm-violation}
    For a given incumbent solution $(\btheta^\star, \bm{\sigma}^\star)$ and a constraint group index $(n, i) \in [N] \times \mathcal{I}$ in problem~\eqref{eq-most-vio-generic}, Algorithm~\ref{alg-most-violated} identifies a most violated constraint index $\bm{z} (n, i)$ in time $\mathcal{O} (M_{\mathbf{z}} + KT + K \log K)$, where $T$ is the time required to compute $f_{ni}$, $g_{ni}$ and $h_{ni}$.
\end{thm}

\section{\color{black} Complexity Analysis}\label{sec-complexity}

{\color{black}
Despite its exponential size, our unified problem representation~\eqref{eq-most-vio-generic} admits a polynomial time solution for the classes of loss functions that we consider in this paper. In fact, it follows from Theorem~\ref{thm-violation} that Algorithm~\ref{alg-most-violated} provides a polynomial time separation oracle for problem~\eqref{eq-most-vio-generic}.

\begin{thm}[Complexity of the Unified Problem Representation~\eqref{eq-most-vio-generic}]\label{thm:complexity}
    Assume that the subgradients of the functions $f_{ni}$ and $-h_{ni}$ in problem~\eqref{eq-most-vio-generic} can be computed in polynomial time and that $\Theta$ is compact, full-dimensional and admits a polynomial time weak separation oracle. Then problem~\eqref{eq-most-vio-generic} can be solved to $\delta$-accuracy in polynomial time.
\end{thm}

Recall that an optimization problem is solved to $\delta$-accuracy if a $\delta$-suboptimal solution is identified that satisfies all constraints modulo a violation of at most $\delta$. The consideration of $\delta$-accurate solutions is standard in the numerical solution of nonlinear programs where an optimal solution may be irrational. We show in the Appendices~A and~B that the subgradients of $f_{ni}$ and $-h_{ni}$ can be computed in polynomial time and that $\Theta$ can be assumed to be compact in the Wasserstein classification and regression problems that we consider. The regularity assumptions imposed on problem~\eqref{eq-most-vio-generic} are crucial in this regard. In fact, it follows from Theorem~2 of \cite{selvi2023wasserstein-non-redacted} that if the functions $f_{ni}$ in problem~\eqref{eq-most-vio-generic} were allowed to be non-convex, then the problem would be strongly NP-hard even if $M_{\btheta} = 0$, $N = | \mathcal{I} | = 1$ and $h_{11}$ vanished.

While Theorem~\ref{thm:complexity} principally allows us to solve problem~\eqref{eq-most-vio-generic} in polynomial time, the available solution algorithms (see, \emph{e.g.}, \citealt{LSW15:ellipsoid}) rely on variants of the ellipsoid method and are not competitive on practical problem instances. For this reason, we rely on our cutting plane algorithm throughout our numerical experiments.

A by now classical result shows that when $K = 0$ (absence of discrete features) and $\mX = \mR^{M_\mx}$, the Wasserstein learning problem~\eqref{eq:the_mother_of_all_problems} reduces to a classical learning problem with an additional regularization term in the objective function whenever the output weight $\kappa_{\my}$ in Definition~\ref{def:wasserstein} approaches $\infty$ \citep{NIPS2015, JMLR:v20:17-633, blanchet_kang_murthy_2019, GCK22:wasserstein_dro_variation_reg}. It turns out that this reduction no longer holds when discrete features are present.

\begin{thm}[Absence of Regularizers]\label{thm-equivalency}
    Fix any convex Lipschitz continuous or piece-wise affine convex loss function $L$ that is not constant, any Wasserstein classification or regression instance (\emph{cf.}~Theorem~\ref{thm-convex-reformulation} as well as Appendices~A and~B) with any number $N \geq 1$ training samples, non-zero numbers $M_\mx$ of continuous and $K$ of discrete features, $\mX = \mR^{M_\mx}$, as well as any $\kappa_\mz > 0$ and any $p \geq 1$ in the definition of the Wasserstein ground metric. For sufficiently large radii $\epsilon$ of the ambiguity set, the objective function
    \begin{equation*}
        \sup_{\mQ \in \mathfrak{B}_\epsilon (\widehat{\mP}_N)} \; \mE_\mQ \left[ l_{\bbeta} (\bx, \bz, y) \right],
    \end{equation*}
    of the Wasserstein learning problem does not admit an equivalent reformulation as a classical regularized learning problem,
    \begin{equation*}
        \mathbb{E}_{\widehat{\mathbb{P}}_N} \left[ l_{\bm{\beta}} (\bm{x}, \bm{z}, y) \right] + \mathfrak{R} (\bm{\beta})
        \qquad
        \text{for any } \mathfrak{R} : \mathbb{R}^{1 + {M_\mx} + M_\mz} \rightarrow \mathbb{R},
    \end{equation*}
    even when the weight $\kappa_{\my}$ of the output distance $d_\mathrm{y}$ approaches $\infty$.
\end{thm}

We emphasize that Theorem~\ref{thm-equivalency} applies to \emph{any} loss function $L$ and \emph{any} reguarlizer $\mathfrak{R}$. We are not aware of any prior results of this form in the literature.}

{\color{black}
\section{Comparison with Continuous-Feature Formulations}\label{sec-comparison}

We next contrast the mixed-feature Wasserstein learning problem~\eqref{eq-dro-gen2} with an \emph{unbounded} continuous-feature formulation that replaces the support $\mathbb{Z}$ of the discrete features $\bm{z}$ with $\mathbb{R}^{M_\mz}$ (Section~\ref{subsec-comparison}), as well as with a \emph{bounded} continuous-feature formulation that replaces $\mathbb{Z}$ with its convex hull $\mathrm{conv} (\mathbb{Z})$ (Section~\ref{section:bounded-feature}). This section focuses on a qualitative comparison of the three formulations; a quantitative comparison in terms of runtimes and generalization errors on benchmark instances is relegated to our numerical experiments (Section~\ref{sec-numerics}).

\subsection{Comparison with Unbounded Continuous-Feature Formulation}\label{subsec-comparison}

Our reformulation~\eqref{eq-dro-gen2} of the mixed-feature Wasserstein learning problem scales exponentially in the discrete features. It may thus be tempting to replace the support $\mathbb{Z}$ with $\mathbb{R}^{M_\mz}$, which would allow us to solve problem~\eqref{eq-dro-gen2} in polynomial time using the reformulations proposed by \cite{NIPS2015} and \cite{JMLR:v20:17-633}. We next present a stylized example to illustrate that disregarding the discrete nature of the features $\bm{z}$ can result in pathological worst-case distributions that lack relevance to the problem.

Consider a classification problem where the single feature is binary and follows the Bernoulli distribution $z \sim \frac{1}{2}\delta_{-1} + \frac{1}{2}\delta_{1}$, and where the output variable $y \in \{ -1, +1\}$ is related to the feature $z$ via $\mathrm{Prob}(y = z \mid z) = 0.8$. Note that in slight deviation from our earlier convention, the binary feature is supported on $\{-1, +1\}$ instead of $\{0,1\}$; this allows us to present the example without the use of an intercept $\beta_0$, which in turn will simplify our exposition. We set the label mismatch cost to $\kappa_\mathrm{y} = 1$, and we use a non-smooth Hinge loss function. We generate random datasets comprising $N = 10$ samples, and we compare the following three formulations:
\begin{enumerate}
    \item[\emph{(i)}] \emph{Empirical risk model.} This model replaces the worst-case expectation in the Wasserstein learning problem~\eqref{eq:the_mother_of_all_problems} with an expectation over the empirical distribution $\widehat{\mP}_N$.
    \item[\emph{(ii)}] \emph{Mixed-feature Wasserstein model.} We solve the Wasserstein learning problem~\eqref{eq-dro-gen2} with the discrete support $\mathbb{Z} = \{ -1, +1 \}$.
    \item[\emph{(iii)}] \emph{Unbounded continuous-feature Wasserstein model.} We solve the Wasserstein learning problem~\eqref{eq-dro-gen2} with the unbounded continuous support $\mathbb{Z} = \mathbb{R}$.
\end{enumerate}
For the fixed hypothesis $\beta_\mz = 1$ and across 10,000 randomly generated datasets, the mean expected Hinge loss of the empirical risk model  evaluates to 0.41. The worst-case expected Hinge loss of the mixed-feature Wasserstein learning problem varies from 0.41 to 2; the largest worst-case expected loss is attained for large Wasserstein radii $\epsilon$, where the worst-case distributions approach {any mixture of $\delta_{(+1, -1)}$ and $\delta_{(-1, +1)}$}. The worst-case expected Hinge loss of the unbounded continuous-feature Wasserstein learning problem, on the other hand, diverges as the Wasserstein radii $\epsilon$ increase. In fact, for sufficiently large $\epsilon$ the worst-case distributions place a probability mass of $0.1$ on a non-sensical atom $(\pm c, \pm1)$, where $c$ increases with $\epsilon$.

\begin{figure}[tb]
    \centering
    \includegraphics[width=.75\textwidth]{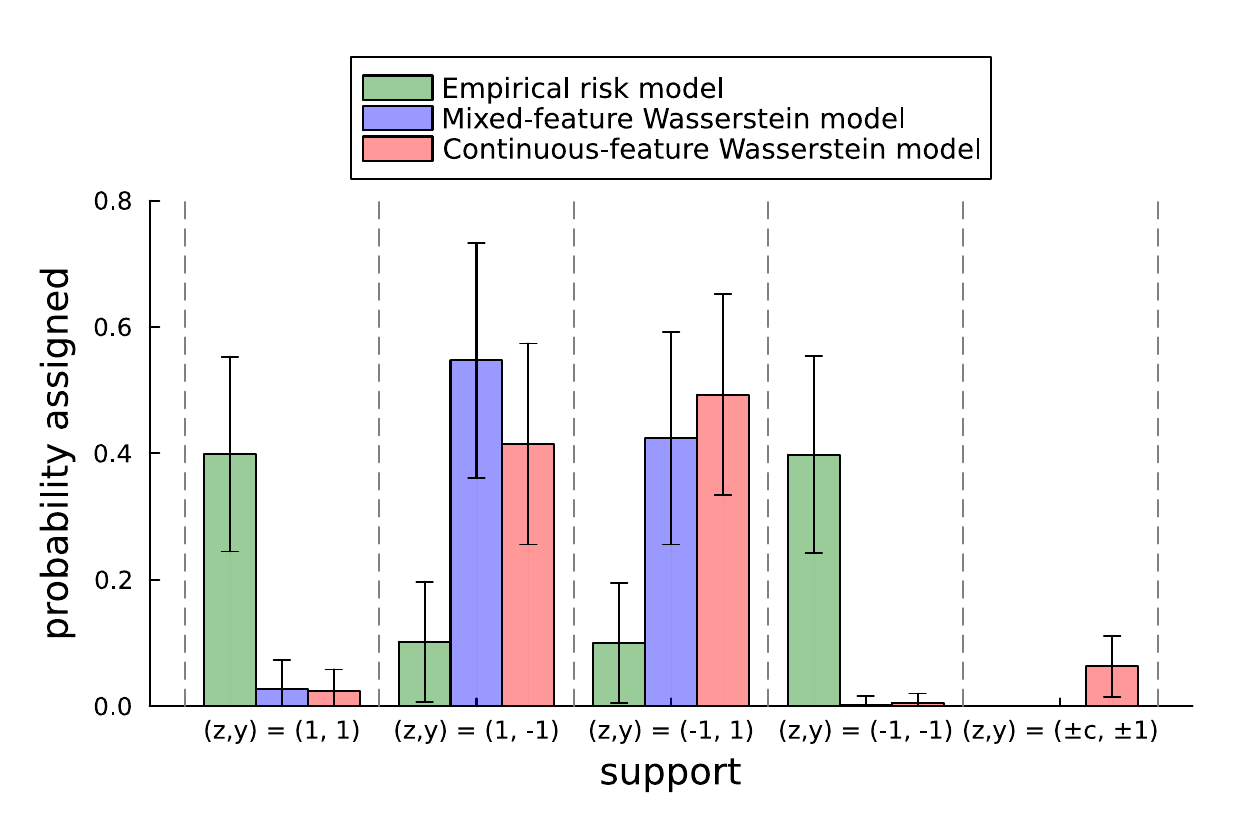} 
    \caption{\color{black} (Worst-case) distributions for our three formulations. The bars represent the mean probabilities, and the whiskers indicate the corresponding standard deviations, computed over 10,000 statistically independent simulations. \label{figure:toy1}}
\end{figure}

Figure~\ref{figure:toy1} illustrates the (worst-case) distributions of the three formulations for an intermediate Wasserstein radius of $\epsilon = 0.85$. As can be seen from the figure, the bounded continuous-feature Wasserstein learning problem places a mean probability mass of 0.063 (std.~dev.~0.048) on an atom $(\pm c, \pm 1)$, where $c$ is instance-specific and where $\lvert c \rvert$ lies in $[1.5, 7.5]$ (mean $2.29$).

One might argue that an illogical worst-case distribution is less concerning in practical applications, where the primary role of Wasserstein learning may be viewed as providing a regularizing effect through the worst-case formulation. However, our numerical results will demonstrate that the unbounded continuous-feature formulation is consistently outperformed by our mixed-feature formulation in terms of generalization error, even when the Wasserstein radii $\epsilon$ are chosen via cross-validation.}

{\color{black}
\subsection{Equivalence to Bounded Continuous-Feature Formulation}\label{section:bounded-feature}

It is natural to ask whether the mixed-feature Wasserstein learning problem~\eqref{eq:the_mother_of_all_problems} is equivalent to its \emph{bounded} continuous-feature counterpart that replaces the support $\mathbb{Z}$ of the discrete features with its convex hull, $\mathrm{conv} (\mathbb{Z})$. We show in the following that this is indeed the case, but we argue that this equivalence bears little computational consequence.

To construct the bounded continuous-feature model, fix any $s \geq 2$ and define $\mathbb{U} (s)$ as the convex hull of the one-hot discrete feature encoding $\mathbb{Z} (s)$. One readily observes that
\begin{equation*}
    \mathbb{U} (s)
    \;\; = \;\;
    \left\{\bm{u} \in [0,1]^{s - 1} : \sum_{i \in [s - 1]} u_i \leq 1 \right\}.
\end{equation*}
Indeed, we have $\mathbb{Z} (s) \subseteq \mathbb{U} (s)$ by construction. To see that $\mathbb{U} (s) \subseteq \mathrm{conv} (\mathbb{Z} (s))$, on the other hand, we note that any $\bm{u} \in \mathbb{U} (s)$ can be represented as $\bm{u} = \lambda_0 \cdot \bm{0} + \sum_{i \in [s-1]} \lambda_i \cdot \mathbf{e}_i$, where $\bm{0} \in \mathbb{Z} (s)$ is the vector of all zeros and where $\mathbf{e}_i \in \mathbb{Z} (s)$ is the $i$-th canonic basis vector in $\mathbb{R}^{s-1}$, with $\lambda_i = u_i$ and $\lambda_0 = 1 - \sum_{i \in [s-1]} \lambda_i$. Defining $\mathbb{U}$ as the convex hull of $\mathbb{Z}$, we observe that
\begin{equation*}
    \mathbb{U}
    \;\; = \;\;
    \mathrm{conv}(\mathbb{Z})
    \;\; = \;\;
    \bigtimes_{m \in [K]} \mathrm{conv} (\mathbb{Z}(k_m))
    \;\; = \;\;
    \bigtimes_{m \in [K]} \mathrm{conv} (\mathbb{U}(k_m)),
\end{equation*}
where the second identity follows from Exercise~1.19 of \cite{bertsekas2009convex}. Since the ground metric $d_\mz$ from Definition~\ref{def:wasserstein} is only suitable for pairs of discrete feature vectors, we replace it with the ground metric
\begin{equation*}
    d_\mathrm{u} (\bm{u}, \bm{u}') = \left(\sum_{m\in[K]} \frac{1}{2}\lVert \bm{u}_m - \bm{u}'_m \rVert_1 + \frac{1}{2} \lvert \mathbf{1}^\top(\bm{u}_m - \bm{u}'_m)\rvert \right)^{1/p}
\end{equation*}
that extends to pairs of continuous feature vectors. One readily verifies that $d_\mathrm{u}$ is a metric, that is, it satisfies the identity of indiscernibles, positivity, symmetry and the triangle inequality. The proof of Proposition~\ref{prop:concave} below will also show that $d_\mathrm{z}$ and $d_\mathrm{u}$ coincide over all pairs of discrete features. We finally define the bounded continuous-feature model as
\begin{equation}\tag{\ref*{eq-dro-gen2}'}\label{eq-dro-gen2-v2}
        \begin{array}{l@{\quad}l}
        \displaystyle \mathop{\mathrm{minimize}}_{\bbeta, \lambda, \bs} & \displaystyle \lambda \epsilon + \frac{1}{N} \sum_{n \in [N]} s_n \\[5mm]
        \displaystyle \mathrm{subject\;to} & \displaystyle \sup_{(\bx, y) \in \mX \times \mY} \left\{ l_{\bbeta} (\bx, \bm{u}, y) - \lambda \lVert \bx - \bx^n \rVert  - \lambda \kappa_\my d_\my (y, y^n) \right\} - \lambda \kappa_\mz d_\mathrm{u} (\bm{u}, \bz^n) \leq s_n \\[-2mm]
        & \displaystyle \mspace{415mu} \forall n \in [N], \; \forall \bm{u} \in \mathbb{U} \\[2mm]
        & \displaystyle \bbeta = (\beta_0, \bbeta_\mx, \bbeta_\mz) \in \mR^{1 + M_\mx + M_\mz}, \;\; \lambda \in \mathbb{R}_+, \;\; \bs \in \mR^N_+,
        \end{array}
\end{equation}
where $\bm{u}$ and $\mathbb{U}$ serve as the continuous representations of the discrete features $\bm{z}$ and their support $\mathbb{Z}$, respectively. Note that in this formulation, the features $\bm{u} \in \mathbb{U}$ can be readily absorbed in the continuous feature set $\bm{x} \in \mathbb{X}$; we keep the distinction for improved clarity.

We next state the main result of this section.

\begin{prop}\label{prop:concave}
    The mixed-feature Wasserstein learning problem~\eqref{eq-dro-gen2} and its corresponding bounded continuous-feature counterpart~\eqref{eq-dro-gen2-v2} share the same optimal value and the same set of optimal solutions.
\end{prop}

The proof of Proposition~\ref{prop:concave} crucially relies on the concavity of the family of mappings $\bm{u} \mapsto d_\mathrm{u}(\bm{u}, \bm{u}')$, $\bm{u}' \in \mathbb{Z}$. Indeed, the equivalence stated in the proposition would \emph{not} hold if we were to replace $d_\mathrm{u}$ with $d'_\mathrm{u}(\bm{u}, \bm{u}') = \left(\sum_{m\in[K]} \lVert \bm{u}_m - \bm{u}'_m \rVert_\infty \right)^{1/p}$, despite $d'_\mathrm{u}$ being a metric on $\mathbb{U}$.

Unfortunately, the equivalence of Proposition~\ref{prop:concave} does not lead to more efficient solution schemes for the mixed-feature Wasserstein learning problem~\eqref{eq-dro-gen2}. Indeed, we are not aware of any finite-dimensional convex reformulations of the bounded continuous-feature counterpart~\eqref{eq-dro-gen2-v2} when the loss function $L$ is convex and Lipschitz continuous. When $L$ is piece-wise affine and convex, on the other hand, problem~\eqref{eq-dro-gen2-v2} indeed admits a finite-dimensional convex reformulation of polynomial size. We will see in our numerical experiments, however, that our cutting plane scheme from Section~\ref{sec-cutting} is much faster than the monolithic solution of this reformulation. This is due to the fact that the reformulation comprises a large number of decision variables as well as a constraint set that lacks desirable sparsity patterns.}

\section{Numerical Results}\label{sec-numerics}

We compare empirically the performance of our mixed-feature Wasserstein learning problem~\eqref{eq:the_mother_of_all_problems} with classical and regularized learning methods as well as continuous-feature {\color{black} formulations of} problem~\eqref{eq:the_mother_of_all_problems}. To this end, Section~\ref{sec-numerics:features} compares the out-of-sample losses of our mixed-feature Wasserstein learning problem~\eqref{eq:the_mother_of_all_problems} with those of the {\color{black} unbounded} continuous-feature approximation when the Wasserstein radius is selected via cross-validation. {\color{black} Section~\ref{sec-numerics:new} compares the runtime of our cutting plane solution scheme for the mixed-feature Wasserstein learning problem with that of solving the equivalent {\color{black} bounded} continuous-feature formulation monolithically.} Finally, Section~\ref{sec-numerics:uci} compares the out-of-sample performance of our formulation~\eqref{eq:the_mother_of_all_problems} with that of alternative methods on standard benchmark instances.

All algorithms were implemented in Julia v1.9.2 using the JuMP package~{\color{black} \citep{Lubin2023}} and MOSEK v10.0, and all experiments were run on Intel Xeon 2.66GHz cluster nodes with $8$GB memory in single-core and single-thread mode (unless otherwise specified). All implementations, datasets and experimental results are available on the \mbox{GitHub repository accompanying this work.}

\subsection{Comparison with {\color{black} Unbounded} Continuous-Feature Formulation}\label{sec-numerics:features}

Section~\ref{subsec-comparison} demonstrated that modeling discrete features in the Wasserstein learning  problem~\eqref{eq:the_mother_of_all_problems} as continuous and subsequently solving {\color{black} an unbounded} continuous-feature formulation may inadvertently hedge against pathological worst-case distributions.

\begin{figure}[tb]
    \centering
    \includegraphics[scale = 0.35]{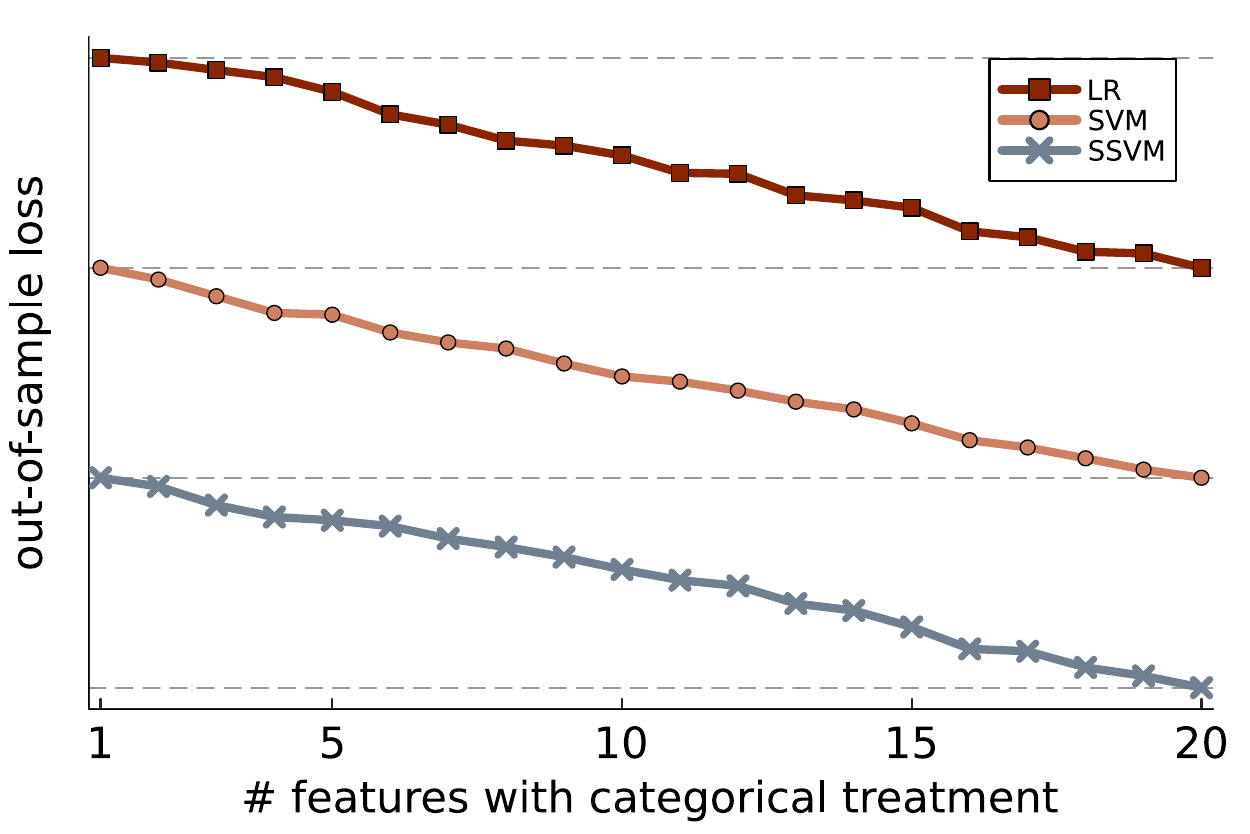} \quad
    \includegraphics[scale = 0.35]{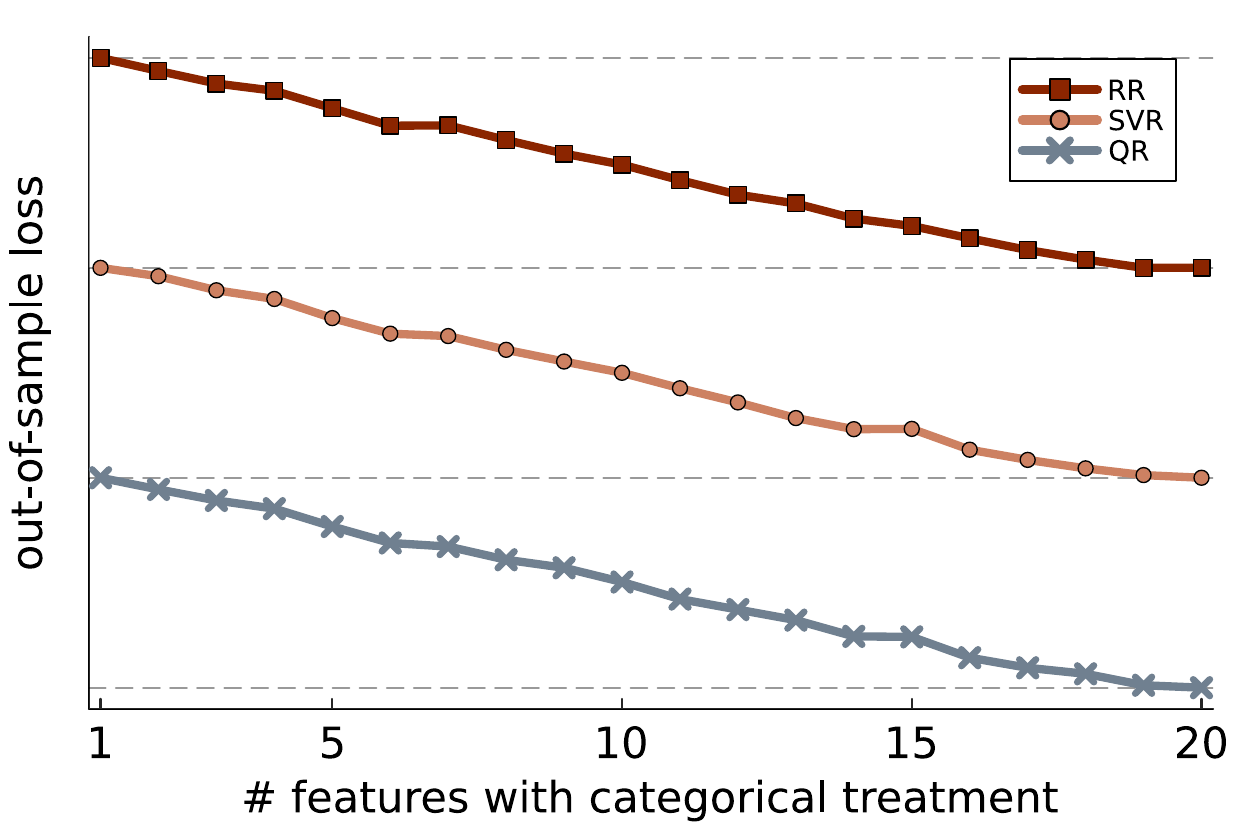}
    \caption{Mean out-of-sample losses for various classification (left) and regression (right) tasks when the number of discrete features that are treated as such varies. All losses are scaled to $[0,1]$ and shifted so that the curves do not overlap.}
    \label{fig:treatment}
\end{figure}

We now explore whether the findings from Section~\ref{subsec-comparison} {\color{black} have implications on the out-of-sample performance of the mixed-feature and unbounded continuous-feature formulations} when the Wasserstein radius $\epsilon$ is selected via cross-validation. We generate synthetic problem instances to be in full control of the problem dimensions. In particular, we generate $100$ synthetic instances for each of the $6$ loss functions considered {\color{black} in Appendices~A and~B: log-loss (LR), non-smooth Hinge Loss (SVM), smooth Hinge loss (SSVM), Huber loss (RR), $\tau$-insensitive loss (SVR) and pinball loss (QR)}, assuming that the loss functions explain a large part (but not all) of the variability in the data. All instances comprise $N = 20$ data points, no numerical features, and $K = 20$ binary features. The small number of data points ensures that distributional robustness is required to obtain small out-of-sample losses. For each instance, we solve a mixed-feature Wasserstein learning problem that treats some of the features as binary, whereas the other features are treated as {\color{black} continuous and unbounded}. In particular, the extreme cases of modeling all and none of the features as binary correspond to our mixed-feature Wasserstein learning problem~\eqref{eq:the_mother_of_all_problems} and its {\color{black} unbounded} continuous-feature approximation, respectively. We cross-validate the Wasserstein radius from the set $\epsilon \in \{ 0.005, 0.01, 0.015, \ldots, 0.1\}$, and we choose $(\kappa_{\mz}, \kappa_{\my}, p) = (1, 1, 1)$ as well as $\lVert \cdot \rVert = \lVert \cdot \rVert_1$ in our ground metric~\eqref{eq:ground_metric}. We refer to the GitHub repository for further details of the instance generation procedure. Figure~\ref{fig:treatment} reports the average out-of-sample losses across $100$ randomly generated problem instances for the different loss functions. The figure reveals an overall trend of improved results when the number of binary features that are treated as such increases. Qualitatively, we observe that this conclusion is robust to different choices of the $\epsilon$-grid used for cross-validation; we refer to the GitHub repository for further details.

{\color{black}
\subsection{Comparison with Bounded Continuous-Feature Formulation}\label{sec-numerics:new}

In Section~\ref{section:bounded-feature}, we showed that our mixed-feature Wasserstein learning problem~\eqref{eq:the_mother_of_all_problems} is equivalent to a bounded continuous-feature model that replaces the support $\mathbb{Z}$ of the discrete features with its convex hull $\mathrm{conv} (\mathbb{Z})$. In this section, we compare the runtimes of solving the mixed-feature formulation via our cutting plane method with those of solving the bounded continuous-feature formulation via an off-the-shelf solver.

When the loss function $L$ in the mixed-feature Wasserstein learning problem~\eqref{eq:the_mother_of_all_problems} is piece-wise affine and convex, the equivalent bounded continuous-feature model admits a reformulation as a polynomial-size convex optimization problem. However, solving this reformulation is computationally expensive since it involves $\mathcal{O}(N \cdot J \cdot K)$ decision variables, where $N$, $J$ and $K$ denote the numbers of data points, affine pieces in the loss function and discrete features, respectively. Moreover, the reformulation lacks desirable sparsity features. The consequences are illustrated in Table~\ref{tab:runtime_Aras_piece-wise}, which compares the monolithic solution of the bounded continuous-feature reformulation with the solution of the mixed-feature model using our cutting plane approach. The table reports median runtimes over $100$ datasets with $N \in \{500, \, 1{,}000, \, 2{,}500\}$ data points. All instances contain $30$ features, out of which $25\%, \ 50\%, \ 75\%$ are discrete with $4$ possible values, whereas the rest are continuous. We set the Wasserstein radius $\epsilon = 10^{-2}$ and use $(\kappa_\mathrm{z}, \kappa_\mathrm{y}, p) = (1, K, 1)$ for our ground metric. Further details of the implementation are available on the GitHub repository. We observe that our cutting plane scheme is significantly faster than the solution of the equivalent bounded continuous-feature model.

\begin{table}[!h]
    \centering
    \resizebox{0.8\columnwidth}{!}{%
    \begin{tabular}{l|ccc|ccc|ccc}
        \hline
        & & & & & & & & & \\[-4mm]
        & \multicolumn{3}{c|}{\textbf{\underline{N = 500}}} & \multicolumn{3}{c|}{\textbf{\underline{N = 1,000}}} & \multicolumn{3}{c}{\textbf{\underline{N = 2,500}}} \\
        \textbf{Method} & $\% 25$ & $\%50$ & $\%75$ & $\% 25$ & $\%50$ & $\%75$ & $\% 25$ & $\%50$ & $\%75$  \\
        \hline
        & & & & & & & & & \\[-2.8mm]
        SVM-cut & 0.05 & 0.20 & 0.25 & 0.13 & 0.21 & 0.45 & 0.29 & 0.61 & 1.23 \\
        {\color{black}SVM-piece}  & 6.35 & 11.80 & 15.25 & 12.48 & 22.99 & 40.97 & 30.4 & 109.78 & 146.62 \\[1mm] \hline
        & & & & & & & & & \\[-2.8mm]
        QR-cut & 0.15 & 0.24 & 0.36 & 0.28 & 0.57 & 1.18 & 1.08 & 1.64 & 3.08 \\
        {\color{black}QR-piece}  & 3.35 & 5.59 & 9.24 & 8.71 & 19.95 & 26.1 & 22.83 & 48.34 & 114.19 \\[1mm] \hline
        & & & & & & & & & \\[-2.8mm]
        SVR-cut & 0.12 & 0.22 & 0.24 & 0.26 & 0.44 & 0.55 & 0.76 & 0.92 & 2.73 \\
        {\color{black}SVR-piece} & 4.35 & 9.20 & 17.29 & 10.2 & 16.78 & 35.46 & 27.13 & 52.02 &  95.08\\[1mm] \hline
    \end{tabular}
    }
    \caption{\color{black} Median runtimes in secs to solve the mixed-feature model (-cut; using our cutting plane method) and the bounded continuous-feature model (-piece; solving a polynomial-size convex reformulation) for the hinge (SVM), pinball (QR; with $\tau = 0.5$) and $\tau$-insensitive (SVR; with $\tau = 10^{-2}$) loss functions.}
    \label{tab:runtime_Aras_piece-wise}
\end{table}

Although no polynomial-size convex reformulations are known for the equivalent bounded continuous-feature model when the loss function $L$ in the mixed-feature Wasserstein learning problem~\eqref{eq:the_mother_of_all_problems} is Lipschitz continuous, we can approximate such loss functions to any desired accuracy via piece-wise affine convex functions. To this end, we take the pointwise maximum of tangent lines at equally spaced breakpoints and subsequently solve the bounded continuous-feature formulation for piece-wise affine convex loss functions. This provides a lower bound on the original problem, whose value approaches the optimal value of the original problem as the number of affine pieces in the approximated loss function increases. Table~\ref{tab:rebuttal_piece} reports the median runtimes for $100$ datasets; we use the same problem parameters as in our previous experiment, and we report results for loss function approximations with $J \in \{5,10,25\}$ affine pieces. Similar to our previous experiments, the table shows that our cutting plane scheme is significantly faster than the solution of the equivalent bounded continuous-feature model. We note that the approximated loss functions would require more than $25$ pieces to guarantee a suboptimality of $0.1\%$ or less; details are relegated to the GitHub repository.

\begin{table}[!h]
    \centering
    \resizebox{1.\columnwidth}{!}{%
    \begin{tabular}{l|ccc|ccc|ccc}
        \hline
        & & & & & & & & & \\[-4mm]
        & \multicolumn{3}{c|}{\textbf{\underline{N = 500}}} & \multicolumn{3}{c|}{\textbf{\underline{N = 1,000}}} & \multicolumn{3}{c}{\textbf{\underline{N = 2,500}}} \\
        \textbf{Method} & $\% 25$ & $\%50$ & $\%75$ & $\% 25$ & $\%50$ & $\%75$ & $\% 25$ & $\%50$ & $\%75$  \\
        \hline
        & & & & & & & & & \\[-2.8mm]
        LR-cut & 0.16& 0.61& 1.32& 0.22& 0.60 & 1.48& 0.56& 0.67& 1.87 \\
        {\color{black}LR-piece [5]} & 42.08& 108.0& 157.06& 127.97& 382.6& 608.76& 743.01& 2,371.29*& \textbf{NaN} \\
        {\color{black}LR-piece [10]} & 50.32& 133.0& 186.56& 482.51& 236.94*& 339.82*& 3,107.19*& \textbf{NaN}& 2,595.02* \\
        {\color{black}LR-piece [25]} & 526.72& 280.11& 499.72& 2,178.79*& 1,634.96*& 1,137.18*& \textbf{NaN}& \textbf{NaN}& \textbf{NaN} \\[1mm] \hline
        & & & & & & & & & \\[-2.8mm]
        SSVM-cut & 0.14& 0.31& 0.54& 0.22& 0.46& 0.63& 0.6& 0.61& 1.57 \\
        {\color{black}SSVM-piece [5]} & 20.54& 59.6& 84.11& 44.43& 125.33& 201.43& 460.29& 349.74& 612.18* \\
        {\color{black}SSVM-piece [10]} & 34.85& 80.19& 114.63& 82.52& 166.36& 277.17*& 1,924.21*& 439.99*& 730.46*  \\
        {\color{black}SSVM-piece [25]} & 90.25& 198.32*& 332.86*& 221.74*& 406.68& 763.53*& \textbf{NaN}& 1,161.48*& 1,996.73* \\[1mm] \hline
        & & & & & & & & & \\[-2.8mm]
        RR-cut  & 0.16& 0.19& 0.3& 0.27& 0.3& 0.52& 0.7& 0.78& 1.31 \\
        {\color{black}RR-piece [5]} & 9.87& 23.23& 38.57& 23.13& 55.95& 87.75& 73.02& 171.16& 269.98  \\
        {\color{black}RR-piece [10]} & 16.52& 38.92& 60.38& 37.69& 87.57& 143.39& 105.4& 281.6& 464.49 \\
        {\color{black}RR-piece [25]} & 43.16& 95.53& 155.26& 93.21& 222.79& 373.44*& 296.18*& 824.93*& 1,584.85*\\[1mm] \hline
    \end{tabular}
    }
    \caption{\color{black} Median runtimes in secs to solve the mixed-feature model (-cut; using our cutting plane method) and the bounded continuous-feature model (-piece [$J$]; solving a polynomial-size convex reformulation) for the logistic (LR), smooth hinge (SSVM) and Huber (RR; with $\delta = 0.05$) loss functions. A \textbf{NaN} indicates that more than half of the instances could not be solved within an hour of runtime. An asterisk (*) indicates that at least one instance could not be solved within an hour of runtime.}
    \label{tab:rebuttal_piece}
\end{table}
}

\subsection{Performance on Benchmark Instances}\label{sec-numerics:uci}

In {\color{black} Section~\ref{sec-numerics:features}}, we observed that the mixed-feature Wasserstein learning problem~\eqref{eq:the_mother_of_all_problems} can outperform its {\color{black} unbounded} continuous-feature approximation in terms of out-of-sample losses. In practice, however, loss functions merely serve as surrogates for the misclassification rate (in classification problems) or mean squared error (in regression problems). Moreover, {\color{black} Section~\ref{sec-numerics:features}} considered synthetically generated problem instances, which {\color{black} may lack} some of the intricate structure of real-life datasets. This section therefore compares the out-of-sample misclassification rates and mean squared errors of problem~\eqref{eq:the_mother_of_all_problems} with those of alternative methods on standard benchmark instances. In particular, we selected 8 of the most popular classification and 7 of the most popular regression datasets from the UCI machine learning repository \citep{UCI}. We processed the datasets to \emph{(i)} handle missing values and inconsistencies in the data, \emph{(ii)} employ a one-hot encoding for discrete variables, and \emph{(iii)} convert the output variable (to a binary value for classification tasks and a $[-1, 1]$-interval for regression tasks). All datasets, processing scripts as well as further results on other UCI datasets can be found in the GitHub repository.

Tables~\ref{tab:logitic}--\ref{tab:quant} report averaged results over $100$ random splits into $80\%$ training set and $20\%$ test set for both the original datasets as well as parsimonious variants where only half of the data points are available. In the tables, the column groups report (from left to right) the problem instances' names and dimensions; the results for the unregularized implementations of the nominal problem as well as the mixed-feature Wasserstein learning problem~\eqref{eq:the_mother_of_all_problems} using $(\kappa_{\mz}, p) = (1, 1)$ and $\lVert \cdot \rVert = \lVert \cdot \rVert_1$ in our ground metric~\eqref{eq:ground_metric}; the results for the {\color{black}corresponding} $l_2$-regularized versions; and the results for the {\color{black} unbounded} continuous-feature approximations of problem~\eqref{eq:the_mother_of_all_problems}. For the unregularized mixed-feature and continuous-feature Wasserstein learning problems, we {jointly cross-validate the hyperparameters $(\epsilon, \kappa_{\my})$ where $\epsilon \in \{ 0, 10^{-5}, 10^{-3}, 10^{-1}\}$ and {\color{black}$\kappa_{\my} \in \{ 1, K, \infty \}$}. For the regularized nominal model, we cross-validate the regularization penalty $\alpha$ from the set $\{ 0 \} \cup \{c \cdot 10^{-p} : c \in \{1,5\}, \ p = 1, \ldots, 6\}$. For the regularized mixed-feature Wasserstein learning problem, finally, we cross-validate the hyperparameters {\color{black}$(\epsilon, \kappa_{\my}, \alpha)$ from the Cartesian product of the previous three sets}. The method with the smallest error in each column group is highlighted in bold, and the method with the smallest error across all column groups has a grey background. For instances where a version of the mixed-feature Wasserstein learning problem attains the smallest error across all column groups, the symbols $\dag$ and $\ddag$ indicate a statistically significant improvement of that model over the nominal and regularized nominal learning problem {\color{black}as well as over the continuous-feature approximation, respectively, at a $p$-value of 0.05}. Most nominal models were solved within seconds, most continuous-feature models were solved within minutes, and most of the mixed-feature models were solved within 10 minutes. We relegate the details of the statistical significance tests and runtimes to the GitHub repository.

Overall, we observe from the tables that the mixed-feature Wasserstein learning problems outperform both the (regularized) nominal problems and the {\color{black} unbounded} continuous-feature approximations in the classification and regression tasks. The outperformance of the mixed-feature Wasserstein learning problems over the continuous-feature approximations tends to be more substantial for problem instances with many discrete features, which further confirms our findings from Sections~\ref{subsec-comparison} and~\ref{sec-numerics:features}. Also, the mixed-feature Wasserstein learning problems tend to perform better on the parsimonious versions of the problem instances. This is intuitive as we expect robust optimization to be particularly effective in data sparse environments. Regularizing the mixed-feature Wasserstein learning problem does not yield significant advantages in the classification tasks, but it does help in the regression tasks. Finally, while we do not observe any significant  differences among the classification loss functions, the Huber loss function performs best on the regression problem instances.

{\color{black}
\section{Conclusions}\label{sec-conclusions}

Wasserstein learning offers a principled approach to mitigating overfitting in supervised learning. While Wasserstein learning has been studied extensively for datasets with exclusively continuous features, the existing solution techniques cannot handle discrete features faithfully without sacrificing performance in real-world applications.

In this paper, we propose a cutting plane algorithm for Wasserstein classification and regression problems that involve both continuous and discrete features. While the worst-case complexity of our method is exponential, it performs well on both synthetic and benchmark datasets. The core of our algorithm is an efficient identification of the most violated constraints with respect to the current solution. Despite the discrete nature of the problem, we show that it can be solved in polynomial time. We also study the complexity of mixed-feature Wasserstein learning more broadly, and we investigate its connections to regularized as well as different continuous-feature formulations.

Our work opens several promising directions for future research. Most notably, it would be instructive to explore whether similar algorithmic techniques can be developed for other supervised learning models, including decision trees, ensemble methods, and neural networks. Furthermore, it would be valuable to investigate the development of polynomial-time solution schemes that avoid reliance on the ellipsoid method and that overcome the scalability limitations of monolithic reformulations. For example, the fact that our algorithm for the mixed-feature problem relies on a sorting argument to identify the most violated constraint could potentially pave the way to developing robust counterparts using standard convex reformulations of sorting-type constraints. 

\paragraph{Acknowledgments}
The authors gratefully acknowledge the detailed and constructive feedback provided by the department editor, the anonymous associate editor, and the two anonymous referees, which has substantially improved the quality of this manuscript. We are especially indebted to the reviewers who suggested the possibility that our mixed-feature formulation might be equivalent to a bounded continuous-feature model, which directly inspired Section~5.2 of the manuscript. The authors are also indebted to Oscar Dowson for his invaluable invaluable guidance on efficient implementations of our Julia source codes.}

\clearpage
\begin{table}[tb]
{\color{black}
    \centering
    \resizebox{0.85\columnwidth}{!}{%
     \begin{tabular}{lcccc|c|cc|cc|c}
\toprule
Dataset &  $N$ & $M_\mx$ & $M_\mz$ & $K$ & Reduced? & Nom  &  MixF & r-Nom  & r-MixF & ConF \\
\midrule
     \multirow{2}*{\underline{balance-scale}} & \multirow{2}*{625} & \multirow{2}*{0} & \multirow{2}*{16} & \multirow{2}*{4} & \xmark & 0.56\% & \cg{\textbf{0.44\% \dag}}  & 0.52\% & 0.48\% & \cg{0.44\%}    \\
     & & & & &  \cmark & 1.65\% & \cg{\textbf{1.47\% \dag \ddag}} & 1.92\% & 1.77\% & 1.60\% \\
     \hline
     & & & & & & & & & & \\[-2.8mm]
     \multirow{2}*{\underline{breast-cancer}} & \multirow{2}*{277} & \multirow{2}*{0} & \multirow{2}*{42} & \multirow{2}*{9} & \xmark & 30.10\% & \cg{\textbf{28.46\% \dag \ddag}} & 29.46\% & \textbf{29.36\%} & 29.45\%  \\ 
     & & & & &  \cmark & 31.87\% & \textbf{29.88\%} & 29.58\% & \cg\textbf{28.31\% \dag \ddag} & 30.48\% \\
     \hline
     & & & & & & & & & & \\[-2.8mm]
     \multirow{2}*{\underline{credit-approval}} & \multirow{2}*{690} & \multirow{2}*{6} & \multirow{2}*{36} & \multirow{2}*{9} & \xmark & 13.59\% & \cg\textbf{13.12\% \dag \ddag} & 13.70\% & \textbf{13.55\%} & 13.88\% \\
     & & & & &  \cmark &  16.30\% & \textbf{14.94\%} & \cg\textbf{14.58\%} & 14.94\% & 15.11\% \\
     \hline
     & & & & & & & & & & \\[-2.8mm]
     \multirow{2}*{\underline{cylinder-bands}} & \multirow{2}*{539} & \multirow{2}*{19} & \multirow{2}*{43} & \multirow{2}*{14} & \xmark & 22.85\% & \textbf{22.71\%} & 22.90\% & \cg{\textbf{22.62\% \ddag}} & 23.13\% \\
     & & & & &  \cmark & 27.37\% & \textbf{26.36\%} & 27.18\% & \cg{\textbf{26.02\% \dag}} & 26.41\% \\
     \hline      
     & & & & & & & & & & \\[-2.8mm]
     \multirow{2}*{\underline{lymphography}} & \multirow{2}*{148} & \multirow{2}*{0} & \multirow{2}*{42} & \multirow{2}*{18} & \xmark & 17.24\% & \cg{\textbf{14.66\% \dag \ddag}} & \textbf{15.86\%} & 16.55\% &  17.41\% \\
     & & & & &  \cmark & 23.98\% & \cg\textbf{17.73\% \ddag} & 17.96\% & \textbf{18.07\%} &  19.14\% \\
     \hline      
     
     & & & & & & & & & & \\[-2.8mm]
     \multirow{2}*{\underline{primacy}} & \multirow{2}*{339} & \multirow{2}*{0} & \multirow{2}*{25} & \multirow{2}*{17} & \xmark &  \cg\textbf{13.51\%} & 13.96\% & 14.25\% & \textbf{13.58\%} & 14.33\% \\
     & & & & &  \cmark & 15.89\% & \textbf{14.70\%} & 14.70\% & \cg{\textbf{14.48\% \dag}} & 14.68\% \\
     \hline      
     & & & & & & & & & & \\[-2.8mm]
     \multirow{2}*{\underline{spect}} & \multirow{2}*{267} & \multirow{2}*{0} & \multirow{2}*{22} & \multirow{2}*{22} & \xmark & 20.28\% & \textbf{18.40\%} & 20.09\% & \textbf{18.30\%} & \cg{17.93\%} \\
     & & & & &  \cmark & 23.22\% & \textbf{19.44\%} & 23.34\% & \cg\textbf{19.13\% \dag \ddag} & 19.66\% \\
     \hline      
     & & & & & & & & & & \\[-2.8mm]
     \multirow{2}*{\underline{tic-tac-toe}} & \multirow{2}*{958} & \multirow{2}*{0} & \multirow{2}*{18} & \multirow{2}*{9} & \xmark & 1.94\% & \cg\textbf{1.70\%} & \cg\textbf{1.70\%} & \cg\textbf{1.70\%} & \cg\textbf{1.70\%} \\
     & & & & &  \cmark & 2.75\% & \cg\textbf{1.65\% \dag} & 1.81\% & \textbf{1.68\%} & \cg1.65\% \\
\bottomrule
\end{tabular}} \medskip
    \caption{\color{black}{Mean classification error for the \textbf{logistic regression loss function}. $N$, $M_\mx$, $M_\mz$ and $K$ refer to the problem parameters from Section~\ref{sec-wasser}. The column `Reduced?' indicates whether the full or sparse version of the dataset is being considered. Nom, MixF and ConF refer to the nominal problem formulation, the mixed-feature Wasserstein learning problem~\eqref{eq:the_mother_of_all_problems} and its unbounded continuous-feature approximation, respectively. The prefix `r-' indicates the use of an $l_2$-regularization.}}
    \label{tab:logitic} \vspace{-0.5cm}
}
\end{table}

\begin{table}[tb]
{\color{black}
    \centering
    \resizebox{0.85\columnwidth}{!}{%
     \begin{tabular}{lcccc|c|cc|cc|c}
\toprule
Dataset &  $N$ & $M_\mx$ & $M_\mz$ & $K$ & Reduced? & Nom  &  MixF & r-Nom  & r-MixF & ConF \\
\midrule
     \multirow{2}*{\underline{balance-scale}} & \multirow{2}*{625} & \multirow{2}*{0} & \multirow{2}*{16} & \multirow{2}*{4} & \xmark & \cg\textbf{0.44\%} & \cg{\textbf{0.44\% \ddag}} & 0.56\% & \textbf{0.48\%} & 0.56\%   \\
     & & & & &  \cmark & \textbf{1.91\%} & 1.93\% & \textbf{2.29\%} & 2.31\% & \cg{1.81\%}  \\
     \hline
     & & & & & & & & & & \\[-2.8mm]
     \multirow{2}*{\underline{breast-cancer}} & \multirow{2}*{277} & \multirow{2}*{0} & \multirow{2}*{42} & \multirow{2}*{9} & \xmark & 30.27\% & \textbf{29.27\%} & 30.46\% & \cg{\textbf{28.36\% \dag \ddag}} & 31.27\%   \\
     & & & & &  \cmark & 31.66\% & \textbf{30.24\%} & 31.05\% & \cg{\textbf{28.98\% \dag \ddag}} & 30.57\% \\
     \hline
     & & & & & & & & & & \\[-2.8mm]
     \multirow{2}*{\underline{credit-approval}} & \multirow{2}*{690} & \multirow{2}*{6} & \multirow{2}*{36} & \multirow{2}*{9} & \xmark & 13.84\% & \textbf{13.59\%} & \textbf{14.06\%} & \cg{\textbf{14.06\%}} & 14.20\% \\
     & & & & &  \cmark & 15.86\% & \textbf{15.74\%} & \cg{\textbf{14.84\%}} & 15.33\% &  15.15\% \\
     \hline
     & & & & & & & & & & \\[-2.8mm]
     \multirow{2}*{\underline{cylinder-bands}} & \multirow{2}*{539} & \multirow{2}*{19} & \multirow{2}*{43} & \multirow{2}*{14} & \xmark & 23.22\% & \textbf{22.71\%} & 23.08\% & \cg\textbf{22.48\% \dag \ddag} & 23.60\%  \\
     & & & & &  \cmark & 28.47\% & \cg\textbf{27.07\% \dag} & 28.55\% & 27.43\% & 27.10\%  \\
     \hline      
     & & & & & & & & & & \\[-2.8mm]
     \multirow{2}*{\underline{lymphography}} & \multirow{2}*{148} & \multirow{2}*{0} & \multirow{2}*{42} & \multirow{2}*{18} & \xmark & 19.14\% & \cg\textbf{17.41\% \dag \ddag} & 18.79\% & \textbf{17.93\%} & 18.79\% \\
     & & & & &  \cmark & \textbf{20.28\%} & 21.08\% & 19.49\% & \cg\textbf{18.24\% \dag \ddag} & 22.78\% \\
     \hline      
     & & & & & & & & & & \\[-2.8mm]
     \multirow{2}*{\underline{primacy}} & \multirow{2}*{339} & \multirow{2}*{0} & \multirow{2}*{25} & \multirow{2}*{17} & \xmark & 14.03\% & \textbf{13.81\%} & \cg{\textbf{13.36\%}} & 14.43\% & 13.96\%  \\
     & & & & &  \cmark & 15.96\% & \textbf{14.88\%} & \cg\textbf{14.53\%} & \cg\textbf{14.53\% \ddag} & 14.88\% \\
     \hline      
     & & & & & & & & & & \\[-2.8mm]
     \multirow{2}*{\underline{spect}} & \multirow{2}*{267} & \multirow{2}*{0} & \multirow{2}*{22} & \multirow{2}*{22} & \xmark & 20.19\% & \textbf{19.62\%} & 21.04\% & \cg\textbf{19.15\% \dag \ddag} & 20.76\% \\
     & & & & &  \cmark & 24.16\% & \textbf{21.13\%} & 22.03\% & \textbf{21.22\%} & \cg{21.10\%} \\
     \hline      
     & & & & & & & & & & \\[-2.8mm]
     \multirow{2}*{\underline{tic-tac-toe}} & \multirow{2}*{958} & \multirow{2}*{0} & \multirow{2}*{18} & \multirow{2}*{9} & \xmark & \cg\textbf{1.70\%} & \cg\textbf{1.70\%} & \cg\textbf{1.70\%} & \cg\textbf{1.70\%} & \cg1.70\% \\
     & & & & &  \cmark & 2.58\% & \cg\textbf{1.65\%} & \cg\textbf{1.65\%} & \cg\textbf{1.65\%} & \cg1.65\% \\
\bottomrule
\end{tabular}} \medskip
    \caption{\color{black}Mean classification error for the \textbf{hinge loss function}. We use the same abbreviations as in Table~\ref{tab:logitic}.}
    \label{tab:discrete_results_svm_reduced} \vspace{-0.5cm}
}
\end{table}

\begin{table}[tb]
{\color{black}
    \centering
    \resizebox{0.85\columnwidth}{!}{%
     \begin{tabular}{lcccc|c|cc|cc|c}
\toprule
Dataset &  $N$ & $M_\mx$ & $M_\mz$ & $K$ & Reduced? & Nom  &  MixF & r-Nom  & r-MixF & ConF \\
\midrule
     \multirow{2}*{\underline{balance-scale}} & \multirow{2}*{625} & \multirow{2}*{0} & \multirow{2}*{16} & \multirow{2}*{4} & \xmark & 1.36\% & \cg{\textbf{0.40\% \dag \ddag}} & 1.52\% & \textbf{0.48\%} & 0.52\% \\
     & & & & &  \cmark & 3.15\% & \textbf{1.89\%} & 2.07\% & \textbf{2.01\%} & \cg{\textbf{1.71\%}} \\
     \hline
     & & & & & & & & & & \\[-2.8mm]
     \multirow{2}*{\underline{breast-cancer}} & \multirow{2}*{277} & \multirow{2}*{0} & \multirow{2}*{42} & \multirow{2}*{9} & \xmark & 29.55\% & \cg{\textbf{29.09\% \dag \ddag}} & 30.18\% & \textbf{30.00\%} & 30.18\%  \\
     & & & & &  \cmark & 31.30\% & \textbf{30.00\%} & 28.55\% & \cg{\textbf{28.19\% \dag \ddag}} & 31.08\% \\
     \hline
     & & & & & & & & & & \\[-2.8mm]
     \multirow{2}*{\underline{credit-approval}} & \multirow{2}*{690} & \multirow{2}*{6} & \multirow{2}*{36} & \multirow{2}*{9} & \xmark &  13.44\% & \textbf{13.37\%} & 13.44\% & \cg{\textbf{13.33\% \ddag}} & 14.42\%  \\
     & & & & &  \cmark & 15.59\% & \textbf{15.31\%} & 14.65\% & \cg{\textbf{14.43\% \dag \ddag}} & 15.09\% \\
     \hline
     & & & & & & & & & & \\[-2.8mm]
     \multirow{2}*{\underline{cylinder-bands}} & \multirow{2}*{539} & \multirow{2}*{19} & \multirow{2}*{43} & \multirow{2}*{14} & \xmark & 23.27\% & \cg\textbf{22.06\% \dag \ddag} & 23.04\% & \textbf{22.66\%} & 22.62\%  \\
     & & & & &  \cmark &  27.43\% & \cg{\textbf{26.49\% \dag}} & 27.09\% & \textbf{26.75\%} & 26.73\% \\
     \hline      
     & & & & & & & & & & \\[-2.8mm]
     \multirow{2}*{\underline{lymphography}} & \multirow{2}*{148} & \multirow{2}*{0} & \multirow{2}*{42} & \multirow{2}*{18} & \xmark & 19.31\% & \textbf{16.90\%} & \cg{\textbf{16.38\%}} & \cg{\textbf{16.38\% \ddag}} & 17.76\% \\
     & & & & &  \cmark & 21.31\% & \textbf{20.11\%} & 18.47\% & \cg{\textbf{18.07\% \ddag}} & 20.17\% \\
     \hline      
     & & & & & & & & & & \\[-2.8mm]
     \multirow{2}*{\underline{primacy}} & \multirow{2}*{339} & \multirow{2}*{0} & \multirow{2}*{25} & \multirow{2}*{17} & \xmark &  13.51\% & \textbf{13.43\%} & \cg\textbf{13.13\%} & 13.66\% & 14.03\% \\
     & & & & &  \cmark & 15.10\% & \textbf{14.78\%} & 14.38\% & \cg\textbf{14.34\%} & 14.41\% \\
     \hline      
     & & & & & & & & & & \\[-2.8mm]
     \multirow{2}*{\underline{spect}} & \multirow{2}*{267} & \multirow{2}*{0} & \multirow{2}*{22} & \multirow{2}*{22} & \xmark & 20.28\% & \textbf{18.30\%} & 19.25\% & \cg\textbf{16.79\% \dag \ddag} & 18.11\% \\
     & & & & &  \cmark & 23.16\% & \cg\textbf{19.91\% \dag} & 23.16\% & \textbf{20.16\%} & 20.03\% \\
     \hline      
     & & & & & & & & & & \\[-2.8mm]
     \multirow{2}*{\underline{tic-tac-toe}} & \multirow{2}*{958} & \multirow{2}*{0} & \multirow{2}*{18} & \multirow{2}*{9} & \xmark & 1.73\% & \cg\textbf{1.70\%} & \cg\textbf{1.70\%} & \cg\textbf{1.70\%} & \cg1.70\% \\
      & & & & &  \cmark & 2.77\% & \cg\textbf{1.65\%} & \cg\textbf{1.65\%} & \cg\textbf{1.65\%} & \cg1.65\% \\
\bottomrule
\end{tabular}} \medskip
    \caption{\color{black}Mean classification error for the \textbf{smooth hinge loss function}. We use the same abbreviations as in Table~\ref{tab:logitic}.}
    \label{tab:ssvm} \vspace{-0.5cm}
}
\end{table}

\begin{table}[tb]
{\color{black}
    \centering
    \resizebox{0.85\columnwidth}{!}{%
     \begin{tabular}{lcccc|c|cc|cc|c}
\toprule
Dataset &  $N$ & $M_\mx$ & $M_\mz$ & $K$ & Reduced? & Nom  &  MixF & r-Nom  & r-MixF  & ConF \\
\midrule
     \multirow{2}*{\underline{bike}} & \multirow{2}*{17,379} & \multirow{2}*{4} & \multirow{2}*{31} & \multirow{2}*{5} & \xmark & \cg\textbf{210.23} & \cg\textbf{210.23} & \textbf{210.23} & \textbf{210.23} & 210.23 \\
     & & & & &  \cmark & 210.82 & \cg\textbf{210.81} & \textbf{210.84} & \textbf{210.84} & 213.47  \\
    \hline
     & & & & & & & &\\[-2.8mm]
     \multirow{2}*{\underline{fire}} & \multirow{2}*{517} & \multirow{2}*{10} & \multirow{2}*{31} & \multirow{2}*{4} & \xmark & \textbf{123.33} & \textbf{123.33} & 123.48 & \cg{\textbf{122.32 \dag}} & 122.44 \\
     & & & & &  \cmark & 132.71 &  \textbf{109.84} & 114.39 & \textbf{108.40} & \cg108.12 \\
    \hline
     & & & & & & & &\\[-2.8mm]
     \multirow{2}*{\underline{flare}} & \multirow{2}*{1,066} & \multirow{2}*{0} & \multirow{2}*{21} & \multirow{2}*{9} & \xmark & \textbf{198.46} & \textbf{198.46} & 198.96 & \textbf{198.46}  & \cg198.44 \\
     & & & & &  \cmark & 197.85 & \cg\textbf{195.65 \dag \ddag} & 198.56 & \textbf{197.15} & 198.91 \\
    \hline
     & & & & & & & &\\[-2.8mm]
     \multirow{2}*{\underline{garments}} & \multirow{2}*{1,197} & \multirow{2}*{7} & \multirow{2}*{22} & \multirow{2}*{4} & \xmark & \textbf{363.11} & 363.44 & \cg{\textbf{362.55}} & 362.55 & 363.36 \\
     & & & & &  \cmark & 890.81 & \textbf{367.95} & 368.15 & \cg{\textbf{367.12 \dag \ddag}} & 368.90 \\
     \hline
     & & & & & & & &\\[-2.8mm]
     \multirow{2}*{\underline{imports}} & \multirow{2}*{193} & \multirow{2}*{14} & \multirow{2}*{45} & \multirow{2}*{10} & \xmark & 442.71 & \textbf{288.57} & 311.08 & \cg\textbf{269.95 \dag \ddag} & 288.85 \\
     & & & & &  \cmark & 660.04 & \textbf{369.06} & 352.79 &  \cg{\textbf{335.30 \dag \ddag}} & 370.57 \\
     \hline
     & & & & & & & &\\[-2.8mm]
     \multirow{2}*{\underline{student}} & \multirow{2}*{395} & \multirow{2}*{13} & \multirow{2}*{26} & \multirow{2}*{17} & \xmark & \textbf{379.13} & 382.07 & 378.66 & \cg{\textbf{378.12 \dag \ddag}} &  383.78 \\
     & & & & &  \cmark & \textbf{411.98} & 414.29 & \cg\textbf{384.20} & 384.46 & 399.94 \\
    \hline
     & & & & & & & &\\[-2.8mm]
     \multirow{2}*{\underline{vegas}} & \multirow{2}*{504} & \multirow{2}*{5} & \multirow{2}*{106} & \multirow{2}*{14} & \xmark & 232,330.88 & \textbf{506.61} & \cg{\textbf{479.91}} & \cg{\textbf{479.91 \ddag}} & 482.06 \\
     & & & & &  \cmark & 287,615.23 & \textbf{515.22} & \cg{\textbf{493.60}} & 496.51 & 498.68 \\
\bottomrule
\end{tabular}} \medskip
    \caption{\color{black}Mean squared errors for the \textbf{Huber loss function} (with $\delta = 0.5$). We use the same abbreviations as in Table~\ref{tab:logitic}.}
    \label{tab:huber} \vspace{-0.5cm}
}
\end{table}

\begin{table}[tb]
{\color{black}
    \centering
    \resizebox{0.85\columnwidth}{!}{%
     \begin{tabular}{lcccc|c|cc|cc|c}
    \toprule
Dataset &  $N$ & $M_\mx$ & $M_\mz$ & $K$ & Reduced? & Nom  &  MixF & r-Nom  & r-MixF  & ConF \\
\midrule
     \multirow{2}*{\underline{bike}} & \multirow{2}*{17,379} & \multirow{2}*{4} & \multirow{2}*{31} & \multirow{2}*{5} & \xmark & 217.87 & \textbf{217.86} & 217.87 & \cg\textbf{217.83} & 217.86 \\
     & & & & &  \cmark & 218.17 & \textbf{218.15} & 218.17 & \cg\textbf{218.12} & 218.19 \\
    \hline
     & & & & & & & &\\[-2.8mm]
     \multirow{2}*{\underline{fire}} & \multirow{2}*{517} & \multirow{2}*{10} & \multirow{2}*{31} & \multirow{2}*{4} & \xmark & 124.36 & \cg\textbf{124.35} & 130.45 & \textbf{124.46} & 124.36 \\
     & & & & &  \cmark & 120.54 & \cg\textbf{109.80} & 116.67 & \textbf{109.85} & 109.81 \\
    \hline
     & & & & & & & &\\[-2.8mm]
     \multirow{2}*{\underline{flare}} & \multirow{2}*{1,066} & \multirow{2}*{0} & \multirow{2}*{21} & \multirow{2}*{9} & \xmark & \textbf{218.56} & 218.67 & 218.27 & \cg\textbf{204.93 \dag \ddag} & 218.56 \\
     & & & & & \cmark & 2,230.29 & \textbf{484.92} & 215.22 & \cg\textbf{203.17 \dag \ddag} & 216.09 \\
    \hline
     & & & & & & & &\\[-2.8mm]
     \multirow{2}*{\underline{garments}} & \multirow{2}*{1,197} & \multirow{2}*{7} & \multirow{2}*{22} & \multirow{2}*{4} & \xmark & 374.08 & \textbf{373.93} & 373.32 & \cg\textbf{372.87 \dag \ddag} & 376.05\\
     & & & & &  \cmark & 894.07 & \cg{\textbf{376.22 \dag}} & 377.78 & \textbf{376.44} &  376.92 \\
     \hline
     & & & & & & & &\\[-2.8mm]
     \multirow{2}*{\underline{imports}} & \multirow{2}*{193} & \multirow{2}*{14} & \multirow{2}*{45} & \multirow{2}*{10} & \xmark & 671.67 & \textbf{303.77} &351.66 & \cg\textbf{283.95 \dag \ddag} & 314.84 \\
     & & & & &  \cmark & 726.46 & \textbf{386.98} & 371.56 & \cg{\textbf{350.52 \dag \ddag}} & 394.28 \\
     \hline
     & & & & & & & &\\[-2.8mm]
     \multirow{2}*{\underline{student}} & \multirow{2}*{395} & \multirow{2}*{13} & \multirow{2}*{26} & \multirow{2}*{17} & \xmark & 399.65 & \textbf{396.69} & 387.07 & \cg{\textbf{386.44 \ddag}} & 404.71 \\
     & & & & &  \cmark & 451.85 & \textbf{397.77} & 403.51 & \cg{\textbf{393.42 \dag \ddag}} & 414.37 \\ 
    \hline
     & & & & & & & &\\[-2.8mm]
     \multirow{2}*{\underline{vegas}} & \multirow{2}*{504} & \multirow{2}*{5} & \multirow{2}*{106} & \multirow{2}*{14} & \xmark & 232,323.47 & \textbf{503.93} & 493.28 & \cg{\textbf{490.19 \dag \ddag}} & 507.81 \\
     & & & & &  \cmark & 287,607.592 & \textbf{507.60} & 517.85 & \cg\textbf{498.17 \dag \ddag} & 510.78 \\
\bottomrule
\end{tabular}} \medskip
    \caption{\color{black}Mean squared errors for the \textbf{pinball loss function} (with $\tau = 0.5$). We use the same abbreviations as in Table~\ref{tab:logitic}.}
    \label{tab:svr} \vspace{-0.5cm}
}
\end{table}

\begin{table}[tb]
{\color{black}
    \centering
    \resizebox{0.85\columnwidth}{!}{%
     \begin{tabular}{lcccc|c|cc|cc|c}
\toprule
Dataset &  $N$ & $M_\mx$ & $M_\mz$ & $K$ & Reduced? & Nom  &  MixF & r-Nom  & r-MixF  & ConF \\
\midrule
     \multirow{2}*{\underline{bike}} & \multirow{2}*{17,379} & \multirow{2}*{4} & \multirow{2}*{31} & \multirow{2}*{5} & \xmark & 217.87 & \cg\textbf{217.61} & 217.88 & \textbf{217.62} & \cg217.61  \\
     & & & & &  \cmark & 218.18  & \textbf{217.96}  &218.15 & \cg\textbf{217.95} & 217.99  \\
    \hline
     & & & & & & & &\\[-2.8mm]
     \multirow{2}*{\underline{fire}} & \multirow{2}*{517} & \multirow{2}*{10} & \multirow{2}*{31} & \multirow{2}*{4} & \xmark & 133.10 & \textbf{122.58} & 136.26 & \cg\textbf{121.13 \dag \ddag} & 122.58\\
     & & & & &  \cmark & 185.04 & \textbf{108.54} & 125.23 & \cg\textbf{108.28 \dag} & \textbf{108.54} \\
    \hline
     & & & & & & & &\\[-2.8mm]
     \multirow{2}*{\underline{flare}} & \multirow{2}*{1,066} & \multirow{2}*{0} & \multirow{2}*{21} & \multirow{2}*{9} & \xmark & \textbf{212.93} & 213.55 & 215.49 & \cg\textbf{211.32 \dag \ddag} & 214.00 \\
     & & & & &  \cmark & 1,562.88 & \textbf{209.84} & 232.67 & \cg\textbf{207.79 \dag \ddag} & 210.22 \\
    \hline
     & & & & & & & &\\[-2.8mm]
     \multirow{2}*{\underline{garments}} & \multirow{2}*{1,197} & \multirow{2}*{7} & \multirow{2}*{22} & \multirow{2}*{4} & \xmark & 369.10 & \cg\textbf{368.69} & 370.08 & \textbf{369.04} & 368.89 \\
     & & & & &  \cmark & 871.46 & \textbf{371.12} & \cg\textbf{370.72} & 371.45 & 371.79\\
     \hline
     & & & & & & & &\\[-2.8mm]
     \multirow{2}*{\underline{imports}} & \multirow{2}*{193} & \multirow{2}*{14} & \multirow{2}*{45} & \multirow{2}*{10} & \xmark & 621.38 & \cg{\textbf{293.31 \dag \ddag}} & 331.36 & \textbf{297.68} & 314.97 \\
     & & & & &  \cmark & 761.05 & \textbf{378.83} & 370.97 & \cg{\textbf{338.95 \dag \ddag}} & 382.32 \\
     \hline
     & & & & & & & &\\[-2.8mm]
     \multirow{2}*{\underline{student}} & \multirow{2}*{395} & \multirow{2}*{13} & \multirow{2}*{26} & \multirow{2}*{17} & \xmark & 399.65 & \textbf{397.63} & 383.57 &  \cg\textbf{383.18 \ddag} & 404.89 \\
     & & & & &  \cmark & 451.93 & \cg{\textbf{397.62 \dag \ddag}} & 410.40 & \textbf{399.46} & 412.98 \\
    \hline
     & & & & & & & &\\[-2.8mm]
     \multirow{2}*{\underline{vegas}} & \multirow{2}*{504} & \multirow{2}*{5} & \multirow{2}*{106} & \multirow{2}*{14} & \xmark & 232,332.44 & \textbf{502.89} & 503.37 & \cg\textbf{492.94 \dag \ddag} & 506.67 \\
     & & & & &  \cmark & 287,597.53 & \textbf{507.90} & 507.74 & \cg\textbf{502.73 \dag \ddag} & 510.15 \\
\bottomrule
\end{tabular}} \medskip
    \caption{\color{black}Mean squared errors for the \textbf{$\tau$-insensitive loss function} (with $\tau = 0.1$). We use the same abbreviations as in Table~\ref{tab:logitic}.}
    \label{tab:quant} \vspace{-0.5cm}
}
\end{table}

\clearpage

{
    \singlespacing
    \small
    \bibliographystyle{chicago}
    \bibliography{bibliography}
}

\clearpage

\setcounter{page}{1}
\renewcommand{\thepage}{EC-\arabic{page}}
\linespread{1.5}

\begin{center}
    \LARGE E-Companion to ``It's All in the Mix: \\ Wasserstein Machine Learning with Mixed Features''
\end{center}

\section*{\color{black} Appendix A: Wasserstein Classification with Mixed Features}\label{sec-classification}

{\color{black} This section derives exponential-size convex reformulations of the mixed-feature Wasserstein learning problem~\eqref{eq-dro-gen2} for classification problems with convex and Lipschitz continuous as well as piece-wise affine and convex loss functions $L$. One readily confirms that the resulting reformulations are special cases of our unified representation~\eqref{eq-most-vio-generic}, which implies that our results in this appendix prove Theorem~\ref{thm-convex-reformulation} for the case of classification problems.}

\begin{prop}\label{prop-der-class-conv}
    Consider a classification problem with a convex and Lipschitz continuous loss function as well as $\mX = \mR^{M_\mx}$. In this case, the Wasserstein learning problem~\eqref{eq-dro-gen2} is equivalent to
    \begin{myequation}\label{eq-prop-class-conv}
        \mspace{-50mu}
        \begin{array}{l@{\quad}l@{\qquad}l}
            \displaystyle \mathop{\mathrm{minimize}}_{\bbeta, \lambda, \bs} & \displaystyle \lambda \epsilon + \dfrac{1}{N} \sum_{n \in [N]} s_n \\[5mm]
            \displaystyle \mathrm{subject\;to} & 
            \left. \mspace{-10mu}
            \begin{array}{l}
                \displaystyle l_{\bbeta}(\bx^n, \bz, y^n) - \lambda \kappa_{\mz} d_{\mz}(\bz,\bz^n)  \leq s_n \\
                \displaystyle l_{\bbeta}(\bx^n, \bz, -y^n)  - \lambda \kappa_{\mz} d_{\mz}(\bz,\bz^n) - \lambda\kappa_{\my} \leq s_n
            \end{array}
            \right\}
            \forall n \in [N], \; \forall \bz \in \mC \\
            & \displaystyle \mathrm{lip}(L) \cdot \lVert \bbeta_{\mx}  \rVert_* \leq \lambda \\
            & \multicolumn{2}{l}{\displaystyle \mspace{-8mu} \bbeta = (\beta_0, \bbeta_\mx, \bbeta_\mz) \in \mR^{1 + M_\mx + M_\mz}, \;\; \lambda \in \mR_+, \;\; \bs \in \mR^N_+.}
        \end{array}
    \end{myequation}
\end{prop}

{\color{black} When $\kappa_\my = \infty$, the second constraint set in Proposition~\ref{prop-der-class-conv} is redundant for any $\lambda>0$. This is the case whenever $\lambda$ does not vanish at optimality (\textit{i.e.}, the Wasserstein radius $\varepsilon$ actually affects the optimal solution).}

{\color{black} Formulation~\eqref{eq-prop-class-conv} emerges as a special case of our unified representation~\eqref{eq-most-vio-generic} if we set $\mathcal{I} = \{ \pm 1 \}$, $\btheta = (\bbeta, \lambda)$, $\bsigma = \bs$, $f_0(\btheta,\bsigma) = \lambda \epsilon + \dfrac{1}{N} \sum_{n \in [N]} s_n$, $f_{ni}(e) = L(e)$, $g_{ni}(\btheta, \bm{\xi}^n_{-\mathbf{z}}; \bz) = i y^n \cdot [\beta_0 + \bbeta_\mx{}^\top \bx^n + \bbeta_\mz{}^\top \bz]$, $h_{ni}(\btheta, \bm{\xi}^n_{-\mathbf{z}}; d_{\mz} (\bz, \bz^n)) = \lambda \kappa_{\mz} d_{\mz}(\bz,\bz^n) + \lambda \kappa_{\my} \cdot \indicate{i = -1}$ and $\Theta = \{\btheta : \mathrm{lip}(L) \cdot \lVert \bbeta_{\mx}  \rVert_* \leq \lambda,\, \lambda \in \mR_+\}$. One verifies that this choice of $\mathcal{I}$, $\btheta$, $\bsigma$, $f_0$, $f_{ni}$, $g_{ni}$, $h_{ni}$ and $\Theta$ satisfies the regularity conditions imposed on problem~\eqref{eq-most-vio-generic} in the main text.

\begin{prop}\label{prop-bounded-class-conv}
    If $\bbeta$ in problem~\eqref{eq-prop-class-conv} is restricted to a bounded hypothesis set $\mathcal{H}$, then $\lambda$ can be restricted to a bounded set as well.
\end{prop}}

Proposition~\ref{prop-der-class-conv} covers Wasserstein support vector machines with a smooth Hinge loss,
\begin{equation*}
    L(e) = \begin{cases}
		1/2 - e & \text{if } e \leq 0, \\
        (1/2) \cdot (1 - e)^{2} & \text{if } e \in (0, 1), \\
        0 & \text{otherwise},
    \end{cases}
\end{equation*}
where the Lipschitz modulus is $\mathrm{lip}(L) = 1$, and logistic regression with a log-loss,
\begin{equation*}
    L(e) = \log (1 + \exp[-e]),
\end{equation*}
where again $\mathrm{lip}(L) = 1$. We provide the corresponding reformulations next.

\begin{coro}\label{coro-class-con}
    The first set of inequality constraints in~\eqref{eq-prop-class-conv} can be reformulated as follows.
    \begin{enumerate}[(i)]
        \item For the smooth Hinge loss function:
        \begin{equation*}
            \mspace{-20mu}
            \left.
            \begin{array}{l}
                \displaystyle \dfrac{1}{2}  \left(w_{\bz,n}^+ - y^n \cdot ( \beta_0 + \bbeta_{\mx}{}^\top \bx^n + \bbeta_{\mz}{}^\top \bz )\right)^2 + 1 - w_{\bz,n}^+ - \lambda \kappa_{\mz} d_{\mz}(\bz,\bz^n) \leq s_n \\[5mm]
                \dfrac{1}{2}  \left(w_{\bz,n}^+ - y^n \cdot ( \beta_0 + \bbeta_{\mx}{}^\top \bx^n + \bbeta_{\mz}{}^\top \bz )\right)^2 - \lambda \kappa_{\mz} d_{\mz}(\bz,\bz^n) \leq s_n
            \end{array}
            \right\}
            \forall n \in [N], \; \forall \bz \in \mC
        \end{equation*}
        \item For the log-loss function:
        \begin{equation*}
            \log\left(1 + \exp \left[- y^n \cdot ( \beta_0 + \bbeta_{\mx}{}^\top \bx^n + \bbeta_{\mz}{}^\top \bz) \right] \right) - \lambda \kappa_{\mz} d_{\mz}(\bz,\bz^n) \leq s_n
            \quad \forall n \in [N], \; \forall \bz \in \mC
        \end{equation*}
    \end{enumerate}
    Here, $w_{\bz,n}^+ \in \mR$ are auxiliary decision variables. The second set of inequality constraints in~\eqref{eq-prop-class-conv} follows similarly if we replace $w_{\bz,n}^+ \in \mR$ with additional auxiliary decision variables $w_{\bz,n}^- \in \mR$, replace $-y^n$ with $+y^n$ and subtract the expression $\lambda \kappa_{\my}$ from the constraint left-hand sides.
\end{coro}

We now provide a reformulation of problem~\eqref{eq-dro-gen2} without embedded maximizations when the loss function $L$ is piece-wise affine and convex.

\begin{prop}\label{prop-der-class-aff}
    Consider a classification problem with a piece-wise affine and convex loss function $L(e) = \max_{j \in [J]} \{ a_j e + b_j \}$, and assume that $\mX = \{ \bx \in \mR^{M_\mx} \, : \, \bm{C} \bx \preccurlyeq_{\mathcal{C}} \bm{d} \}$ for some $\bm{C} \in \mR^{r \times {M_\mx}}$, $\bm{d} \in \mR^r$ and proper convex cone $\mathcal{C} \subseteq \mR^{r}$. If $\mX$ admits a Slater point $\bx^\emph{s} \in \mR^{M_\mx}$ such that $\bm{C} \bx^\emph{s} \prec_{\mathcal{C}} \bm{d}$, then the Wasserstein learning problem~\eqref{eq-dro-gen2} is equivalent to
    \begin{myequation}\label{eq-prop-class-aff}
        \begin{array}{l@{\quad}l}
            \displaystyle \mathop{\text{\emph{minimize}}}_{\bbeta, \lambda, \bs,\bq_{nj}^+, \bq_{nj}^-} & \displaystyle \lambda \epsilon + \dfrac{1}{N} \sum_{n \in [N]} s_n \\[5mm]
            \displaystyle \raisebox{2mm}{\text{\emph{subject to}}} &
            \left. \mspace{-10mu}
            \begin{array}{l}
                \displaystyle \bq_{nj}^+{}^\top ( \bm{d} - \bm{C} \bx^n ) +  a_j  y^n \cdot ( \beta_0 + \bbeta_{\mx}{}^\top \bx^n + \bbeta_{\mz}{}^\top \bz ) - \lambda \kappa_{\mz} d_{\mz}(\bz,\bz^n) + b_j \leq s_n \\ 
                \displaystyle \bq_{nj}^-{}^\top ( \bm{d} - \bm{C} \bx^n ) -  a_j  y^n \cdot ( \beta_0 + \bbeta_{\mx}{}^\top \bx^n + \bbeta_{\mz}{}^\top \bz ) - \lambda \kappa_{\mz} d_{\mz}(\bz,\bz^n) - \lambda\kappa_{\my} + b_j \leq s_n
            \end{array}
            \right\} \\
            & \displaystyle \mspace{391mu} \forall n \in [N], \; \forall j \in [J], \; \forall \bz \in \mC \\
            & \displaystyle \lVert a_j  y^n \cdot \bbeta_{\mx} - \bm{C}{}^\top \bq_{nj}^+  \rVert_* \leq \lambda, \;\;
            \lVert a_j  y^n \cdot \bbeta_{\mx} + \bm{C}{}^\top \bq_{nj}^-  \rVert_* \leq \lambda \hfill \displaystyle \forall n \in [N], \; \forall j \in [J] \\
            & \displaystyle \bbeta = (\beta_0, \bbeta_\mx, \bbeta_\mz) \in \mR^{1 + M_\mx + M_\mz}, \;\; \lambda \in \mR_+, \;\; \bs \in \mR^N_+ \\
            & \displaystyle \bq_{nj}^+, \bq_{nj}^- \in \mathcal{C}^*, \, n \in [N] \text{ and } j \in [J].
        \end{array}
    \end{myequation}
\end{prop}

{\color{black} When $\kappa_\my = \infty$, the constraints involving $\bq_{nj}^-$ are redundant whenever $\lambda > 0$. In this case, the variables $\bq_{nj}^-$ and all constraints involving $\bq_{nj}^-$ can be removed from Proposition~\ref{prop-der-class-aff}.}

{\color{black} Formulation~\eqref{eq-prop-class-aff} emerges as a special case of our unified representation~\eqref{eq-most-vio-generic} if we set $i = (j, t)$, $\mathcal{I} = [J] \times \{ \pm 1 \}$, $\btheta = (\bbeta, \lambda, \bq_{ni})$, $\bsigma = \bs$, $f_0(\btheta,\bsigma) = \lambda \epsilon + \dfrac{1}{N} \sum_{n \in [N]} s_n$, $f_{ni}(e) = a_j e + b_j$, $g_{ni}(\btheta, \bm{\xi}^n_{-\mathbf{z}}; \bz) = t y^n \cdot [\beta_0 + \bbeta_\mx{}^\top \bx^n + \bbeta_\mz{}^\top \bz]$, $h_{ni}(\btheta, \bm{\xi}^n_{-\mathbf{z}}; d_{\mz} (\bz, \bz^n)) = - \bq_{ni}{}^\top ( \bm{d} - \bm{C} \bx^n ) + \lambda \kappa_{\mz} d_{\mz}(\bz,\bz^n) + \lambda \kappa_{\my} \cdot \indicate{t = -1}$ and $\Theta = \{\btheta : \lVert a_j  y^n \cdot \bbeta_{\mx} - t \cdot \bm{C}{}^\top \bq_{ni}  \rVert_* \leq \lambda \, \forall (n,i) \in [N] \times \mathcal{I},\, \bq_{ni} \in \mathcal{C}^* \, \forall (n,i) \in [N] \times \mathcal{I},\, \lambda \in \mR_+\}$. One verifies that this choice of $\mathcal{I}$, $\btheta$, $\bsigma$, $f_0$, $f_{ni}$, $g_{ni}$, $h_{ni}$ and $\Theta$ satisfies the regularity conditions imposed on problem~\eqref{eq-most-vio-generic} in the main text.

\begin{prop}\label{prop-bounded-class-aff}
    If $\bbeta$ in problem~\eqref{eq-prop-class-aff} is restricted to a bounded hypothesis set $\mathcal{H}$, then $(\bq_{ni}^+, \bq_{ni}^-)_{n,i}$ and $\lambda$ can be restricted to a bounded set as well.
\end{prop}}

Proposition~\ref{prop-der-class-aff} covers Wasserstein support vector machines with a (non-smooth) Hinge loss,
\begin{equation*}
    L(e) = \max \, \{ 1 - e, \; 0 \},
\end{equation*}
which is a piece-wise affine convex function with $J=2$, $a_1 = -1$, $b_1 = 1$, $a_2 = 0$ and $b_2 = 0$. We provide the corresponding reformulation next.

\begin{coro}\label{coro-class-aff}
    For the (non-smooth) Hinge loss function, the first set of inequality constraints in~\eqref{eq-prop-class-aff} can be reformulated as
    \begin{equation*}
        \bq_{n}^+{}^\top ( \bm{d} - \bm{C} \bx^n ) - y^n ( \beta_0 + \bbeta_{\mx}{}^\top \bx^n + \bbeta_{\mz}{}^\top \bz ) - \lambda \kappa_{\mz} d_{\mz}(\bz,\bz^n) + 1 \leq s_n
        \quad \forall n \in [N], \; \forall \bz \in \mC.
    \end{equation*}
    The second set of inequality constraints in~\eqref{eq-prop-class-aff} follows similarly if we replace $-y^n$ with $+y^n$ as well as $\bq_{n}^+$ with $\bq_{n}^-$ and subtract the expression $\lambda \kappa_{\my}$ from the constraint left-hand sides.
\end{coro}

\newpage

\section*{\color{black} Appendix B: Wasserstein Regression with Mixed Features}\label{sec-regression}

{\color{black} This section derives exponential-size convex reformulations of the mixed-feature Wasserstein learning problem~\eqref{eq-dro-gen2} for regression problems with convex and Lipschitz continuous as well as piece-wise affine and convex loss functions $L$. One readily confirms that the resulting reformulations are special cases of our unified representation~\eqref{eq-most-vio-generic}, which implies that our results in this appendix prove Theorem~\ref{thm-convex-reformulation} for the case of regression problems.}


\begin{prop}\label{prop-der-reg-conv}
    Consider a regression problem with a convex and Lipschitz continuous loss function $L$ and $(\mX, \mY) = (\mR^{M_\mx}, \mR)$. In this case, the Wasserstein learning problem~\eqref{eq-dro-gen2} equals
    \begin{myequation}\label{eq-prop-reg-conv}
        \begin{array}{l@{\quad}l@{\qquad}l}
            \displaystyle \mathop{\mathrm{minimize}}_{\bbeta,\lambda, \bs} & \displaystyle \lambda \epsilon + \dfrac{1}{N} \sum_{n \in [N]} s_n \\[5mm]
            \displaystyle \mathrm{subject\;to} & \displaystyle l_{\bbeta}(\bx^n, \bz, y^n) - \lambda \kappa_{\mz} d_{\mz} (\bz,\bz^n)  \leq s_n & \displaystyle \forall n \in [N],\; \forall \bz \in \mC  \\
            & \displaystyle \mathrm{lip}(L) \cdot \lVert \bbeta_{\mx} \rVert_{*}  \leq \lambda\\
            & \displaystyle \mathrm{lip}(L) \leq \lambda \kappa_{\my}\\
            & \multicolumn{2}{l}{\displaystyle \mspace{-8mu} \bbeta = (\beta_0, \bbeta_\mx, \bbeta_\mz) \in \mR^{1 + M_\mx + M_\mz}, \;\; \lambda \in \mR_+, \;\; \bs \in \mR^N_+.}
        \end{array}
    \end{myequation}
\end{prop}

{\color{black} When $\kappa_\my = \infty$ and $\lambda > 0$, the third constraint in Proposition~\ref{prop-der-reg-conv} is redundant.}

{\color{black} Formulation~\eqref{eq-prop-reg-conv} emerges as a special case of our unified representation~\eqref{eq-most-vio-generic} if we set $\btheta = (\bbeta, \lambda)$, $\bsigma = \bs$, $f_0(\btheta,\bsigma) = \lambda \epsilon + \dfrac{1}{N} \sum_{n \in [N]} s_n$, $f_{ni}(e) = L(e)$, $g_{ni}(\btheta, \bm{\xi}^n_{-\mathbf{z}}; \bz) = \beta_0 + \bbeta_\mx{}^\top \bx^n + \bbeta_\mz{}^\top \bz - y^n$, $h_{ni}(\btheta, \bm{\xi}^n_{-\mathbf{z}}; d_{\mz} (\bz, \bz^n)) = \lambda \kappa_{\mz} d_{\mz}(\bz,\bz^n)$ and $\Theta = \{\btheta : \mathrm{lip}(L) \cdot \lVert \bbeta_{\mx}  \rVert_* \leq \lambda,\, \mathrm{lip}(L) \leq \lambda \kappa_{\my},\, \lambda \in \mR_+\}$. One verifies that this choice of $\btheta$, $\bsigma$, $f_0$, $f_{ni}$, $g_{ni}$, $h_{ni}$ and $\Theta$ satisfies the regularity conditions imposed on problem~\eqref{eq-most-vio-generic} in the main text.

\begin{coro}\label{coro-bounded-reg-conv}
    If $\bbeta$ in problem~\eqref{eq-prop-reg-conv} is restricted to a bounded hypothesis set $\mathcal{H}$, then $\lambda$ can be restricted to a bounded set as well.
\end{coro}}

Proposition~\ref{prop-der-reg-conv} covers Wasserstein regression with a Huber loss,
\begin{equation*}
    L(e) = 
    \begin{cases}
        (1/2) \cdot e^2 & \text{if } \lvert e \rvert \leq \delta,\\
        \delta \cdot (\lvert e \rvert - (1/2) \cdot \delta) & \mbox{otherwise},
    \end{cases}
\end{equation*}
where $\delta \in \mR_+$ determines the boundary between the quadratic and the absolute loss, implying that $\mathrm{lip}(L) = \delta$. We provide the corresponding reformulation next.

\begin{coro}\label{coro-reg-con}
    For the Huber loss function, the first set of inequality constraints in~\eqref{eq-prop-reg-conv} can be reformulated as
    \begin{equation*}
        \left.
        \begin{array}{l}
            \dfrac{1}{2}( \beta_0 + \bbeta_{\mx}{}^\top \bx^n + \bbeta_{\mz}{}^\top \bz - y^n - p_{\bz,n})^2 + \delta p_{\bz,n} - \lambda \kappa_{\mz} d_{\mz} (\bz,\bz^n)  \leq s_n \\[4mm]
            \dfrac{1}{2}( \beta_0 + \bbeta_{\mx}{}^\top \bx^n + \bbeta_{\mz}{}^\top \bz - y^n - p_{\bz,n})^2 - \delta p_{\bz,n} - \lambda \kappa_{\mz} d_{\mz} (\bz,\bz^n)  \leq s_n
        \end{array}
        \right\}
        \forall n \in [N],\; \forall \bz \in \mC,
    \end{equation*}
    where $p_{\bz,n} \in \mR$ are auxiliary decision variables.
\end{coro}

We now provide a reformulation of problem~\eqref{eq-dro-gen2} without embedded maximizations when the loss function $L$ is piece-wise affine and convex.

\begin{prop}\label{prop-der-reg-aff}
    Consider a regression problem with a piece-wise affine and convex loss function $L(e) = \max_{j \in [J]} \{ a_j e + b_j \}$, and assume that $\mX \times \mY = \left\{ (\bx,y) \in \mR^{M_\mx+1} \, : \, \bm{C}_\mx \bx + \bm{c}_\my \cdot y \preccurlyeq_{\mathcal{C}} \bm{d} \right\}$ for some $\bm{C}_\mx \in \mR^{r \times {M_\mx}}$, $\bm{c}_\my \in \mR^{r}$, $\bm{d} \in \mR^{r}$ and proper convex cone $\mathcal{C} \subseteq \mR^{r}$. If this set admits a Slater point $(\bx^{\emph{s}}, y^{\emph{s}}) \in \mR^{M_\mx + 1}$ such that $\bm{C}_\mx \bx^{\emph{s}} + \bm{c}_\my \cdot  y^{\emph{s}} \prec_{\mathcal{C}} \bm{d}$, then problem~\eqref{eq-dro-gen2} is equivalent to
    \begin{myequation}\label{eq-prop-reg-aff}
        \begin{array}{l@{\quad}l@{\qquad}l}
            \displaystyle \mathop{\mathrm{minimize}}_{\bbeta,\lambda, \bs, \bq_{nj}} & \displaystyle \lambda \epsilon + \dfrac{1}{N} \sum_{n \in [N]} s_n \\[5mm]
            \displaystyle \mathrm{subject\;to} & \displaystyle \bq_{nj}^\top (\bm{d} - \bm{C}_{\mx} \bx^n - \bm{c}_\my \cdot y^n )  +  a_j \cdot ( \beta_0 +  \bbeta_{\mx}^\top \bx^n +  \bbeta_{\mz}{}^\top \bz - y^n ) - \lambda \kappa_{\mz} d_{\mz} (\bz,\bz^n) + b_j \leq s_n &\\
            & \mspace{362mu} \forall n \in [N], \; \forall j \in [J], \; \forall \bz \in \mC \\
            & \displaystyle \lVert a_j \cdot \bbeta_{\mx} - \bm{C}_\mx{}^\top \bq_{nj} \rVert_{*} \leq \lambda , \;\; \lvert -a_j - \bm{c}_\my^\top \bq_{nj} \rvert \leq \lambda \kappa_\my \mspace{57mu} \forall n \in [N] ,\; \forall j \in [J]\\
            & \multicolumn{2}{l}{\displaystyle \mspace{-8mu} \bbeta = (\beta_0, \bbeta_\mx, \bbeta_\mz) \in \mR^{1 + M_\mx + M_\mz}, \;\; \lambda \in \mR_+, \;\; \bs \in \mR^N_+} \\
            & \bq_{nj} \in \mathcal{C}^\star, \, n \in [N] \text{ and } j \in [J].
        \end{array}
    \end{myequation}
\end{prop}

{\color{black} When $\kappa_\my = \infty$ and $\lambda > 0$, the constraints $\lvert -a_j - \bm{c}_\my^\top \bq_{nj} \rvert \leq \lambda \kappa_\my$ in Proposition~\ref{prop-der-reg-aff} are redundant.}

{\color{black} Formulation~\eqref{eq-prop-reg-aff} emerges as a special case of our unified representation~\eqref{eq-most-vio-generic} if we set $\mathcal{I} = [J]$, $\btheta = (\bbeta, \lambda, \bq_{ni})$, $\bsigma = \bs$, $f_0(\btheta,\bsigma) = \lambda \epsilon + \dfrac{1}{N} \sum_{n \in [N]} s_n$, $f_{ni}(e) = a_i e + b_i$, $g_{ni}(\btheta, \bm{\xi}^n_{-\mathbf{z}}; \bz) = \beta_0 + \bbeta_\mx{}^\top \bx^n + \bbeta_\mz{}^\top \bz - y^n$, $h_{ni}(\btheta, \bm{\xi}^n_{-\mathbf{z}}; d_{\mz} (\bz, \bz^n)) = - \bq_{ni}{}^\top (\bm{d} - \bm{C}_{\mx} \bx^n - \bm{c}_\my \cdot y^n ) + \lambda \kappa_{\mz} d_{\mz}(\bz,\bz^n)$ and $\Theta = \{\btheta : \lVert a_i \cdot \bbeta_{\mx} - \bm{C}{}^\top \bq_{ni}  \rVert_* \leq \lambda \, \forall (n,i) \in [N] \times \mathcal{I},\, \lvert -a_i - \bm{c}_\my^\top \bq_{ni} \rvert \leq \lambda \kappa_\my \, \forall (n,i) \in [N] \times \mathcal{I},\, \bq_{ni} \in \mathcal{C}^* \, \forall (n,i) \in [N] \times \mathcal{I},\, \lambda \in \mR_+\}$. One verifies that this choice of $\mathcal{I}$, $\btheta$, $\bsigma$, $f_0$, $f_{ni}$, $g_{ni}$, $h_{ni}$ and $\Theta$ satisfies the regularity conditions imposed on problem~\eqref{eq-most-vio-generic} in the main text.

\begin{coro}\label{coro-bounded-reg-aff}
    If $\bbeta$ in problem~\eqref{eq-prop-reg-aff} is restricted to a bounded hypothesis set $\mathcal{H}$, then $(\bq_{ni})_{n, i}$ and $\lambda$ can be restricted to a bounded set as well.
\end{coro}}

Proposition~\ref{prop-der-reg-aff} covers Wasserstein support vector regression with an $\tau$-insensitive loss function,
\begin{equation*}
    L(e) = \max \left\{ \lvert e \rvert - \tau, \; 0 \right\}
\end{equation*}
with robustness parameter $\tau \in \mR_+$, which is a piece-wise affine convex function with $J=3$, $a_1 = 1$, $b_1 = -\tau$, $a_2 = -1$, $b_2 = -\tau$, $a_3 = 0$ and $b_3 = 0$. It also covers Wasserstein quantile regression with a pinball loss function,
\begin{equation*}
    L(e) = \displaystyle \mathop{\mathrm{max}} \left\{ -\tau e, \; (1-\tau) \mspace{1mu} e \right\}
\end{equation*}
with robustness parameter $0 \leq \tau \leq 1$, which is a piece-wise affine convex function with $J=2$, $a_1 = -\tau$, $b_1 = 0$, $a_2 = 1-\tau$ and $b_2 = 0$. We provide the corresponding reformulations next.

\begin{coro}\label{coro-reg-aff}
    The first set of inequality constraints in~\eqref{eq-prop-reg-aff} can be reformulated as follows.
    \begin{enumerate}[(i)]
        \item For the $\tau$-insensitive loss function:
        \begin{equation*}
            \mspace{-55mu}
            \left.
            \begin{array}{l}
                \displaystyle \bt_{n1}{}^\top ( \bm{d}_1 - \bm{C}_1\bx^n  ) + \bv_{n1}{}^\top ( \bm{d}_2 - \bm{C}_2 y^n  ) + (\beta_0 + \bbeta_{\mx}{}^\top \bx^n + \bbeta_{\mz}{}^\top \bz -y^n) \\
                \displaystyle \mspace{380mu} -\tau - \lambda \kappa_{\mz} d_{\mz} (\bz,\bz^n) \leq s_n  \\
                \displaystyle \bt_{n2}{}^\top ( \bm{d}_1 - \bm{C}_1\bx^n  ) + \bv_{n2}{}^\top ( \bm{d}_2 - \bm{C}_2 y^n  ) - ( \beta_0 + \bbeta_{\mx}{}^\top \bx^n + \bbeta_{\mz}{}^\top \bz -y^n) \\
                \displaystyle \mspace{380mu} -\tau - \lambda \kappa_{\mz} d_{\mz} (\bz,\bz^n) \leq s_n
            \end{array}
            \right\} \;\;
            \forall n \in [N], \; \forall \bz \in \mC \\
        \end{equation*}
        \item For the pinball loss function:
        \begin{equation*}
            \mspace{-80mu}
            \left.
            \begin{array}{l}
                \displaystyle \bt_{n1}{}^\top ( \bm{d}_1 - \bm{C}_1\bx^n  ) + \bv_{n1}{}^\top ( \bm{d}_2 - \bm{C}_2 y^n  ) - \tau ( \beta_0 + \bbeta_{\mx}{}^\top \bx^n + \bbeta_{\mz}{}^\top \bz -y^n ) \\
                \displaystyle \mspace{465mu} - \lambda \kappa_{\mz} d_{\mz} (\bz,\bz^n) \leq s_n  \\
                \displaystyle \bt_{n2}{}^\top ( \bm{d}_1 - \bm{C}_1\bx^n  ) + \bv_{n2}{}^\top ( \bm{d}_2 - \bm{C}_2 y^n  ) + (1-\tau) ( \beta_0 + \bbeta_{\mx}{}^\top \bx^n + \bbeta_{\mz}{}^\top \bz -y^n ) \\
                \displaystyle \mspace{465mu} - \lambda \kappa_{\mz} d_{\mz} (\bz,\bz^n) \leq s_n
            \end{array}
            \right\} \;\;
            \forall n \in [N], \; \forall \bz \in \mC
        \end{equation*}
    \end{enumerate}
\end{coro}

\newpage

\section*{\color{black} Appendix C: Proofs}

\begin{proof}[{\bf Proof of Observation~\ref{obs-derivation}}.]
    The statement follows from similar arguments as in the proof of Theorem~1 by \cite{NIPS2015}. Details are omitted for the sake of brevity.
\end{proof}

\begin{proof}[{\bf Proof of Proposition~\ref{prop-cutting}}]
    We first show that $\mathrm{LB}$ and $\mathrm{UB}$ constitute lower and upper bounds on the optimal value of~\eqref{eq-most-vio-generic} throughout the execution of the algorithm, and then we conclude that Algorithm~\ref{alg-cutting} terminates in finite time with an optimal solution $(\btheta^\star, \bsigma^\star)$ to problem~\eqref{eq-most-vio-generic}.

    Algorithm~\ref{alg-cutting} updates $\mathrm{LB}$ to $f_0 (\btheta^\star, \bsigma^\star)$ in each iteration of the while-loop. Since $(\btheta^\star, \bsigma^\star)$ is an optimal solution to the relaxation of problem~\eqref{eq-most-vio-generic} that only contains the constraints $\mathcal{W} \subseteq [N] \times \mathcal{I} \times \mC$, $\mathrm{LB}$ indeed constitutes a lower bound on the optimal value of~\eqref{eq-most-vio-generic}. Moreover, since no elements are ever removed from $\mathcal{W}$, the sequence of lower bounds $\mathrm{LB}$ is monotonic.

    To see that $\mathrm{UB}$ constitutes an upper bound on the optimal value of~\eqref{eq-most-vio-generic} throughout the execution of Algorithm~\ref{alg-cutting}, we claim that in each iteration, $(\btheta^\star, \bsigma^\star + \bm{\vartheta}^\star)$ constitutes a feasible solution to problem~\eqref{eq-most-vio-generic}. Indeed, we have $(\btheta^\star, \bsigma^\star) \in \Theta \times \mathbb{R}^N_+$ and $\bm{\vartheta}^\star \in\mathbb{R}^N_+$ by construction, while for all $n \in [N]$, $i \in \mathcal{I}$ and $\bm{z} \in \mC$, we have that
    \begin{equation*}
        f_{ni} (g_{ni} (\btheta^\star, \bm{\xi}^n_{-\mathbf{z}}; \bz)) - h_{ni} (\btheta^\star, \bm{\xi}^n_{-\mathbf{z}}; d_{\mz} (\bz, \bz^n))
        \; \leq \;
        \sigma_n^\star + \vartheta (n, i)
        \; \leq \;
        \sigma_n^\star + \vartheta (n, i(n))
        \; \leq \;
        \sigma_n^\star + \vartheta^\star_n,
    \end{equation*}
    that is, $(\btheta^\star, \bsigma^\star + \bm{\vartheta}^\star)$ is indeed feasible in problem~\eqref{eq-most-vio-generic}. Moreover, the sequence of upper bounds $\mathrm{UB}$ is monotonic by construction of the updates.

    To see that Algorithm~\ref{alg-cutting} terminates in finite time, finally, note that each iteration of the while-loop either adds a new constraint index $(n, i(n), \bm{z} (n, i(n)))$ to $\mathcal{W}$, or we have $\vartheta (n, i(n)) \leq 0$ for all $n \in [N]$. In the latter case, however, we have $\bm{\vartheta}^\star = \bm{0}$ and thus $\mathrm{LB} = \mathrm{UB}$ at the end of the iteration. The claim now follows from the fact that the index set $[N] \times \mathcal{I} \times \mC$ is finite.
\end{proof}

\begin{proof}[{\bf Proof of Theorem~\ref{thm-violation}}]
    For fixed $(\btheta^\star, \bm{\sigma}^\star)$, finding the most violated constraint index $\bm{z} (n, i)$ in constraint group $(n, i)$ amounts to solving the combinatorial optimization problem
    \begin{equation*}
        \begin{array}{ll}
            \underset{\bz}{\mathrm{maximize}} & \displaystyle f_{ni} (g_{ni} (\btheta^\star, \bm{\xi}^n_{-\mathbf{z}}; \bz)) - h_{ni} (\btheta^\star, \bm{\xi}^n_{-\mathbf{z}}; d_{\mz} (\bz, \bz^n)) - \sigma^\star_n \\
            \mathrm{subject\;to} & \bz \in \mC.
        \end{array}
    \end{equation*}
    We can solve this problem by solving the $K + 1$ problems
    \begin{equation*}
        \left[
        \begin{array}{ll}
            \underset{\bz}{\mathrm{maximize}} & \displaystyle f_{ni} (g_{ni} (\btheta^\star, \bm{\xi}^n_{-\mathbf{z}}; \bz)) - h_{ni} (\btheta^\star, \bm{\xi}^n_{-\mathbf{z}}; \delta) - \sigma^\star_n \\
            \mathrm{subject\;to} & d_{\mz} (\bz, \bz^n) = \delta \\
            & \bz \in \mC
        \end{array}
        \right]
        \quad \forall \delta \in [K] \cup \{ 0 \},
    \end{equation*}
    where each problem conditions on a fixed number $\delta$ of discrepancies between $\bm{z}$ and $\bm{z}^n$, and subsequently choosing any solution $\bm{z} (\delta)$ that attains the maximum optimal objective value among those $K + 1$ problems. Removing constant terms from those $K + 1$ problems, we observe that the $\delta$-th problem shares its set of optimal solutions with the problem
    \begin{equation*}
        \begin{array}{ll}
            \underset{\bz}{\mathrm{maximize}} & \displaystyle f_{ni} (g_{ni} (\btheta^\star, \bm{\xi}^n_{-\mathbf{z}}; \bz)) \\
            \mathrm{subject\;to} & d_{\mz} (\bz, \bz^n) = \delta \\
            & \bz \in \mC.
        \end{array}
    \end{equation*}
    Since the outer function $f_{ni}$ in the objective function is convex, the objective function is maximized whenever the inner function $g_{ni}$ in the objective function is either maximized or minimized. We thus conclude that the $\delta$-th problem is solved by solving the two problems
    \begin{equation*}
        \left[
        \begin{array}{ll}
            \underset{\bz}{\mathrm{maximize}} & \displaystyle \mu \cdot g_{ni} (\btheta^\star, \bm{\xi}^n_{-\mathbf{z}}; \bz) \\
            \mathrm{subject\;to} & d_{\mz} (\bz, \bz^n) = \delta \\
            & \bz \in \mC
        \end{array}
        \right]
        \quad \forall \mu \in \{ \pm 1 \}
    \end{equation*}
    and subsequently choosing the solution $\bm{z} (\mu, \delta)$ that attains the larger value $f_{ni} (g_{ni} (\btheta^\star, \bm{\xi}^n_{-\mathbf{z}}; \bz (\mu, \delta))$ among those two solutions (with ties broken arbitrarily). Fixing $\mu$ to either value, adopting the notation for $\bm{w}$ and $w_0$ of Algorithm~\ref{alg-most-violated} and ignoring constant terms, we can write the problem as
    \begin{equation*}
        \begin{array}{ll}
            \underset{\bz}{\mathrm{maximize}} & \displaystyle \mu \cdot \Big[ \sum_{m \in [K]} \bm{w}_m{}^\top \bz_m \Big] \\[5mm]
            \mathrm{subject\;to} & d_{\mz} (\bz, \bz^n) = \delta \\
            & \bz \in \mC.
        \end{array}
    \end{equation*}
    The rectangularity of $\mC$ implies that the decisions of this problem admit a decomposition into the selection $\mathcal{M} \subseteq [K]$ of $\delta$ discrete features $m \in [K]$ along which $\bm{z}_m$ differs from $\bm{z}^n_m$ and, for those features $m \in [K]$ where $\bm{z}_m \neq \bm{z}^n_m$, the choice of $\bm{z}_m \in \mathbb{Z} (k_m) \setminus \{ \bm{z}^n_m \}$:
    \begin{equation*}
        \begin{array}{ll}
            \underset{\mathcal{M}}{\mathrm{maximize}} & \displaystyle \max_{\bz} \left\{
                \mu \cdot \Big[ \sum_{m \in [K]} \bm{w}_m{}^\top \bz_m \Big]
                \, : \,
                \left[
                \begin{array}{ll}
                    \bz_m \in \mathbb{Z} (k_m) \setminus \{ \bz^n_m \} & \forall m \in \mathcal{M} \\
                    \bz_m = \bz^n_m & \forall m \in [K] \setminus \mathcal{M}
                \end{array}
                \right]
            \right\} \\[7mm]
            \mathrm{subject\;to} & \mathcal{M} \subseteq [K], \;\; | \mathcal{M} | = \delta.
        \end{array}
    \end{equation*}
    Noticing that the embedded maximization problem decomposes along the discrete features $m \in [K]$, we can adopt the notation for $\bz^\star_m$ of Algorithm~\ref{alg-most-violated} to obtain the equivalent formulation
    \begin{equation*}
        \begin{array}{ll}
            \underset{\mathcal{M}}{\mathrm{maximize}} & \displaystyle \Big[ \sum_{m \in \mathcal{M}} \mu \cdot \bm{w}_m{}^\top \bz^\star_m \Big] + \Big[ \sum_{m \in [K] \setminus \mathcal{M}} \mu \cdot \bm{w}_m{}^\top \bz^n_m \Big] \\[5mm]
            \mathrm{subject\;to} & \mathcal{M} \subseteq [K], \;\; | \mathcal{M} | = \delta.
        \end{array}
    \end{equation*}
    The two summations in the objective function of this problem admit the reformulation
    \begin{equation*}
        \Big[ \sum_{m \in \mathcal{M}} \mu \cdot \bm{w}_m{}^\top \bz^\star_m \Big] + \Big[ \sum_{m \in [K] \setminus \mathcal{M}} \mu \cdot \bm{w}_m{}^\top \bz^n_m \Big]
        \; = \;
        \Big[ \sum_{m \in [K]} \mu \cdot \bm{w}_m{}^\top \bz^n_m \Big] + \Big[ \sum_{m \in \mathcal{M}} \mu \cdot \bm{w}_m{}^\top (\bz^\star_m - \bz^n_m) \Big],
    \end{equation*}
    and ignoring constant terms once more simplifies our optimization problem to
    \begin{equation*}
        \begin{array}{ll}
            \underset{\mathcal{M}}{\mathrm{maximize}} & \displaystyle \sum_{m \in \mathcal{M}} \mu \cdot \bm{w}_m{}^\top (\bz^\star_m - \bz^n_m) \\[5mm]
            \mathrm{subject\;to} & \mathcal{M} \subseteq [K], \;\; | \mathcal{M} | = \delta.
        \end{array}
    \end{equation*}
    This problem is solved by identifying $\mathcal{M}$ with the indices of the $\delta$ largest elements of the sequence $\mu \cdot \bm{w}_1{}^\top (\bz^\star_1 - \bz^n_1)$, \ldots, $\mu \cdot \bm{w}_K{}^\top (\bz^\star_K - \bz^n_K)$, and this problem can be solved by a simple sorting algorithm. One readily verifies that Algorithm~\ref{alg-most-violated} adopts the solution approach just described to determine a maximally violated constraint index $\bz (n, i)$.

    The runtime of Algorithm~\ref{alg-most-violated}, finally, is dominated by determining the $2K$ maximizers $\bz^\star_m$, $m \in [K]$ and $\mu \in \{ \pm 1 \}$, which takes time $\mathcal{O} (M_{\mathbf{z}})$ due to the one-hot encoding employed by $\mC$, sorting the $2K$ values $\mu \cdot \bm{w}_m{}^\top (\bz^\star_m - \bz^n_m)$, which takes time $\mathcal{O} (K \log K)$, as well as determining a maximally violated constraint among the $2K + 2$ candidates $\bm{z} \in \mathcal{Z}$, which takes time $\mathcal{O} (KT)$.
\end{proof}

{\color{black}
\begin{proof}[\bf Proof of Theorem~\ref{thm:complexity}.]
    Recall from \citet[Corollary~4.2.7]{GLS88:geom_algos} that problem~\eqref{eq-most-vio-generic} can be solved to $\delta$-accuracy in polynomial time if \emph{(i)} the problem admits a polynomial time weak separation oracle and if \emph{(ii)} the feasible region of the problem is a circumscribed convex body. We prove both of these properties next.

    Fix a convex and compact set $\mathcal{K} \subseteq \mR^n$ and a rational number $\delta > 0$. \citet[Definition~2.1.13]{GLS88:geom_algos} define a \emph{weak} separation oracle for $\mathcal{K}$ as an algorithm which for any vector $\bq \in \mR^n$ either confirms that $\bq \in \mathcal{S}(\mathcal{K},\delta)$, where $\mathcal{S}(\mathcal{K},\delta) = \left\{ \bm{r} \in \mR^{n} \, : \, \lVert \bm{r} - \bm{r}' \rVert_{2} \leq \delta \text{ for some } \bm{r}' \in \mathcal{K}\right\}$ is the $\delta$-enclosure around $\mathcal{K}$ (\emph{i.e.}, $\bq$ is \emph{almost} in $\mathcal{K}$), or finds a vector $\bm{c} \in \mR^n$ with $\lVert \bm{c} \rVert_\infty = 1$ such that $\bm{c}^\top \bp \leq \bm{c}^\top \bq + \delta$ for all $\bp \in \mathcal{S}(\mathcal{K},-\delta)$, where $\mathcal{S}(\mathcal{K},-\delta) = \left\{ \bm{r} \in \mathcal{K} \, : \, \mathcal{S}(\{ \bm{r} \},\delta) \subseteq \mathcal{K} \right\}$ is the $\delta$-interior of $\mathcal{K}$ (\emph{i.e.}, $\bm{c}$ is an \emph{almost} separating hyperplane). We construct a weak separation oracle for problem \eqref{eq-most-vio-generic} as follows. Denote the feasible region of problem~\eqref{eq-most-vio-generic} by $\mathcal{K} \subseteq \Theta \times \mR^{N}_{+}$, and fix any $\bq' = (\btheta', \bsigma') \in \mR^{M_{\btheta}} \times \mR^{N}$. If $\sigma'_n < 0$ for any $n \in [N]$, then our algorithm returns the separating hyperplane $\bm{c}^\top = (\bm{0}^\top \, -\mathbf{e}_n{}^\top)$. If $\btheta' \notin \Theta$, then we can augment the weakly separating hyperplane $\bm{c} \in \mR^{M_{\btheta}}$ for the set $\Theta$ to a weakly separating hyperplane $(\bm{c}^\top \; \bm{0}^\top)$ for problem~\eqref{eq-most-vio-generic}. Otherwise, we leverage Algorithm~\ref{alg-most-violated} to obtain a \emph{strong} separation oracle for $\mathcal{K}$, which \citet[Definition~2.1.4]{GLS88:geom_algos} define as an algorithm which for any vector $\bq \in \mR^n$ either confirms that $\bq \in \mathcal{K}$ or finds a vector $\bm{c} \in \mR^n$ such that $\bm{c}^\top \bp < \bm{c}^\top \bq$ for all $\bp \in \mathcal{K}$. To this end, we execute Algorithm~\ref{alg-most-violated} for all constraint group indices $(n,i) \in [N] \times \mathcal{I}$ to identify most violated constraint indices $\bz(n,i)$ along with the constraint violations $\vartheta(n,i)$. If $\vartheta(n,i) \leq 0$ for all $(n, i) \in [N] \times \mathcal{I}$, then $\bm{q}'$ is feasible in problem~\eqref{eq-most-vio-generic} and our oracle terminates. If $\vartheta(n^\star,i^\star) > 0$ for some $(n^\star, i^\star) \in [N] \times \mathcal{I}$, on the other hand, then our oracle computes a subgradient $\bm{c}$ of the function
    \begin{equation*}
        r (\btheta, \bsigma)
        \;\; = \;\;
        f_{n^\star i^\star} (g_{n^\star i^\star} (\btheta, \bm{\xi}^{n^\star}_{-\mathbf{z}}; \bz(n^\star, i^\star))) - h_{n^\star i^\star} (\btheta, \bm{\xi}^{n^\star}_{-\mathbf{z}}; d_{\mz} (\bz(n^\star,i^\star), \bz^{n^\star})) - \sigma_{n^\star}
    \end{equation*}
    at $(\btheta, \bsigma) = (\btheta', \bsigma')$. We have
    \begin{equation*}
        \bm{c}^\top (\btheta, \bsigma)
        \;\; \leq \;\;
        \bm{c}^\top (\btheta', \bsigma') + r(\btheta, \bsigma) - r(\btheta', \bsigma')
        \;\; < \;\;
        \bm{c}^\top (\btheta', \bsigma')
        \qquad \forall (\btheta, \bsigma) \in \mathcal{K},
    \end{equation*}
    where the first inequality is due to the definition of subgradients and the fact that $r$ is convex, and the second inequality holds since $r (\btheta', \bsigma') > 0$ while $r (\btheta, \bsigma) \leq 0$ for all $(\btheta, \bsigma) \in \mathcal{K}$. In conclusion, $\bm{c}$ is a strong separating hyperplane for $(\btheta', \bsigma')$ as desired. Note that by construction, a strong separation oracle is also a weak separation oracle. Finally, the assumptions in the statement of our theorem, together with Theorem~\ref{thm-violation}, \mbox{imply that our oracle runs in polynomial time.}

    We next turn to the claim that the feasible region $\mathcal{K} \subseteq \Theta \times \mR^{N}_{+}$ of problem~\eqref{eq-most-vio-generic} is a circumscribed convex body. According to \citet[Definition~2.1.16]{GLS88:geom_algos}, this is the case whenever $\mathcal{K}$ is a finite-dimensional, full-dimensional, closed and convex subset of a ball whose finite radius we can specify. By construction, $\mathcal{K}$ is finite-dimensional, closed and convex. Moreover, $\mathcal{K}$ is full-dimensional since $\Theta$ is full-dimensional by construction and the constraints have an epigraphical structure. To see that $\mathcal{K}$ can be circumscribed by a ball whose finite radius we can specify, we note that $\Theta$ is bounded by assumption and $\bsigma$ is non-negative by construction. Since the objective function $f_0$ in~\eqref{eq-most-vio-generic} is non-decreasing in $\bsigma$, we can without loss of generality include in~\eqref{eq-most-vio-generic} the additional constraints $\sigma_n \leq \overline{\sigma}_n$ for
    \begin{equation*}
        \overline{\sigma}_n
        \; = \;
        \max_{\btheta \in \Theta} \; \max_{i \in \mathcal{I}} \; \max_{\bz \in \mC} \left\{
        f_{ni} (g_{ni} (\btheta, \bm{\xi}^n_{-\mathbf{z}}; \bz)) - h_{ni} (\btheta, \bm{\xi}^n_{-\mathbf{z}}; d_{\mz} (\bz, \bz^n))
        \right\},
        \qquad n \in [N].
    \end{equation*}
    Note that all $\overline{\sigma}_n$ are finite since $\Theta$ is compact, $\mathcal{I}$ and $\mC$ are finite sets and the objective function is continuous. We thus conclude that $\bsigma$ can be bounded as well, and thus $\mathcal{K}$ can indeed be circumscribed by a ball whose finite radius we can specify.
\end{proof}}

{\color{black}
\begin{proof}[{\bf Proof of Theorem~\ref{thm-equivalency}}]
    Fix any Wasserstein classification or regression instance as described in the statement of the theorem, fix any ambiguity radius $\epsilon > \kappa_{\mz}{K}^{1/p}$, any non-zero numbers ${M_\mx}$ of continuous and $K$ of discrete features, and denote the training set as $\{ \bm{\xi}^n \}_{n \in [N]}$ with $\bm{\xi}^n = (\bx^n, \bz^n, y^n)$, $n \in [N]$.
    Our proof proceeds in three steps. We first derive a closed-form expression for the objective function of the Wasserstein learning problem~\eqref{eq:the_mother_of_all_problems} at a judiciously chosen model $\hat{\bbeta}$. Our derivation will show that this objective function constitutes the sum of the empirical loss $N^{-1} \cdot \sum_{n \in [N]} l_{\hat{\bbeta}}(\bm{\xi}^n)$ and a function $h_{\hat{\bbeta}} (\{\bxi^n\}_{n \in [N]})$. We then construct two sets of data points $\{\hat{\bxi}^n\}_{n \in [N]}$ and $\{\check{\bxi}^n\}_{n \in [N]}$ at which $h_{\hat{\bbeta}} (\{\hat{\bxi}^n\}_{n \in [N]}) \neq h_{\hat{\bbeta}} (\{\check{\bxi}^n\}_{n \in [N]})$, showing that $h_{\hat{\bbeta}} (\{\bxi^n\}_{n \in [N]})$ exhibits a dependence on the dataset $\{\bxi^n\}_{n \in [N]}$ that cannot be replicated by any data-agnostic regularizer $\mathfrak{R} (\hat{\bbeta})$.

    Fix the model $\hat{\bbeta} = (\hat{\beta}_0, \hat{\bbeta}_\mx, \hat{\bbeta}_\mz)$ with $\hat{\beta}_0 \in \mathbb{R}$ and $\hat{\bbeta}_\mx \neq \bm{0}$ selected arbitrarily. Fix any discrete feature realization $\bm{z}^\star \in \mathbb{Z} \setminus \{ \bm{z}^1 \}$ as well as the discrete feature coefficients
    \begin{equation*}
        \hat{\beta}_{\mz,mi} = 
        \begin{cases}
			-\dfrac{2}{d'} \cdot {\kappa_{\mz} \mathrm{lip}(L)} \cdot \lVert \hat{\bbeta}_\mx \rVert_* \cdot d_\mz(\bz^\star,\bz^1) & \text{if} \ z^1_{mi} = 1, \\[2mm]
            \mspace{14mu} \dfrac{2}{d'} \cdot {\kappa_{\mz} \mathrm{lip}(L)} \cdot \lVert \hat{\bbeta}_\mx \rVert_* \cdot d_\mz(\bz^\star,\bz^1) & \text{otherwise}
    \end{cases}
    \qquad \text{for all } m \in [K] \text{ and } i \in [k_m],
    \end{equation*}
    where $d' = (\mathrm{d} / \mathrm{d} {e}) L(e) \; \big|_{e = x'}$ is the derivative of the loss function $L$ at any point $x' \in \mathbb{R}$ where the derivative does not vanish. Such points $x'$ exist due to Rademacher's theorem, which ensures that a Lipschitz continuous function is differentiable almost everywhere, as well as the assumption that the loss function $L$ is non-constant. We make two observations that we will leverage later on in this proof:
    \begin{enumerate}
        \item[\emph{(i)}] We have $d' \cdot \hat{\beta}_{\mz,mi} \cdot (z_{mi} - z^1_{mi}) \geq 0$ for all $\bz \in \mZ$, $m \in [K]$ and $i \in [k_m]$. Indeed, fix any $\bz \in \mZ$, $m \in [K]$ and $i \in [k_m]$. If $z^1_{mi} = 0$, then $d' \cdot \hat{\beta}_{\mz,mi} = 2 \cdot {\kappa_{\mz} \mathrm{lip}(L)} \cdot \lVert \hat{\bbeta}_\mx \rVert_* \cdot d_\mz(\bz^\star,\bz^1) \geq 0$ and $z_{mi} - z^1_{mi} \geq 0$. If $z^1_{mi} = 1$, on the other hand, then $d' \cdot \hat{\beta}_{\mz,mi} = -2 \cdot {\kappa_{\mz} \mathrm{lip}(L)} \cdot \lVert \hat{\bbeta}_\mx \rVert_* \cdot d_\mz(\bz^\star,\bz^1) \leq 0$ and $z_{mi} - z^1_{mi} \leq 0$. In particular, we have $d' \cdot \hat{\bbeta}_{\mz}{}^\top (\bz - \bz^1) \geq 0$ for all $\bz \in \mZ$.
        \item[\emph{(ii)}] We have $d' \cdot \hat{\bbeta}_\mz{}^\top (\bz^\star - \bz^1) > {\kappa_{\mz} \mathrm{lip}(L)} \cdot \lVert \hat{\bbeta}_\mx \rVert_* \cdot d_\mz(\bz^\star,\bz^1)$. Indeed, since $\bz^\star \neq \bz^1$, there is at least one $m^\star \in [K]$ and $i^\star \in [k_{m^\star}]$ where $z^\star_{m^\star i^\star} - z^1_{m^\star i^\star} \neq 0$, and thus $d' \cdot \hat{\beta}_{\mz,m^\star i^\star} \cdot (z^\star_{m^\star i^\star} - z^1_{m^\star i^\star}) = 2 \cdot {\kappa_{\mz} \mathrm{lip}(L)} \cdot \lVert \hat{\bbeta}_\mx \rVert_* \cdot d_\mz(\bz^\star,\bz^1) > {\kappa_{\mz} \mathrm{lip}(L)} \cdot \lVert \hat{\bbeta}_\mx \rVert_* \cdot d_\mz(\bz^\star,\bz^1)$. The claim now follows from the fact that $d' \cdot \hat{\beta}_{\mz,mi} \cdot (z_{mi} - z^1_{mi}) \geq 0$ for all other  $m \in [K]$ and $i \in [k_m]$ thanks to our previous observation~\emph{(i)}.
    \end{enumerate}

    Note that piece-wise affine loss functions are Lipschitz continuous. For our selected problem instance, Proposition~\ref{prop-der-class-conv} from Appendix~A therefore implies that the objective function of the  Wasserstein classification problem becomes
    \begin{equation*}
        \begin{array}{l@{\quad}l@{\qquad}l}
            \displaystyle \mathop{\mathrm{minimize}}_{\lambda, \bs} & \displaystyle \lambda \epsilon + \dfrac1{N} \sum_{n \in [N]} s_n \\[5mm]
            \displaystyle \mathrm{subject\;to} & 
            \left. \mspace{-10mu}
            \begin{array}{l}
                \displaystyle l_{\hat{\bbeta}}(\bx^n, \bz, y^n) - \lambda \kappa_{\mz} d_{\mz}(\bz,\bz^n)  \leq s_n \\
                \displaystyle l_{\hat{\bbeta}}(\bx^n, \bz, -y^n)  - \lambda \kappa_{\mz} d_{\mz}(\bz,\bz^n) - \lambda\kappa_{\my} \leq s_n
            \end{array}
            \right\}
            \forall n \in [N], \; \forall \bz \in \mC \\
            & \displaystyle \mathrm{lip}(L) \cdot \lVert \hat{\bbeta}_{\mx}  \rVert_* \leq \lambda \\
            & \displaystyle \lambda \in \mR_+, \;\; \bs \in \mR^N_+,
        \end{array}
    \end{equation*}
    and Proposition~\ref{prop-der-reg-conv} from Appendix~B implies that the objective function of the Wasserstein regression problem becomes
    \begin{equation*}
        \begin{array}{l@{\quad}l@{\qquad}l}
            \displaystyle \mathop{\mathrm{minimize}}_{\lambda, \bs} & \displaystyle \lambda \epsilon + \dfrac1{N} \sum_{n \in [N]} s_n \\[5mm]
            \displaystyle \mathrm{subject\;to} & \displaystyle l_{\hat{\bbeta}}(\bx^n, \bz, y^n) - \lambda \kappa_{\mz} d_{\mz} (\bz,\bz^n)  \leq s_n & \displaystyle \forall n \in [N],\; \forall \bz \in \mC  \\
            & \displaystyle \mathrm{lip}(L) \cdot \lVert \hat{\bbeta}_{\mx} \rVert_{*}  \leq \lambda\\
            & \displaystyle \mathrm{lip}(L) \leq \lambda \kappa_{\my}\\
            & \displaystyle \mspace{-8mu} \lambda \in \mR_+, \;\; \bs \in \mR^N_+.
        \end{array}
    \end{equation*}
    When $\kappa_{\my} \to \infty$, the second set of constraints in the classification problem and the third constraint in the regression problem become redundant since $\lambda \geq \mathrm{lip}(L) \cdot \lVert \hat{\bbeta}_{\mx} \rVert_{*} > 0$ because $\hat{\bbeta}_\mx \neq \bm{0}$. For all $n \in [N]$, we can then replace each $s_n$ with a maximum of the left-hand side of the first constraint set over all $\bz \in \mZ$ in either problem to obtain the unified formulation
    \begin{equation*}
        \begin{array}{l@{\quad}l@{\qquad}l}
            \displaystyle \mathop{\mathrm{minimize}}_{ \lambda} & \displaystyle \lambda \epsilon +  \dfrac1{N} \sum_{n \in [N]} \max_{\bz \in \mZ} \displaystyle \left\{ l_{\hat{\bbeta}}(\bx^n, \bz, y^n) - \lambda \kappa_{\mz} d_{\mz} (\bz,\bz^n) \right\} \\
            \displaystyle \mathrm{subject\;to} & \displaystyle \mathrm{lip}(L) \cdot \lVert \hat{\bbeta}_{\mx} \rVert_{*}  \leq \lambda\\
            & \displaystyle \lambda \in \mR_+
        \end{array}
    \end{equation*}
    of the objective function of both the classification and the regression problem. Note that the non-negativity of $s_n$ for $n \in [N]$ is preserved in the unified formulation since $d_{\mz} (\bz^n,\bz^n) = 0$ and $l_{\hat{\bbeta}}(\bx^n, \bz^n, y^n) \geq 0$ for all $n \in [N]$ by definition of the loss function $L$ that underlies $l_{\hat{\bbeta}}$. We claim that $\lambda ^\star = \mathrm{lip}(L) \cdot \lVert \hat{\bbeta}_{\mx} \rVert_{*}$ at optimality. Indeed, any increment $\Delta \lambda > 0$ in $\lambda$ will cause an increase of $ \Delta \lambda \cdot \epsilon$ and a maximum decrease of $\Delta \lambda \cdot \kappa_{\mz} K^{1/p}$ in the objective function, and we have $\epsilon > \kappa_{\mz} K^{1/p}$ by assumption.  Hence, the objective function of the Wasserstein classification and regression problem simplifies to
    \begin{align*}
        f_1 (\hat{\bbeta}, \{\bxi^n\}_{n\in[N]})
        \; & = \; 
        \lambda ^\star \epsilon + \dfrac1{N} \sum_{n \in [N]} \max_{\bz \in \mZ} \displaystyle \left\{ l_{\hat{\bbeta}}(\bx^n, \bz, y^n) - \lambda^\star \kappa_{\mz} d_{\mz} (\bz,\bz^n) \right\}\\
        \; & = \; \displaystyle \dfrac1{N} \sum_{n \in [N]} l_{\hat{\bbeta}}(\bm{\xi}^n) + h_{\hat{\bbeta}} (\{\bxi^n\}_{n \in [N]}),
    \end{align*}
    where 
    \begin{equation*}
        h_{\hat{\bbeta}} (\{\bxi^n\}_{n \in [N]})
        \; = \;
        \displaystyle \lambda ^\star \epsilon + \dfrac1{N} \sum_{n \in [N]} \max_{\bz \in \mZ} \left\{ l_{\hat{\bbeta}}(\bx^n, \bz, y^n) - l_{\hat{\bbeta}}(\bx^n, \bz^n, y^n) - \lambda^\star \kappa_{\mz} d_{\mz} (\bz,\bz^n) \right\}.
    \end{equation*}
    In contrast, the objective function of a generic regularized learning problem has the form
    \begin{equation*}
        f_2 (\hat{\bbeta}, \{\bxi^n\}_{n\in[N]})
        \; = \;
        \displaystyle \dfrac1{N} \sum_{n \in [N]} l_{\hat{\bbeta}}(\bm{\xi}^n)+ \mathfrak{R} (\hat{\bbeta}).
    \end{equation*}
    By construction, $\mathfrak{R}(\hat{\bbeta})$ does not vary with $\{\bxi^n\}_{n\in[N]}$. In contrast, we claim that $h_{\hat{\bbeta}} (\{\bxi^n\}_{n\in[N]})$ varies with $\{\bxi^n\}_{n\in[N]}$. To this end, we will construct two sets of data points $\{\hat{\bxi}^n\}_{n \in [N]}$ and $\{\check{\bxi}^n\}_{n \in [N]}$ at which $h_{\hat{\bbeta}} (\{\hat{\bxi}^n\}_{n \in [N]}) \neq h_{\hat{\bbeta}} (\{\check{\bxi}^n\}_{n \in [N]})$. The two datasets $\{\hat{\bxi}^n\}_{n \in [N]}$ and $\{\check{\bxi}^n\}_{n \in [N]}$ are identical to $\{\bxi^n\}_{n \in [N]}$ except for the realization ${\bx}^1$ of the continuous features in sample~$1$. In other words, we have $(\hat{\bx}^n, \hat{\bz}^n, \hat{y}^n) = (\check{\bx}^n, \check{\bz}^n, \check{y}^n) = (\bx^n, \bz^n, y^n)$ for all $n \in [N] \setminus 1$ as well as $(\hat{\bz}^1, \hat{y}^1) = (\check{\bz}^1, \check{y}^1) = (\bz^1, y^1)$.

    We select $\hat{\bx}^1$ such that the maximum in the definition of $h_{\hat{\bbeta}}$ is strictly positive at sample $n = 1$. To achieve this, we choose $\hat{\bx}^1$ such that $\hat{y}^1 \cdot [\hat{\beta}_0 + \hat{\bbeta}_\mx{}^\top \hat{\bx}^1 + \hat{\bbeta}_\mz{}^\top \hat{\bz}^1] = x'$ for classification problems and $\hat{\beta}_0 + \hat{\bbeta}_\mx{}^\top \hat{\bx}^1 + \hat{\bbeta}_\mz{}^\top \hat{\bz}^1 - \hat{y}^1 = x'$ for regression problems, respectively, where $x'$ was defined earlier as a point where the derivative of the loss function does not vanish. Note that this is always possible since $\hat{\bbeta}_\mx \neq \bm{0}$. We then have 
    \begin{align*}
        \max_{\bz \in \mZ} l_{\hat{\bbeta}}(\hat{\bx}^1, \bz, \hat{y}^1) - l_{\hat{\bbeta}}(\hat{\bx}^1, \hat{\bz}^1, \hat{y}^1) &- \lambda^\star \kappa_{\mz} d_{\mz} (\bz,\hat{\bz}^1)\\
        \; &\geq \; l_{\hat{\bbeta}}(\hat{\bx}^1, \bz^\star, \hat{y}^1) - l_{\hat{\bbeta}}(\hat{\bx}^1, \hat{\bz}^1, \hat{y}^1) - \lambda^\star \kappa_{\mz} d_{\mz} (\bz^\star,\hat{\bz}^1)\\
        \; &= \;
        L(x' + \hat{\bbeta}_\mz{}^\top (\bz^\star - \hat{\bz}^1)) - L(x') - \lambda^\star \kappa_{\mz} d_{\mz} (\bz^\star,\hat{\bz}^1)
        \\
        \; & \geq \;
        \hat{\bbeta}_\mz{}^\top (\bz^\star - \hat{\bz}^1) \cdot \dfrac{\mathrm{d}}{\mathrm{d}{e}} L(e) \; \big|_{e =x'} - \lambda^\star \kappa_{\mz} d_{\mz} (\bz^\star,\hat{\bz}^1) \\
        &= \;
        d' \cdot \hat{\bbeta}_\mz{}^\top (\bz^\star - \hat{\bz}^1) -  {\kappa_{\mz} \mathrm{lip}(L)} \cdot \lVert \hat{\bbeta}_\mx \rVert_* \cdot d_\mz(\bz^\star,\hat{\bz}^1) > 0,    
    \end{align*}
    where the first inequality holds since $\bz^\star \in \mZ$, the first identity uses the definitions of $l_{\hat{\bbeta}}$ and $\hat{\bx}^1$, the second inequality exploits the convexity of $L$, the second identity uses the definition of $d'$ as well as $\lambda^\star$, and the third inequality follows from our earlier observation \emph{(ii)}. Since the term inside the maximum in the definition of $h_{\hat{\bbeta}}$ is strictly positive for $\bz = \bz^\star$ at sample $n=1$, the maximum is guaranteed to be positive at sample $n=1$ as well.

    We choose $\check{\bx}^1$ such that the term inside the maximum in the definition of $h_{\hat{\bbeta}}$ is strictly negative at sample $n=1$ for all $\bz \in \mZ \setminus \{\check{\bz}^1\}$. Consider the case where $d' > 0$; the alternative case where $d' < 0$ follows from analogous arguments. Since $L$ is non-negative and convex, we have that 
    \begin{myequation}\label{eq:limit}
        \lim_{x \rightarrow -\infty} \dfrac{\mathrm{d}}{\mathrm{d}{e}} L(e) \; \big|_{e = x}\leq 0,
    \end{myequation}
    which in turn implies that for all $\bz \in \mZ \setminus \{\check{\bz}^1\}$, we have
    \begin{myequation}\label{eq:nonpos_in_the_limit2}
        \lim_{x \rightarrow -\infty} L(x + \hat{\bbeta}_\mz{}^\top (\bz - \check{\bz}^1)) - L(x)
        \; \leq \;
        \lim_{x \rightarrow -\infty} \hat{\bbeta}_\mz{}^\top (\bz - \check{\bz}^1) \cdot \dfrac{\mathrm{d}}{\mathrm{d}{e}} L(e) \; \big|_{e = x + \hat{\bbeta}_\mz{}^\top (\bz - \check{\bz}^1)}
        \; \leq \;
        0,
    \end{myequation}
    where the first inequality exploits the fact that the derivative of a convex function is non-decreasing and the second inequality holds because $d' \cdot\hat{\bbeta}_\mz{}^\top (\bz - \check{\bz}^1) \geq 0$ due to our earlier observation \emph{(i)} and the fact that $\lim_{x \rightarrow -\infty} \dfrac{\mathrm{d}}{\mathrm{d}{e}} L(e) \; \big|_{e = x + \hat{\bbeta}_\mz{}^\top (\bz - \check{\bz}^1)} = \lim_{x \rightarrow -\infty} \dfrac{\mathrm{d}}{\mathrm{d}{e}} L(e) \; \big|_{e = x} \leq 0$ by equation~\eqref{eq:limit}. Moreover, for $n=1$ and for all $\bz \in \mZ \setminus \{\check{\bz}^1\}$, we have
    \begin{align*}
        \lim_{\hat{\bbeta}_\mx{}^\top \check{\bx}^1 \rightarrow -\infty} l_{\hat{\bbeta}}(\check{\bx}^1, \bz, \check{y}^1) & - l_{\hat{\bbeta}}(\check{\bx}^1, \check{\bz}^1, \check{y}^1) - \lambda^\star \kappa_{\mz} d_{\mz} (\bz,\check{\bz}^1)\\
        \; & = \;
        \lim_{x'' \rightarrow - \infty} L(x'' + \hat{\bbeta}_\mz{}^\top (\bz - \check{\bz}^1)) - L(x'') - \lambda^\star \kappa_{\mz} d_{\mz} (\bz,\check{\bz}^1)\\
        \; & \leq \;
        - \lambda^\star \kappa_{\mz} d_{\mz} (\bz,\check{\bz}^1)
        \; < \;
        0,
    \end{align*}
    where the identity applies the change of variables $\check{y}^1 \cdot [\hat{\beta}_0 + \hat{\bbeta}_\mx{}^\top \check{\bx}^1 + \hat{\bbeta}_\mz{}^\top \check{\bz}^1] = x''$ for classification problems and $\hat{\beta}_0 + \hat{\bbeta}_\mx{}^\top \check{\bx}^1 + \hat{\bbeta}_\mz{}^\top \check{\bz}^1 - \check{y}^1 = x''$ for regression problems, respectively, the first inequality uses~\eqref{eq:nonpos_in_the_limit2}, and the last inequality is due to the fact that $\lambda^\star$, $\kappa_\mz$ and $d_{\mz} (\bz,\check{\bz}^1)$ are all strictly positive. Since the term inside the maximum in the definition of $h_{\hat{\bbeta}}$ is strictly negative at sample $n=1$ for all $\bz \in \mZ \setminus \{\check{\bz}^1\}$, the maximum is guaranteed to evaluate to $0$ for $n=1$. Given that the datasets $\{\hat{\bxi}^n\}_{n \in [N]}$ and $\{\check{\bxi}^n\}_{n \in [N]}$ are identical for all other samples $n \in [N] \setminus 1$, this implies that $h_{\hat{\bbeta}} (\{\hat{\bxi}^n\}_{n \in [N]}) \neq h_{\hat{\bbeta}} (\{\check{\bxi}^n\}_{n \in [N]})$ and confirms that $h_{\hat{\bbeta}} (\{\bxi^n\}_{n\in[N]})$ has a dependency on the dataset $\{\bxi^n\}_{n\in[N]}$ that cannot be replicated by any data-agnostic regularizer $\mathfrak{R} (\hat{\bbeta})$.
\end{proof}}

{\color{black}
\begin{proof}[{\bf Proof of Proposition~\ref{prop:concave}}]
    We prove the statement of the proposition in three steps. We first show that our revised ground metric $d_{\mathrm{u}}$ coincides with the original ground metric $d_{\mathrm{z}}$ over all pairs of discrete feature vectors $\bm{z}, \bm{z}' \in \mathbb{Z}$. This implies that the left-hand sides of the constraints indexed by $\bm{u} \in \mathbb{Z}$ in~\eqref{eq-dro-gen2-v2} 
    reduce to their counterparts in~\eqref{eq-dro-gen2}. In other words, problem~\eqref{eq-dro-gen2-v2} is a restriction of~\eqref{eq-dro-gen2} since it contains all of the constraints of~\eqref{eq-dro-gen2}. We then prove that $d_{\mathrm{u}} (\bm{u}, \bm{z}')$ is concave in $\bm{u}$ for every fixed $\bm{z}' \in \mathbb{Z}$. Our third step leverages this intermediate result to show that every feasible solution to our mixed-feature Wasserstein learning problem~\eqref{eq-dro-gen2} is also feasible in the bounded continuous-feature formulation~\eqref{eq-dro-gen2-v2}.
    
    In view of the first step, fix any $\bm{z}, \bm{z}' \in \mZ$ and observe that for any $m \in [K]$, we have
    \begin{align*}
        \dfrac{1}{2} \lVert \bm{z}_m - \bm{z}'_m \rVert_1 + \dfrac{1}{2} \lvert \mathbf{1}^\top(\bm{z}_m - \bm{z}'_m) \rvert 
        \;\; = \;\;
        \dfrac{1}{2} \Big\lVert \begin{psmallmatrix} \bm{z}_m \\ 1 - \mathbf{1}^\top \bm{z}_m \end{psmallmatrix} - \begin{psmallmatrix} \bm{z}'_m \\ 1 - \mathbf{1}^\top \bm{z}'_m \end{psmallmatrix} \Big\rVert_{1}
        \;\; = \;\;
        \begin{cases}
            0 & \text{if } \bm{z}_m = \bm{z}'_m, \\
            1 & \text{otherwise}.
        \end{cases}
    \end{align*}
    Here, the first identity applies basic algebraic manipulations, and the second identity holds since $(\bm{z}_m^\top, \, 1 - \mathbf{1}^\top \bm{z}_m)$ and $(\bm{z}^{\prime\top}_m, \, 1 - \mathbf{1}^\top \bm{z}'_m)$ are canonic basis vectors in $\mathbb{R}^{k_m}$. We thus conclude that
    \begin{equation*}
        d_{\mathrm{u}} (\bm{z}, \bm{z}')
        =
        \left(\sum_{m\in[K]} \frac{1}{2}\lVert \bm{z}_m - \bm{z}'_m \rVert_1 + \frac{1}{2} \lvert \mathbf{1}^\top(\bm{z}_m - \bm{z}'_m)\rvert \right)^{1/p}
        =
        \sum_{m \in [K]} (\indicate{\bm{z}_m \neq \bm{z}'_m})^{1/p}
        =
        d_{\mathrm{z}} (\bm{z}, \bm{z}')
    \end{equation*}
    as claimed.

    As for the second statement, fix any $\bm{z}' \in \mathbb{Z}$. We then find that
    \begin{align*}
        d_{\mathrm{u}} (\bm{u}, \bm{z}')
        \;\; = \;\;
        \Bigg( \frac{1}{2} \sum_{m \in [K]} \sum_{j \in [k_m - 1]} &\left[ \indicate{z'_{mj} = 0} \cdot u_{mj} + \indicate{z'_{mj} = 1} \cdot (1 - u_{mj}) \right]
        \;\; + \\
        \frac{1}{2} \sum_{m \in [K]} &\left[ \indicate{\bm{z}'_m = \bm{0}} \cdot \mathbf{1}^\top \bm{u}_m + \indicate{\bm{z}'_m \neq \bm{0}} \cdot (1 - \mathbf{1}^\top \bm{u}_m) \right] \Bigg)^{1/p},
    \end{align*}
    where the first summation evaluates the 1-norm differences between $\bm{u}_m$ and $\bm{z}'_m$ and the second summation computes the absolute values in $d_{\mathrm{u}}$, respectively. The above reformulation shows that the mapping $\bm{u} \mapsto d_{\mathrm{u}} (\bm{u}, \bm{z}')$ can be represented as $f(g(\bm{u}))$, where $f(x) = x^{1/p}$ and $g : \mathbb{U} \rightarrow \mathbb{R}$ is affine. It then follows from \S 3.2.2 of \cite{boyd2004convex} that the mapping $\bm{u} \mapsto d_{\mathrm{u}} (\bm{u}, \bm{z}')$ is concave.
    
    In view of the third statement, finally, fix any feasible solution $(\bm{\beta}, \lambda, \bm{s})$ to problem~\eqref{eq-dro-gen2}. Since the objective functions of~\eqref{eq-dro-gen2} and~\eqref{eq-dro-gen2-v2} coincide, we only need to show that $(\bm{\beta}, \lambda, \bm{s})$ satisfies the constraints of problem~\eqref{eq-dro-gen2-v2}. To this end, we observe that for all $n \in [N]$, we have
    \begin{equation*}
        \mspace{-10mu}
        \begin{array}{r@{\quad}l@{\qquad}l}
            & \displaystyle \mspace{-21mu} \sup_{(\bx, y) \in \mX \times \mY} \left\{ l_{\bbeta} (\bx, \bm{u}, y) - \lambda \lVert \bx - \bx^n \rVert  - \lambda \kappa_\my d_\my (y, y^n) \right\} - \lambda \kappa_\mz d_\mathrm{u} (\bm{u}, \bz^n) \leq s_n 
            & \displaystyle
            \forall \bm{u} \in \mathbb{U} \\
            \Longleftrightarrow & \displaystyle
            \sup_{\bm{u} \in \mathbb{U}} \left\{ l_{\bbeta} (\bx, \bm{u}, y) - \lambda \lVert \bx - \bx^n \rVert  - \lambda \kappa_\my d_\my (y, y^n) \right\} - \lambda \kappa_\mz d_\mathrm{u} (\bm{u}, \bz^n) \leq s_n 
            & \displaystyle 
            \forall (\bx, y) \in \mX \times \mY \\
            \Longleftrightarrow & \displaystyle
            \sup_{\bm{u} \in \mathbb{U}} \left\{ l_{\bbeta} (\bx, \bm{u}, y) - \lambda \kappa_\mz d_\mathrm{u} (\bm{u}, \bz^n) \} - \lambda \lVert \bx - \bx^n \rVert  - \lambda \kappa_\my d_\my (y, y^n) \right\}  \leq s_n 
            & \displaystyle
            \forall (\bx, y) \in \mX \times \mY \\
            \Longleftrightarrow & \displaystyle 
            \sup_{\bm{z} \in \mathbb{Z}} \left\{ l_{\bbeta} (\bx, \bm{z}, y) - \lambda \kappa_\mz d_\mathrm{u} (\bm{z}, \bz^n) \} - \lambda \lVert \bx - \bx^n \rVert  - \lambda \kappa_\my d_\my (y, y^n) \right\}  \leq s_n 
            & \displaystyle 
            \forall (\bx, y) \in \mX \times \mY \\
            \Longleftrightarrow & \displaystyle 
            \sup_{\bm{z} \in \mathbb{Z}} \left\{ l_{\bbeta} (\bx, \bm{z}, y) - \lambda \kappa_\mz d_\mathrm{z} (\bm{z}, \bz^n) \} - \lambda \lVert \bx - \bx^n \rVert  - \lambda \kappa_\my d_\my (y, y^n) \right\}  \leq s_n 
            & \displaystyle 
            \forall (\bx, y) \in \mX \times \mY \\
            \Longleftrightarrow & \displaystyle 
            \mspace{-21mu} \sup_{(\bx, y) \in \mX \times \mY} \left\{ l_{\bbeta} (\bx, \bm{z}, y) - \lambda \kappa_\mz d_\mathrm{z} (\bm{z}, \bz^n) \} - \lambda \lVert \bx - \bx^n \rVert  - \lambda \kappa_\my d_\my (y, y^n) \right\}  \leq s_n
            & \displaystyle 
            \forall \bm{z} \in \mathbb{Z}.
        \end{array}
    \end{equation*}
    Here, the first, the second and the last equivalences reorder terms. As for the third equivalence, recall that $d_{\mathrm{u}} (\bm{u}, \bm{z}^n)$ is concave in $\bm{u}$ for every fixed $\bm{z}^n \in \mathbb{Z}$. Thus, the expression inside the supremum in the third row is convex in $\bm{u}$, which implies that it attains its maximum at an extreme point of $\mathbb{U}$. The equivalence then follows from the fact that $\mathbb{Z}$ coincides with the extreme points of $\mathbb{U}$. The fourth equivalence, finally, leverages the first step of this proof, which showed that $d_{\mathrm{u}}$ and $d_{\mathrm{z}}$ agree over all pairs of discrete feature vectors $\bm{z}, \bm{z}' \in \mathbb{Z}$.
\end{proof}
}

The proof of Proposition~\ref{prop-der-class-conv} utilizes the following lemma, which we state and prove first.

\begin{lem}\label{lem-conv-class}
    Assume that the loss function $L$ is convex and Lipschitz continuous. For fixed $\bm{\alpha}, \bx^n \in \mR^{M_\mx}$, $\alpha_0 \in \mR$ and $\lambda \in \mR_+$, we have
\begin{myequation}\label{eq-lem-1}
    \sup_{\bx \in \mR^{M_\mx}} \, L( \bm{{\alpha}}{}^\top \bx + \alpha_0 ) - \lambda \lVert\bx - \bx^n\rVert = \begin{cases}
			L( \bm{\alpha}{}^\top \bx^n + \alpha_0 ) & \text{if} \ \mathrm{{lip}}(L) \cdot \lVert \bm{\alpha} \rVert_{*} \leq \lambda,\\
            + \infty & \mbox{\text{otherwise}}.
		 \end{cases}
\end{myequation}
\end{lem}

Lemma~\ref{lem-conv-class} generalizes Lemma 47 of \cite{JMLR:v20:17-633} in that it includes a constant $\alpha_0$ in the argument of the loss function $L$ and that it extends to the case where $\lambda = 0$.

\begin{proof}[{\bf Proof of Lemma~\ref{lem-conv-class}}]
    Consider first the case where $\bm{\alpha} \neq \bm{0}$. We conduct the change of variables $\bw = \bx + \bm{d}$, where $\bm{d} \in \mR^{M_\mx}$ is any vector such that $\bm{\alpha}{}^\top \bm{d} = \alpha_0$. Note that $\bm{d}$ is guaranteed to exist since $\bm{\alpha} \neq \bm{0}$. Setting $\bw^n = \bx^n + \bm{d}$, the left-hand side of \eqref{eq-lem-1} can then be written as
    \begin{align*}
        \sup_{\bx \in \mR^{M_\mx}} \, L( \bm{\alpha}{}^\top \bx + \alpha_0 ) - \lambda \lVert\bx - \bx^n\rVert & = \sup_{\bw \in \mR^{M_\mx}} \, L( \bm{\alpha}{}^\top \bw) - \lambda \lVert{\bw-\bm{d}} - (\bw^n-\bm{d})\rVert\\
        & = \sup_{\bw \in \mR^{M_\mx}} \, L( \bm{\alpha}{}^\top \bw) - {\lambda} \lVert\bw-\bw^n\rVert 
    \end{align*}
    If $\lambda > 0$, then Lemma 47 of \cite{JMLR:v20:17-633} can be directly applied:
    \begin{align*}
        \sup_{\bw \in \mR^{M_\mx}} \, L( \bm{\alpha}{}^\top \bw) - {\lambda} \lVert\bw-\bw^n\rVert &= \begin{cases}
            L( \bm{\alpha}{}^\top \bw^n  ) & \text{if} \ \mathrm{lip}(L) \cdot \lVert \bm{\alpha} \rVert_{*} \leq {\lambda}, \\
            + \infty & \mbox{otherwise},
		\end{cases}\\
		&= \begin{cases}
			L( \bm{\alpha}{}^\top \bx^n + \alpha_0 ) & \text{if} \ \mathrm{lip}(L) \cdot \lVert \bm{\alpha} \rVert_{*} \leq {\lambda}, \\
            + \infty & \mbox{otherwise}.
		\end{cases}
    \end{align*}
    Note that Lemma~47 of \cite{JMLR:v20:17-633} assumes $\lambda > 0$. For the case where $\lambda = 0$, the left-hand side of \eqref{eq-lem-1} evaluates to $+ \infty$ since $L$ is assumed to be non-constant (\emph{cf.}~Section~\ref{sec-wasser}) and convex. The right-hand side of \eqref{eq-lem-1} also evaluates to $+ \infty$ since $\mathrm{lip}(L) \cdot \lVert \bm{\alpha} \rVert_{*} > \lambda$ due to $L$ being non-constant and $\bm{\alpha} \neq \bm{0}$. Hence, the equivalence also extends to the case where $\lambda = 0$.
    
    Now consider the case where $\bm{\alpha} = \bm{0}$. There is no $\bm{d}$ such that $\bm{\alpha}{}^\top \bm{d} = \alpha_0$ unless $\alpha_0 = 0$. However, when $\bm{\alpha} = \bm{0}$, the left-hand side of \eqref{eq-lem-1} has the trivial solution $\bx = \bx^n$, and the right-hand side of \eqref{eq-lem-1} simplifies to $L(\alpha_0)$ since $0 \leq {\lambda}$ always holds. Thus, the equivalence still holds, which concludes the proof.
\end{proof}

\begin{proof}[{\bf Proof of Proposition~\ref{prop-der-class-conv}}]
    The statement can be proven along the lines of the proof of Theorem~14~(ii) by \cite{JMLR:v20:17-633} if we leverage Lemma~\ref{lem-conv-class} to re-express the embedded maximization over $\bm{x} \in \mX$. Details are omitted for the sake of brevity.
\end{proof}

{\color{black}
\begin{proof}[{\bf Proof of Proposition~\ref{prop-bounded-class-conv}}]
    For ease of exposition, we adopt the notation introduced after problem~\eqref{eq-prop-class-conv}. By construction of problem~\eqref{eq-prop-class-conv}, $\lambda$ is lower bounded by $0$. To see that we can bound $\lambda$ from above as well, we observe that there are optimal solutions for which $\lambda$ does not exceed $\overline{\lambda} = \max \{ \overline{\lambda}_1, \, \overline{\lambda}_2 \}$, where
    \begin{equation*}
        \overline{\lambda}_1 \; = \;
        \sup_{\bbeta \in \mathcal{H}} \; \mathrm{lip}(L) \cdot \lVert \bbeta_{\mx} \rVert_*
    \end{equation*}
    with the bounded set $\mathcal{H} \subseteq \mR^{1 + M_{\bx} + M_{\bz}}$ containing all admissible hypotheses $\bbeta$, and
    \begin{myequation}\label{eq:lipschitz_lambda_2}
        \overline{\lambda}_2 \; = \;
        \sup_{\bbeta \in \mathcal{H}} \max_{n \in [N]} \; \max_{i \in \mathcal{I}} \; \max_{\bz \in \mC} \left\{
        \frac{l_{\bbeta}(\bx^n, \bz, iy^n)}{\kappa_\mz d_\mz(\bz,\bz^n) + \kappa_{\my} \cdot \indicate{i = -1}} \, : \,
        (\bz, i) \neq (\bz^n, 1)
        \right\}.
    \end{myequation}
    Indeed, selecting $\lambda \geq \overline{\lambda}_1$ ensures that all hypotheses $\bbeta \in \mathcal{H}$ are feasible, and selecting $\lambda \geq \overline{\lambda}_2$ implies that for any $\bbeta \in \mathcal{H}$, all left-hand sides in the first two constraint sets of problem~\eqref{eq-prop-class-conv} are non-positive, and thus all of these constraints are weakly dominated by the non-negativity constraints on $\bs$. Note that the numerator in the objective function of~\eqref{eq:lipschitz_lambda_2} is bounded since all suprema and maxima in~\eqref{eq:lipschitz_lambda_2} operate over bounded sets and $l_{\bbeta}$ is Lipschitz continuous. Moreover, the denominator in the objective function of~\eqref{eq:lipschitz_lambda_2} is bounded from below by a strictly positive quantity since $(\bz, i) \neq (\bz^n, 1)$ and $\kappa_{\mz}, \kappa_{\my} > 0$.
\end{proof}}

\begin{proof}[{\bf Proof of Corollary~\ref{coro-class-con}}]
    The proof is similar to those of Corollaries~16~and~17 by \cite{JMLR:v20:17-633}. Details are omitted for the sake of brevity.
\end{proof}

\begin{proof}[{\bf Proof of Proposition~\ref{prop-der-class-aff}}]
    The statement can be proven along the lines of the proof of Theorem~14~(i) by \cite{JMLR:v20:17-633}. We omit the details for the sake of brevity.
\end{proof}

{\color{black}
\begin{proof}[{\bf Proof of Proposition~\ref{prop-bounded-class-aff}}]
    For ease of exposition, we adopt the notation introduced after problem~\eqref{eq-prop-class-aff}. Our proof proceeds in two steps. We first show that without loss of generality, we can impose bounds on each variable $\bm{q}_{ni}$, $n \in [N]$ and $i \in \mathcal{I}$, and we afterwards show that we can impose non-restrictive bounds on $\lambda$ as well.
    
    To see that each $\bm{q}_{ni}$ can be bounded, $n \in [N]$ and $i \in \mathcal{I}$, we proceed in two steps. We first argue that $\bq_{ni}{}^\top ( \bm{d} - \bm{C} \bx^n ) \geq 0$ for all $n$ and $i$, that is, larger values of $\bq_{ni}$ weakly increase the left-hand sides in the first two constraint sets of~\eqref{eq-prop-class-aff}. Due to the non-negativity of $f_0$ in $\bs$ in problem~\eqref{eq-prop-class-aff}, larger values of $\bq_{ni}$ thus weakly increase the objective function in~\eqref{eq-prop-class-aff}. Non-zero values for $\bq_{ni}$ can therefore only be optimal if they allow to reduce $\lambda$ via the constraints $\lVert a_j y^n \cdot \bbeta_{\mx} - t \cdot \bm{C}{}^\top \bq_{ni}  \rVert_* \leq \lambda$. We then derive a bounded set $\mathcal{Q}$ such that $\lVert a_j y^n \cdot \bbeta_{\mx} - t \cdot \bm{C}{}^\top \bq_{ni} \rVert_* > \lVert a_j y^n \cdot \bbeta_{\mx} \rVert_*$ for all $n \in [N]$, $i = (j, t) \in \mathcal{I}$, all admissible $\bbeta$ and all $\bq_{ni} \notin \mathcal{Q}$, that is, the choice $\bq_{ni} = \bm{0} \in \mathcal{C}^*$ dominates any feasible choice of $\bq_{ni}$ outside of $\mathcal{Q}$. In view of the first step, note that $\bm{d} - \bm{C} \bx^n \in \mathcal{C}$ by construction of $\mX$ and the fact that $\bx^n \in \mX$. We thus have $\bq_{ni}{}^\top ( \bm{d} - \bm{C} \bx^n ) \geq 0$ since $\bq_{ni} \in \mathcal{C}^*$. As for the second step, consider for each $n$ and $i$ the orthogonal decomposition of the vectors $\bq_{ni} \in \mathcal{C}^*$ into $\bq_{ni} = \bq_{ni}^0 + \bq_{ni}^+$ where $\bq_{ni}^0 \in \text{Null} (\bm{C}^\top)$ is in the nullspace of $\bm{C}^\top$ and $\bq_{ni}^+ \in \text{Row} (\bm{C}^\top)$ is in the row space of $\bm{C}^\top$. There is a bounded set $\mathcal{Q}^+ \subseteq \text{Row} (\bm{C}^\top)$ such that $\lVert a_j y^n \cdot \bbeta_{\mx} - t \cdot \bm{C}{}^\top \bq_{ni}^+ \rVert_* > \lVert a_j y^n \cdot \bbeta_{\mx} \rVert_*$ for all $n \in [N]$, $i = (j, t) \in \mathcal{I}$, all admissible $\bm{\beta}$ and all $\bq_{ni}^+ \notin \mathcal{Q}^+$; note in particular that $\lVert a_j y^n \cdot \bbeta_{\mx} \rVert_*$ is bounded due to the assumed boundedness of the hypothesis set and the fact that $\mathcal{I}$ and $\mZ$ are finite sets. Similarly, there is a bounded set $\mathcal{Q}^0 \subseteq \text{Null} (\bm{C}^\top)$ such that for all $\bq_{ni}' \in \text{Null} (\bm{C}^\top) \setminus \mathcal{Q}^0$ satisfying $\bq_{ni}^+ + \bq_{ni}' \in \mathcal{C}^*$ for some $\bq_{ni}^+ \in \mathcal{Q}^+$ there is $\bm{q}_{ni}^0 \in \mathcal{Q}^0$ satisfying $\bq_{ni}^+ + \bq_{ni}^0 \in \mathcal{C}^*$ such that $(\bq_{ni}^+ + \bq_{ni}^0)^\top (\bm{d} - \bm{C} \bm{x}^n) \leq (\bq_{ni}^+ + \bq_{ni}')^\top (\bm{d} - \bm{C} \bm{x}^n)$, that is, $\bq_{ni}^0{}^\top \bm{d} \leq \bq_{ni}'{}^\top \bm{d}$, across all $n \in [N]$ and $i = (j, t) \in \mathcal{I}$. Thus, we can without loss of generality restrict the choice of $\bm{q}_{ni}$ to the bounded set $\mathcal{Q} = \mathcal{C}^* \cap (\mathcal{Q}^0 + \mathcal{Q}^+)$, where the sum is taken in the Minkowski sense.
    
    To see that $\lambda$ can be bounded as well, note that by construction, $\lambda$ is bounded from below by $0$. To see that we can bound $\lambda$ from above as well, we observe that there are optimal solutions to problem~\eqref{eq-prop-class-aff} for which $\lambda$ does not exceed $\overline{\lambda} = \max \{ \overline{\lambda}_1, \, \overline{\lambda}_2 \}$, where
    \begin{equation*}
        \overline{\lambda}_1 \; = \; \sup_{\bbeta \in \mathcal{H}} \; \max_{n \in [N]} \; \max_{i = (j, t) \in \mathcal{I}} \;
        \max_{\bq_{ni} \in \mathcal{Q}} \; \lVert a_j y^n \cdot \bbeta_{\mx} - t \cdot \bm{C}{}^\top \bq_{ni}  \rVert_*
    \end{equation*}
    with the bounded set $\mathcal{H} \subseteq \mR^{1 + M_{\bx} + M_{\bz}}$ containing all admissible hypotheses $\bbeta$, and
    \begin{equation*}
        \begin{array}{r@{}l}
            \displaystyle
            \overline{\lambda}_2 \; = \;
            \sup_{\bbeta \in \mathcal{H}} \; \max_{n \in [N]} \; \max_{i = (j, t) \in \mathcal{I}} \; \max_{\bz \in \mC} \; \max_{\bq_{ni} \in \mathcal{Q}} \Bigg\{
            \frac{\displaystyle \bq_{ni}{}^\top ( \bm{d} - \bm{C} \bx^n ) +  a_j t y^n \cdot ( \beta_0 + \bbeta_{\mx}{}^\top \bx^n + \bbeta_{\mz}{}^\top \bz ) + b_j}{\kappa_\mz d_\mz(\bz,\bz^n) + \kappa_\my \cdot \indicate{t = -1}} & \\
            \displaystyle
            : \, (\bz, t) \neq (\bz^n, 1) & \Bigg\}.
        \end{array}
    \end{equation*}
    Indeed, selecting $\lambda \geq \overline{\lambda}_1$ ensures that all hypotheses $\bbeta \in \mathcal{H}$ are feasible, and selecting $\lambda \geq \overline{\lambda}_2$ implies that for any $\bbeta \in \mathcal{H}$, all left-hand sides in the first two constraint sets of problem \eqref{eq-prop-class-aff} are non-positive, and thus all of these constraints are weakly dominated by the non-negativity constraints on $\bs$. Note that the numerator in the objective function of $\overline{\lambda}_2$ is bounded since all suprema and maxima operate over bounded sets. Moreover, the denominator in the objective function of $\overline{\lambda}_2$ is bounded from below by a strictly positive quantity since $(\bz, t) \neq (\bz^n, 1)$ and $\kappa_{\mz}, \kappa_{\my} > 0$. We thus conclude that variable $\lambda$ in problem~\eqref{eq-prop-class-aff} can be bounded as well.
\end{proof}}

\begin{proof}[{\bf Proof of Corollary~\ref{coro-class-aff}}]
    The proof is similar to that of Corollary~15 by \cite{JMLR:v20:17-633}. Details are omitted for the sake of brevity.
\end{proof}

The proof of Proposition~\ref{prop-der-reg-conv} relies on two lemmas that we will state and prove first.

\begin{lem}\label{lem-dual-norm}
    The compound norm $ \lVert [\bm{\alpha},\nu] \rVert_{\emph{comp}} = \lVert \bm{\alpha} \rVert + \kappa \lvert \nu \rvert$, $\bm{\alpha} \in \mR^{{M_\mx}}$, $\nu \in \mR$ and $\kappa > 0$, satisfies
    \begin{equation*}\label{eq-lem-dual-norm}
        \lVert [\bm{\alpha},\nu] \rVert_{\emph{comp}^*} = \mathrm{max} \left\{\lVert \bm{\alpha} \rVert_{*}, \; \dfrac{ \lvert \nu \rvert }{\kappa}\right\}.
    \end{equation*}
\end{lem}

\begin{proof}[{\bf Proof of Lemma~\ref{lem-dual-norm}}]
    By definition of the dual norm, we have that
    \begin{equation*}\label{proof-lem-dual-norm}
        \lVert [\bm{\alpha},\nu] \rVert_{\mathrm{comp}^*}  = \left\{ \begin{array}{l@{\quad}l}
            \displaystyle \mathop{\mathrm{maximize}}_{(\bx,y) \in \mR^{M_\mx} \times \mR} & \displaystyle \bm{\alpha}{}^\top \bx + \nu y  \\[4mm]
            \displaystyle \mathrm{subject\;to} & \displaystyle \lVert \bx \rVert + \kappa \lvert y \rvert \leq 1.
        \end{array} \right.
    \end{equation*}
    Note that the optimization problem on the right-hand side satisfies Slater's condition since the feasible region includes the interior point $(\bx,y) = (\bm{0},0)$. The optimization problem thus has a strong dual. To derive the dual problem, we consider the Lagrange dual function for $\gamma \geq 0$,
    \begin{equation*}\label{proof-lem-dual-norm-2}
        g(\gamma) =  \displaystyle \underset{(\bx,y) \in \mR^{M_\mx} \times \mR}{\sup}\, \bm{\alpha}{}^\top \bx + \nu y - \gamma (\lVert \bx \rVert + \kappa \lvert y \rvert - 1).
    \end{equation*}
    Note that the maximization is separable over $\bx$ and $y$. Focusing on the variable $y$, in order for the Lagrange dual function to attain a finite value, we need to have $\lvert \nu \rvert \leq \gamma \kappa$ so that $y^\star = 0$. Under this condition, the Lagrange dual function simplifies to
    \begin{equation*}\label{proof-lem-dual-norm-3}
        g(\gamma) =  \displaystyle \underset{\bx \in \mR^{M_\mx} }{\sup}\, \bm{\alpha}{}^\top \bx - \gamma \lVert \bx \rVert + \gamma.
    \end{equation*}
    We can now apply Lemma~\ref{lem-conv-class} with $L$ being the identity to obtain the equivalent reformulation
    \begin{equation*}
        g(\gamma) =  \displaystyle \begin{cases}
		  \gamma & \text{if} \  \lVert \bm{\alpha} \rVert_{*} \leq \gamma \ , \ \lvert \nu \rvert \leq \gamma \kappa, \\
		  + \infty & \mbox{otherwise}.
	   \end{cases}
    \end{equation*}
    The dual problem therefore is
    \begin{equation*}
        \begin{array}{l@{\quad}l}
            \displaystyle \mathop{\mathrm{minimize}}_{\gamma} & \displaystyle \gamma  \\[5mm]
            \displaystyle \mathrm{subject\;to} & \displaystyle \lvert \nu \rvert \leq \gamma \kappa\\
            & \lVert \bm{\alpha} \rVert_{*} \leq \gamma\\
            & \gamma \in  \mR{\color{black}_+},
        \end{array}
    \end{equation*}
    which has the optimal solution $\lVert [\bm{\alpha},\nu] \rVert_{\mathrm{comp}^*} = \gamma^\star = \max \{\lVert \bm{\alpha} \rVert_{*}, \, \lvert \nu \rvert / \kappa \}$.
\end{proof}

\begin{lem}\label{lem-conv-reg-alt}
    Assume that the loss function $L$ is convex and Lipschitz continuous. For fixed $\bm{\alpha},\bx^n \in \mR^{M_\mx}$, $\alpha_0, y^n \in \mR$, $\kappa >0$ and $\lambda \in \mR_+$, we have
    \begin{align*}\label{eq-lem-conv-reg-alt}
        &\underset{(\bx,y) \in \mR^{M_\mx} \times \mR}{\sup}\, L(\bm{\alpha}{}^\top \bx - y + \alpha_0) - \lambda \lVert \bx - \bx^n \rVert - \lambda \kappa \lvert y - y^n \rvert \\
        & \mspace{150mu} = 
        \begin{cases}
		    L(  \bm{\alpha}{}^\top \bx^n - y^n + \alpha_0 ) & \text{if } \mathrm{lip}(L) \cdot \mathrm{max} \left\{ \kappa \lVert \bm{\alpha} \rVert_{*} , 1 \right\} \leq \lambda \kappa ,\\
	        + \infty & \mbox{otherwise}.
	    \end{cases}
    \end{align*}
\end{lem}

\begin{proof}[{\bf Proof of Lemma~\ref{lem-conv-reg-alt}}]
    Concatenating the variables $\bx$ and $y$ to $\bw = [\bx^\top,y]^\top \in \mR^{{M_\mx}+1}$, letting $\bm{\eta} = [\bm{\alpha}^\top,-1]^\top$ and $\bw^n = [(\bx^n)^\top,y^n]^\top$ and defining the compound norm $ \lVert [\bx,y] \rVert_{\mathrm{comp}} = \lVert \bx \rVert + \kappa \lvert y \rvert$, we can write the left-hand side of the equation in the statement of the lemma as
    \begin{equation*}
        \underset{\bw \in \mR^{{M_\mx}+1}}{\sup}\, L(\bm{\eta}{}^\top \bw + \alpha_0) - \lambda \lVert \bw - \bw^n \rVert_{\mathrm{comp}}.
    \end{equation*}
    Now we can apply Lemma~\ref{lem-conv-class} directly to conclude that
    \begin{equation*}
        \underset{\bw \in \mR^{{M_\mx}+1}}{\sup}\, L(\bm{\eta}{}^\top \bw + \alpha_0) - \lambda \lVert \bw- \bw^n \rVert_{\mathrm{comp}} = 
        \begin{cases}
		    L(  \bm{\eta}{}^\top \bw^n + \alpha_0 ) & \text{if} \ \mathrm{lip}(L) \cdot \lVert \bm{\eta} \rVert_{\mathrm{comp}^*} \leq \lambda,\\
		    + \infty & \mbox{otherwise}.
	    \end{cases}
    \end{equation*}
    The statement now follows from Lemma~\ref{lem-dual-norm}, which implies that $\mathrm{lip}(L) \cdot \lVert \bm{\eta} \rVert_{\mathrm{comp}^*} \leq \lambda$ if and only if $\mathrm{lip}(L) \cdot \mathrm{{max}} \{\lVert \bm{\alpha} \rVert_{*}, \, \lvert -1 \rvert / \kappa \} \leq \lambda$.
\end{proof}

\begin{proof}[{\bf Proof of Proposition~\ref{prop-der-reg-conv}}]
    The statement can be proven along the lines of the proof of Theorem~4~(ii) by \cite{JMLR:v20:17-633} if we leverage Lemma~\ref{lem-conv-reg-alt} to re-express the embedded maximization over $(\bx,y) \in \mR^{M_\mx} \times \mR$. Details are omitted for the sake of brevity.
\end{proof}

{\color{black}
\begin{proof}[{\bf Proof of Corollary~\ref{coro-bounded-reg-conv}}]
    The proof is similar to that of Proposition~\ref{prop-bounded-class-conv}. Details are omitted for the sake of brevity.
\end{proof}}

\begin{proof}[{\bf Proof of Corollary~\ref{coro-reg-con}}]
    The proof is similar to that of Corollary~5 by \cite{JMLR:v20:17-633}. Details are omitted for the sake of brevity.
\end{proof}

\begin{proof}[{\bf Proof of Proposition~\ref{prop-der-reg-aff}}]
    The statement can be proven along the lines of the proof of Theorem~4~(i) by \cite{JMLR:v20:17-633}. We omit the details for the sake of brevity.
\end{proof}

{\color{black}
\begin{proof}[{\bf Proof of Corollary~\ref{coro-bounded-reg-aff}}]
    The proof is similar to that of Proposition~\ref{prop-bounded-class-aff}. Details are omitted for the sake of brevity.
\end{proof}}

\begin{proof}[{\bf Proof of Corollary~\ref{coro-reg-aff}}]
    The proof is similar to those of Corollaries~6~and~7 by \cite{JMLR:v20:17-633}. Details are omitted for the sake of brevity.
\end{proof}

\end{document}